\documentclass[a4paper,12pt,reqno,oneside]{amsart} 

\usepackage{setspace} 
\usepackage{typearea} 
\usepackage{amscd,amsmath,amssymb,verbatim} 

\usepackage[dvipsnames]{xcolor}
\usepackage{xr-hyper}
\usepackage{cite}
\usepackage[colorlinks,backref,final,bookmarksnumbered,bookmarks]{hyperref}
\usepackage[backrefs]{amsrefs}
\hypersetup{linkcolor=bluex,citecolor=Green,filecolor=cyan,urlcolor=magenta}
\usepackage{graphicx}
\definecolor{bluex}{Hsb}{220, 1, 1}

\newcommand{\borderref}[2]{{\hypersetup{linkcolor=black}\hyperref[#1]{#1 #2}}}

\usepackage{xfrac}

\usepackage{ifthen}
\newcommand{\arxiv}[2][]{\ifthenelse{\equal{#1}{}}
{\href{http://arxiv.org/abs/#2}{\tt arXiv:#2}}
{\href{http://arxiv.org/abs/math/#2}{\tt arXiv:math.#1/#2}}}

\theoremstyle{plain}
\newtheorem{theorem}{Theorem}[section]
\newtheorem*{theorem*}{Theorem}
\newtheorem{mainthm}{Theorem}

\newtheorem{innercustomthm}{Theorem}
\newenvironment{customthm}[1]
  {\renewcommand\theinnercustomthm{#1}\innercustomthm}
  {\endinnercustomthm}
\newtheorem{lemma}[theorem]{Lemma}
\newtheorem{proposition}[theorem]{Proposition}
\newtheorem{corollary}[theorem]{Corollary}
\newtheorem*{corollary*}{Corollary}
\newtheorem{problem}[theorem]{Problem}
\newtheorem*{problem*}{Problem}
\newtheorem*{addendum*}{Addendum}

\theoremstyle{definition}
\newtheorem{remark}[theorem]{Remark}
\newtheorem*{remark*}{Remark}
\newtheorem{example}[theorem]{Example}
\newtheorem*{examples*}{Examples}

\newtheoremstyle{named}{}{}{\itshape}{}{\bfseries}{.}{.5em}{\thmnote{#3}}
\theoremstyle{named}
\newtheorem{namedthm}{}

\def\x{\times}
\def\but{\setminus}

\def\eps{\varepsilon}
\def\phi{\varphi}
\def\emptyset{\varnothing}
\def\CAP{\!\smallfrown\!}
\renewcommand{\:}{\colon}
\def\C{\mathbb{C}}
\def\N{\mathbb{N}}
\def\R{\mathbb{R}}
\def\Q{\mathbb{Q}}
\def\Z{\mathbb{Z}}
\def\B{\mathcal{B}}
\def\F{\mathcal{F}}
\def\K{\mathcal{K}}
\def\L{\mathcal{L}}

\def\X{\mathcal{X}}

\def\xr#1{\xrightarrow{#1}} 

\def\onehalf{1\hspace{-1.1pt}/\hspace{-0.5pt}2}
\makeatletter
\newcommand{\xR}[2][]{\ext@arrow 0359\Rightarrowfill@{#1}{#2}}
\newcommand{\xL}[2][]{\ext@arrow 0359\Leftarrowfill@{#1}{#2}}
\makeatother
\DeclareMathOperator{\lk}{lk}
\DeclareMathOperator{\rk}{rk}

\DeclareMathOperator{\id}{id}
\DeclareMathOperator{\colim}{colim}
\DeclareMathOperator{\Hom}{Hom}

\newcommand{\newsym}[5]{\fontfamily{#2}\fontencoding{#1}\fontseries{#3}\fontshape{#4}\selectfont\char#5}
\makeatletter
\newcommand{\newmathsymbol}[6]{#1{\@Pimathsymbol{#2}{#3}{#4}{#5}{#6}}}
\def\@Pimathsymbol#1#2#3#4#5{\mathchoice
  {\@Pim@thsymbol{#1}{#2}{#3}{#4}{#5}\tf@size}
  {\@Pim@thsymbol{#1}{#2}{#3}{#4}{#5}\tf@size}
  {\@Pim@thsymbol{#1}{#2}{#3}{#4}{#5}\sf@size}
  {\@Pim@thsymbol{#1}{#2}{#3}{#4}{#5}\ssf@size}}
\def\@Pim@thsymbol#1#2#3#4#5#6{\mbox{\fontsize{#6}{#6}\newsym{#1}{#2}{#3}{#4}{#5}}}
\makeatother

\DeclareFontEncoding{LS1}{}{}
\DeclareFontSubstitution{LS1}{stix}{m}{n}
\def\varnabla{\newmathsymbol{\mathord}{LS1}{stix}{m}{n}{"28}}

\begin{document}
\title{Is every knot isotopic to the unknot?} 
\author{Sergey A. Melikhov}
\address{Steklov Mathematical Institute of Russian Academy of Sciences, Moscow, Russia}
\email{melikhov@mi-ras.ru}

\begin{abstract}
In 1974, D. Rolfsen asked: Is every knot in $S^3$ isotopic (=homotopic through embeddings) to a PL knot or, 
equivalently, to the unknot? In particular, is the Bing sling isotopic to a PL knot?
We show that the Bing sling is not isotopic to any PL knot:
\begin{enumerate}
\item by an isotopy which extends to an isotopy of $2$-component links with $\lk=1$;
\item through knots that are intersections of nested sequences of solid tori.
\end{enumerate}

There are also stronger versions of these results.
In (1), the additional component may be allowed to self-intersect, and even to get replaced by a new one
as long as it represents the same conjugacy class in $G/[G',G'']$, where $G$ is the fundamental group of 
the complement to the original component.
In (2), the ``solid tori'' can be replaced by ``boundary-link-like handlebodies'', where a handlebody 
$V\subset S^3$ of genus $g$ is called boundary-link-like if $\pi_1(\overline{S^3\but V})$ admits 
a homomorphism to the free group $F_g$ such that the composition 
$\pi_1(\partial V)\to\pi_1(\overline{S^3\but V})\to F_g$ is surjective.
\end{abstract}

\maketitle

\section{Introduction} \label{intro}
Fifty years ago%
\footnote{In his AMS Math Rewiew of \cite{Brin}, published in 1985, R. Daverman speaks of this problem as
``being around for almost 20 years''.
Although it is clear from \cite{Mi2} (1957), \cite{Fox2}*{p.~171} (1962), \cite{Gi} (1964), \cite{HS} (1966)
that Rolfsen's problem was probably known to a number of people long before it was stated explicitly in 
Rolfsen's 1974 paper, I'm not aware of its earlier explicit appearances in the literature.}
D. Rolfsen \cite{Ro1} posed the following problem.

\begin{namedthm}[Rolfsen's Problem] \label{Rolfsen's} ``All PL knots are isotopic to one another.%
\footnote{See for instance Figure 1.1 in the Burde--Zieschang textbook \cite{BZ}.}
Is this true of wild knots? A knot in the $3$-space which is not isotopic to a tame knot would have to be 
so wild as to fail to pierce a disk at each of its points.%
\footnote{This assertion of Rolfsen is easy to prove in the case where the disk is tame
(see Lemma \ref{Rolfsen-assertion}). 
A proof of the general case was found recently by R. Ancel (see \cite{An} for a sketch).}
The `Bing sling' {\rm [\cite{Bi}]} is such a candidate.''
\end{namedthm}

\begin{figure}[h]
\includegraphics[width=0.75\linewidth]{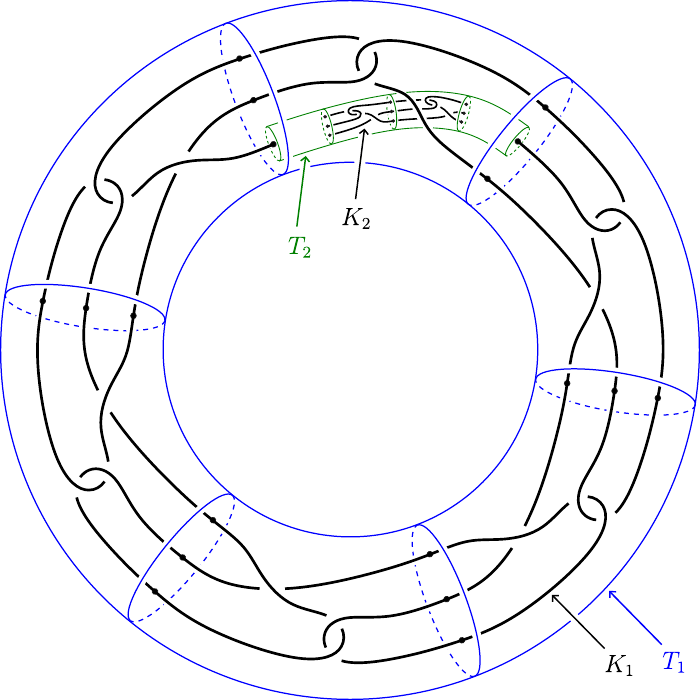}
\caption{The Bing sling.}
\label{bingsling}
\end{figure}

The ``Bing sling'' is R. H. Bing's example of a wild knot that ``pierces no disk'' \cite{Bi}.
It is the intersection of a chain of solid tori $T_1\supset T_2\supset\dots$, where $T_1$ is 
unknotted in $S^3$ and each $T_{i+1}$ is a suitable regular neighborhood of a certain PL knot 
$K_i$ in $T_i$ (see Figure ~\ref{bingsling}).
We discuss this construction and its outcome in more detail in \S\ref{whatisbingsling}.

Until recently, \borderref{Rolfsen's}{Problem} has been mostly approached with geometric methods.
D. Gillman constructed a knot that pierces no disk but lies on a disk (and hence is isotopic to the unknot);
but he also proved that no subarc of the Bing sling lies on a disk \cite{Gi}.
M. Brin constructed a knot that at each point is locally equivalent to the Bing sling, but is isotopic to the unknot \cite{Brin}.
The most striking result was obtained by C. Giffen:

\begin{namedthm}[Giffen's Theorem]\label{Giffen's} {\rm (see Appendix \ref{app})} Every knot in the solid torus $T$ which 
represents a generator of $H_1(T)$ is $I$-equivalent to the core circle of $T$.
\end{namedthm}

Two $m$-component (possibly wild) links in a $3$-manifold $M$ are called {\it $I$-equivalent} if they cobound
a topological embedding of $m$ disjoint annuli in $M\x I$.
(The embedding is not assumed to be locally flat, so this is not the same as topological concordance.)

An easy, though not entirely obvious (cf.\ \cite{Ro3}*{Question 1}) consequence of \borderref{Giffen's}{Theorem} is

\begin{corollary*} {\rm (see Appendix \ref{app})} Every knot in $S^3$ is $I$-equivalent to the unknot.
\end{corollary*}

The following was shown in a recent paper by the author.

\begin{theorem}[\cite{M21}] \label{lk0} Not every link in $S^3$ is $I$-equivalent to a PL link.
\end{theorem}

\begin{corollary*} Not every link is $S^3$ is isotopic to a PL link.
\end{corollary*}

The proof of Theorem \ref{lk0} in \cite{M21} can be called geometric. 
It is based on Cochran's derived invariants $\beta_i$, which are extended to $I$-equivalence invariants of topological links by using 
infinite {\it homological} Seifert surfaces.
By using these, it is shown in \cite{M21} that there exist uncountably many 2-component links of linking number $0$, with unknotted
first component, which are not $I$-equivalent to each other.

By contrast, the proofs of the main results of the present paper can be called algebraic.
They are based on the following result (and its version for colored links):

\begin{theorem}[\cite{M24-1}*{Corollary \ref{part1:main1''}}] \label{finite-type} Let $v$ be a finite type invariant of PL links which 
is invariant under PL isotopy.
Then 

(a) $v$ assumes the same value on all sufficiently close PL $C^0$-approximations of any given topological link;

(b) the extension of $v$ by continuity to topological links is an invariant of isotopy.
\end{theorem}

The proof of Theorem \ref{finite-type} in \cite{M24-1} is in turn based on the notion of $n$-quasi-isotopy, which was introduced by the author \cite{MR1} 
but can be traced back to a number of older constructions (from the Penrose--Whitehead--Zeeman--Irwin trick to Casson handles), most notably 
to the Homma--Bryant proof of the Chernavsky--Miller codimension three PL approximation theorem (see \cite{MR1}*{Remarks (i)--(iii) in \S1.5}).

\begin{mainthm} \label{main1} Let $\B$ be the Bing sling, and let $Q$ be a PL knot in $S^3\but\B$ with $\lk(Q,\B)=1$.
Then $(Q,\B)$ is not isotopic to any PL link.
\end{mainthm}

While this is a special case of the results to follow, we single it out since its proof is considerably easier.
Apart from Theorem \ref{finite-type} it uses only the Conway polynomial $\nabla_L$ (and some of its properties).
Or rather its reduced version $\nabla_{(K_1,K_2)}(z)/\nabla_{K_1}(z)\nabla_{K_2}(z)$, which is easily seen to be
an invariant of PL isotopy.

A minor elaboration on the proof of Theorem \ref{main1} yields

\begin{mainthm} \label{main2} The Bing sling is not isotopic to any PL knot through knots that are intersections of nested sequences of solid tori.
\end{mainthm}

Not every wild knot is the intersection of a nested sequence of solid tori.
There are at least four proofs of this fact in the literature.
\begin{enumerate}
\item J. Milnor \cite{Mi}*{Theorem 10 and assertions (1) and (2) on p.\ 303} constructed a $2$-component wild link 
which is isotopic to the unlink, but not through links satisfying property ($S$), where a $2$-component link $(Q,\K)$ is said to have 
{\it property ($S$)} if $\K$ lies in a solid torus $T$ disjoint from $Q$ and such that $\K$ represents a generator of $H_1(T)$.
His proof is based on the fundamental group.
But it is easy to see that $(Q,\K)$ satisfies property ($S$) if $\K$ is the intersection of a nested 
sequence of solid tori (see Lemma \ref{colimit}).
\item J. McPherson \cite{Mc} proved that Fox's ``remarkable simple closed curve'' is not the intersection of any nested 
sequence of solid tori.
The proof is based on the observation (found already in \cite{Br}) that the Alexander module of a wild knot can have rank $>1$.
\item S. Kojima and M. Yamasaki \cite{KY}*{assertions (1), (3) of Theorem 4 and the first line of \S7} proved that a certain 
$2$-component wild link whose first component is PL (a variation of Milnor's wild link) is isotopic to the unlink but not F-isotopic 
to any PL link.
In particular, it does not satisfy Milnor's property ($S$).
The proof is based on the observation that the knot module 
of a wild knot need not be a torsion module (which is equivalent to the observation mentioned in the previous item, 
see \cite{Kaw}*{Proposition 7.3.4(1)}).
\item Since the first component of the link in Theorem \ref{lk0} is PL, it follows from \borderref{Giffen's}{Theorem} and
Theorem \ref{lk0} that it does not satisfy Milnor's property ($S$).
\end{enumerate}

\begin{customthm}{A$'$} \label{main1'} The Bing sling is not isotopic to any PL knot by any isotopy which extends to 
an isotopy of $2$-component links with linking number $1$.
\end{customthm}

This is like Theorem \ref{main1}, except that the knot $Q$ and the terminal knot of the isotopy which $Q$ undergoes 
are now allowed to be wild.
This may seem like a minor improvement, but here is one application.
It is well-known that the Whitehead link is symmetric (i.e.\ its components may be interchanged by an ambient isotopy);
it is easy to check that its counterparts among links of linking number 1, the Mazur link and the fake Mazur link 
(see Figure \ref{mazur} below and comments in footnote \ref{zeeman} to Figure \ref{mazur-curve} below) are also symmetric.
(For a symmetric diagram of the Mazur link see \cite{Akb}*{Figure 4}.)
Even though the Bing sling is basically made up of copies of the fake Mazur link (see \S\ref{whatisbingsling}),
one consequence of Theorem \ref{main1'} is

\begin{corollary*} Let $\B$ be the Bing sling, and let $Q$ be a PL knot in $S^3\but\B$ with $\lk(Q,\B)=1$.
Then $(Q,\B)$ is not isotopic to $(\B,Q)$.%
\footnote{Such an isotopy, if it existed, would provide an elegant solution to \borderref{Rolfsen's}{Problem}.}
\end{corollary*}

But the real difference between Theorems \ref{main1} and \ref{main1'} is that a minor extension of the proof of 
the latter yields

\begin{mainthm} \label{main3} The Bing sling is not isotopic to any PL knot by any isotopy which extends to 
a link homotopy of $2$-component links with linking number $1$.
\end{mainthm}

We recall that a {\it link homotopy} is a homotopy of a link in which the link components remain disjoint.

The proofs of Theorems \ref{main1'} and \ref{main3} are based on a version of Theorem \ref{finite-type} for {\it colored} 
finite type invariants and on the $2$-variable Conway polynomial $\varnabla_L(u,v)$ introduced in \cite{M24-2}.
More precisely we make use of a $1$-parameter subfamily $\beta_1,\beta_2,\dots$ of the coefficients of its 
reduced version $\varnabla_{(K_1,K_2)}(u,v)/\nabla_{K_1}(u)\nabla_{K_2}(v)$, which in the case $\lk=0$ 
coincides with the sequence of Cochran's derived invariants.
However, in our case $\lk=1$ these integers $\beta_i$ exhibit a much subtler geometric behavior 
than the original Cochran's invariants, as we will see shortly.

Although the proofs of Theorems \ref{main1}, \ref{main2}, \ref{main3} were written up only in 2023, these results 
were announced in my 2005 talk at the conference ``Manifolds and their Mappings'' in Siegen, Germany.
A question which I asked in the same talk was recently answered by A. Zastrow \cite{Zast}, \cite{Zast2}:

\begin{namedthm}[Zastrow's Example]\label{Zastrow's} There exists an isotopy from the unknot to itself
which does not extend to any link homotopy of $2$-component links with linking number ~$1$.
\end{namedthm}

We include a short self-contained review of Zastrow's Example (see Example \ref{zastrow}), as it will
be used in the proof of Theorem \ref{main5}.

To counter the difficulty highlighted by Zastrow's Example, we consider relations weaker than link homotopy.

The finite derived series of a group $G$ is usually defined by $G^{(0)}=G$ and 
$G^{(n)}=[G^{(n-1)},G^{(n-1)}]$ for positive integer $n$.
But sometimes it is expanded to include half-integer terms, defined by $G^{(n+0.5)}=[G^{(n)},G^{(n-1)}]$
for positive integer $n$.
Speaking geometrically, it is well-known and easy to see 
that a loop $\varphi\:S^1\to X$ represents 
an element of $\pi_1(X)^{(n)}$ if and only if $\varphi$ bounds a map $\Gamma\to X$ of a grope of height%
\footnote{Not to be confused with gropes of {\it class} $n$, which are similarly related to 
the lower central series of $G$.}
$n$ (see \cite{COT}*{Definition 7.9 and Lemma 7.10}, where gropes of half-integer height are included).

We call $2$-component links {\it $n$-bordant} if they are equivalent under the equivalence relation generated 
by the following two relations:
\begin{enumerate}
\item link homotopy which restricts to an isotopy on the second component;
\item $(Q,\K)\sim(Q',\K)$ if $Q$ and $Q'$ represent the same conjugacy class in $G/G^{(n)}$, where 
$G=\pi_1(S^3\but\K)$.
\end{enumerate}
Obviously, links $(Q_0,\K_0)$ and $(Q_1,\K_1)$ are $n$-bordant if and only if they are related by 
a sequence of link homotopies which extend consecutive time intervals of a single isotopy $\K_t$ 
between $\K_0$ and $\K_1$, with ``jumps'' between consecutive link homotopies related by 
the relation (2).
Such a sequence of link homotopies will be called an {\it $n$-bordism} between $(Q_0,\K_0)$ and $(Q_1,\K_1)$.

By using $2$-bordism in place of link homotopy we can avoid \borderref{Zastrow's}{Example}:

\begin{customthm}{E$'$}\label{main5'} If $L$ and $L'$ are links of linking number $1$, then every isotopy between their 
second components extends to a $2$-bordism between $L$ and $L'$.
\end{customthm}

As discussed in \S\ref{d+}, from the geometric viewpoint this result was initially very surprising and hard to believe.
Eventually, there are 3 independent proofs, none of them particularly hard (one based on the commutator identities, another 
on the knot module, and a geometric one, based on Seifert surfaces).

On the other hand, Corollary to Theorem \ref{lk0} holds with 2-bordism in place of isotopy:

\begin{customthm}{D$'$}\label{main4'} 
There exists a $2$-component link in $S^3$ with linking number $0$ which is not $2$-bordant to any link
whose second component is PL.
\end{customthm}

Like the proof of Theorem \ref{lk0}, the proof of Theorem \ref{main4'} is based on Cochran's 
derived invariants $\beta_i$. 
So one could hope to obtain a similar result for links of linking number $1$ by using their $\beta_i$.
One issue with this approach (which was one reason why it took me 18 years to find time for writing up  
the results announced in 2005) was that while the geometry of Cochran's original invariants $\beta_i$ was 
well-understood since the 1980s, virtually nothing was known about the geometry of the $\beta_i$ for links 
of linking number $1$.
But finally some basic understanding has been gained \cite{M24-2}*{Theorem \ref{part2:mainmain}}), and as its application 
we obtain the following result.

\begin{mainthm} \label{main4} The Bing sling is not isotopic to a PL knot by any isotopy which
extends to a $2.5$-bordism of $2$-component links with linking number ~$1$.
\end{mainthm}

A more precise version of Theorem \ref{main4} says that the power series $C_L(z)=\sum_{i=1}^\infty\beta_iz^i$
is rational if the second component of $L$ is PL and not rational if it is the Bing sling, 
but $C_L-C_{L'}$ is rational if $L$ and $L'$ are 2.5-bordant links of linking number $1$.
But in fact, $C_L-C_{L'}$ is also rational if $L$ and $L'$ are 2-bordant links of linking number $0$ --- and this is 
where Theorem \ref{main4'} comes from.
However, Theorem \ref{main5'} and Theorem \ref{non-invariance} imply that $C_L-C_{L'}$ need not be rational for 
2-bordant links $L$ and $L'$ of linking number $1$.%
\footnote{This observation refutes a premature announcement from my 2005 talk in Siegen.
At that time I wrongly believed that Cochran's theorem on additivity of the $\beta_i$ under band connected sum 
\cite{Co1}*{Theorem 5.4} extends to links of any linking number.}
A further blow is provided by the following example.

\begin{mainthm}\label{main5} There exists an isotopy from the unknot to itself which does not extend to 
any $2.5$-bordism of $2$-component links with linking number ~$1$.
\end{mainthm}

\begin{remark}\label{torsion-remark}
The proof of Theorem \ref{main4} shows that it can be improved by replacing $G/[G',G'']$, which occurs in the definition
of $2.5$-bordism, with $G/[G',T]$, where $T$ is preimage in $G'$ of the torsion submodule of the link module $G'/G''$.
Unfortunately, Theorem \ref{main5} can also be improved in the same way (see Proposition \ref{torsion-free}).
\end{remark}

Let us now turn to our final result, which applies the techniques of Theorem \ref{main4} 
to the setting where we refrain from inserting a new component, like in Theorem \ref{main2}.
As noted above, a wild knot need not be the intersection of a nested sequence of solid tori.
But it is well-known that not only every wild knot, but every $1$-dimensional compact ANR subset $X$ of $S^3$ is the intersection of 
a nested sequence of handlebodies \cite{McM}*{Lemmas 2, 3} (see also \cite{Kir}*{Theorem}).%
\footnote{Indeed, $X$ is certainly the intersection of a nested sequence of closed polyhedral neighborhoods $M_1\supset M_2\supset\dots$.
Each $M_i$ has a handle decomposition with $0$-, $1$- and $2$-handles.
To eliminate the $2$-handles it suffices to ``puncture'' them, i.e.\ to make their cocore 1-disks disjoint from $\K$.
This can always be done since $X$ is tame in the sense of Shtan'ko (or in other words has embedding dimension 1, see \cite{Sh1})
by \cite{McM}*{Lemma 3} (see also \cite{Bo}*{Satz 4}, \cite{Sh2}*{Theorem 4}).
In fact, a $1$-dimensional compact subset $X$ of $S^3$ is the intersection of a nested sequence of {\it thin} handlebodies (i.e.\ such that 
for each $\epsilon>0$ almost all of the handlebodies admit $\epsilon$-retractions onto graphs) if and only if $X$ is tame in the sense of 
Shtan'ko (see \cite{Bo2}*{p.\ 176}).}
Before we discuss these nested sequences, let us recall one classical fragment of the theory of knotted handlebodies.

\begin{theorem} \label{retractable} 
Let $V$ be a genus $g$ PL handlebody in $S^3$ and let $W=\overline{S^3\but V}$. 
The following are equivalent:
\begin{enumerate}
\item $W$ admits a retraction onto a PL embedded wedge of $g$ circles;
\item $\pi_1(W)$ admits an epimorphism onto the free group $F_g$;
\item $\pi_1(W)/\gamma_\omega\pi_1(W)$ is a free group;%
\footnote{$\gamma_\omega G$ denotes the intersection of the finite lower central series $\gamma_n G$
of the group $G$, where $\gamma_1G=G$ and $\gamma_{n+1}G=[\gamma_nG,G]$.}
\item $W$ contains pairwise disjoint properly%
\footnote{An embedding $g\:P\to Q$ between manifolds with boundary is called {\it proper} if 
$g^{-1}(\partial Q)=\partial P$.}
PL embedded surfaces $\Sigma_1,\dots,\Sigma_g$ whose union is non-separating in $W$.
\end{enumerate}
\end{theorem}

\begin{proof} This is similar to Smythe's characterizations of homology boundary links \cite{Sm}.
Formally, (1)$\Leftrightarrow$(2) is a special case of \cite{JM}*{Corollary to Theorem 2}
and (2)$\Leftrightarrow$(4) is a special case of \cite{J2}*{Theorem 2.1}.
Finally, (2)$\Rightarrow$(3) follows from \cite{JM}*{Corollary to Theorem 1} and
the converse is obvious (since the rank of a free group is determined by its abelianization).
\end{proof}

\begin{theorem} \label{boundary-retractable}
Let $V$ be a genus $g$ PL handlebody in $S^3$ and let $W=\overline{S^3\but V}$. 
The following are equivalent:
\begin{enumerate}
\item $W$ admits a retraction onto a PL embedded wedge of $g$ circles contained in $\partial W$;
\item $\pi_1(W)$ admits a homomorphism to the free group $F_g$ such that the composition
$\pi_1(\partial W)\to\pi_1(W)\to F_g$ is surjective;
\item $\pi_1(W)/\gamma_\omega\pi_1(W)$ is a free group and $\pi_1(\partial W)$ maps surjectively onto it;
\item $W$ contains pairwise disjoint properly PL embedded surfaces $\Sigma_1,\dots,\Sigma_g$ such that 
$\partial\Sigma_1\cup\dots\cup\partial\Sigma_g$ is non-separating in $\partial W$;
\item $W$ admits a map onto a handlebody $U$ extending a homeomorphism $\partial W\to\partial U$.
\end{enumerate}
\end{theorem}

As noted in \cite{BeF}, using Perelman's solution of the Poincar\'e Conjecture one can 
reformulate some of the equivalent conditions of Theorem \ref{boundary-retractable}
in a more rigid way (see Theorem \ref{boundary-retractable2}).
A further condition equivalent to those of Theorem \ref{boundary-retractable}, in terms of Dehn surgery,
is given in \cite{OS}.

\begin{proof} This is similar to Smythe's characterizations of boundary links \cite{Sm}.
More formally, a proof of (4)$\Leftrightarrow$(5) is discussed in \cite{Lam}*{Theorem 2}, 
and (5)$\Leftrightarrow$(1) is established in \cite{JM}*{Theorem ~3}, whose proof proceeds 
by showing that (5)$\Rightarrow$(1)$\Rightarrow$(4)$\Rightarrow$(5).
The equivalence of (1) and (3) follows from that in Theorem \ref{retractable} along with
\cite{J1}*{Theorems 2, ~3}.
The equivalence of (2) and (3) follows from the proof of that in Theorem \ref{retractable}.
\end{proof}

Knotted handlebodies satisfying any of the equivalent conditions of Theorem \ref{boundary-retractable}
will be called {\it boundary-link-like} handlebodies; while those satisfying the conditions of Theorem \ref{retractable} 
can be called {\it HBL-like}.

\begin{examples*}
It is easy to see that every genus $1$ handlebody is boundary-link-like.
Moreover, given a boundary-link-like handlebody, it will remain boundary-link-like if we add any local knot to it.
A regular neighborhood in $\R^3$ of a connected graph in $\R^2\x\{0\}$ is always a boundary-link-like handlebody.
More generally, if $G\subset\R^3$ is an embedded graph which contains a boundary link $L$ whose complement in $G$ is 
a tree which is disjoint from some collection of disjoint Seifert surfaces for $L$, then a regular neighborhood of $G$ is 
a boundary-link-like handlebody. 
(See \cite{Su}*{Example 4.4} for a specific example.)

Examples of handlebodies that are not boundary-link-like but HBL-like; or not HBL-like can be found in
\cite{Lam}, \cite{JM}, \cite{J1}, \cite{Su}, \cite{BeF}*{\S8}, \cite{HT}.
\end{examples*}

\begin{mainthm} \label{main6} The Bing sling is not isotopic to any PL knot through knots that are intersections of nested sequences of 
boundary-link-like handlebodies.
\end{mainthm}

The proof consists of two steps.
\begin{enumerate}
\item If a knot $\K$ is the intersection of a nested sequence of boundary-link-like handlebodies, then its knot module%
\footnote{That is, $H_1(\widetilde{S^3\but\K})$. It is closely related to the Alexander module, see \cite{Kaw}*{Proposition 7.1.2(1)}.} 
is a torsion module (Corollary \ref{main6''}).
\item The Bing sling is not isotopic to any PL knot through knots whose knot module
is a torsion module (Corollary \ref{main6'}).
\end{enumerate}

The proof of (1) in fact works under a weaker hypothesis.
Instead of assuming that each handlebody $V$ in the given nested sequence is boundary-link-like, it suffices to assume that $[\K]\in H_1(V)$ 
can be represented by a PL knot $K\subset\partial V$ which bounds a Seifert surface $\Sigma$ in $W:=\overline{S^3\but V_i}$ such that 
each knot in $\Sigma$ is null-homologous in $W$.
However, we show that in the case where $V$ has genus $2$, this apparently weaker hypothesis is in fact equivalent to the hypothesis
that $V$ is boundary-link-like. 
This provides yet another characterization of boundary-link-like handlebodies in the case of genus $2$ (Theorem \ref{disjoint seifert}).

\begin{remark} J. McPherson \cite{Mc} constructed for any $g\ge 1$ a wild knot which is the intersection of a nested sequence of handlebodies 
of genus $g+1$, but not of those of genus $g$.
In fact, if $\K$ is a wild knot whose knot module has rank $\ge 1$ and which contains a tame arc, then, as noted in \cite{Br},
a connected sum $\K_\infty$ of a null-sequence of copies of $\K$ has knot module of infinite rank; then by \cite{Mc}*{Theorem 2} 
$\K_\infty$ is not the intersection of any nested sequence of handlebodies of bounded genus.
\end{remark}

\section{Some basic examples}

\begin{lemma} \label{Rolfsen-assertion}
Let $K$ be a knot in $S^3$ which meets an embedded PL disk $D^2$ in a single point $p$ of $D^2$
so that $D^2$ locally separates $K$ at $p$.%
\footnote{This means that for every sufficiently small neighborhood $U$ of $p$, not all connected components of 
$(K\cap U)\but\{p\}$ lie in the same component of $U\but D^2$. 
Let us note that this condition implies that $p\notin\partial D^2$.}
Then $K$ is isotopic to the unknot.
\end{lemma}
 
\begin{proof}
We may identify $(S^3,D^2)$ with the quotient of $(B^3,\,\partial B^3)$ by $\rho|_{\partial B^3}$, 
where $B^3$ is the unit $3$-ball in $\R^3$ and $\rho\:\R^3\to\R^3$ is the reflection $(x,y,z)\mapsto(-x,y,z)$.
The preimage of $K$ in $B^3$ is an arc $J_0$ with endpoints in $\partial B^3$, which may be assumed to be $(\pm1,0,0)$.
Let $J_1=[-1,1]\x\{0\}\x\{0\}$ and for $0<t<1$ let $J_t$ be the union of $([-1,-1+t]\cup [1-t,1])\x\{0\}\x\{0\}$
with the image of $J_0$ under the homothety $H_t\:\R^3\to\R^3$, $v\mapsto(1-t)v$.
It is easy to see that the family of arcs $J_t$ is parameterized by an isotopy and the image of this isotopy in $S^3$
is an isotopy from $K$ to an unknot.
\end{proof}

\begin{example}[Zastrow's Example] \label{zastrow}
Let $K_{\onehalf}$ be Rolfsen's wild knot \cite{Ro1}*{Figure 1}, which can be described
as the visible outline of a standard drawing of Alexander's horned sphere 
(cf.\ \cite{DV}*{Figure 2.11}) --- in other words, an equatorial circle of the horned sphere
which passes through all its wild points.
Thus $K_{\onehalf}$ has a Cantor set $C$ of wild points, and its complement $K_{\onehalf}\but C$ consists
of countably many tame arcs.
Let $A_0$ and $A_1$ be the two ``longest'' ones of all these arcs, and let $D_i$ be a small embedded PL disk
which meets $K_{\onehalf}$ in a single point $x\in A_i$ so that $D_i$ locally separates $K$ at $x$.
Let us orient each $D_i$ so that $\lk(K_{\onehalf},\partial D_i)=+1$.
By applying the construction in the proof of Lemma \ref{Rolfsen-assertion} to the disk $D_1$ 
we obtain an isotopy from $K_{1\!/\!2}$ to an unknot, which will be the second half $K_{[\onehalf,\,1]}$ 
of the desired isotopy $K_t$.
Its first half $K_{[0,\,\onehalf]}$ will be the time-reversed version of the isotopy from $K_{\onehalf}$ to an unknot 
yielded by the same construction in the proof of Lemma \ref{Rolfsen-assertion}, this time applied to the disk $D_0$.
Thus $K_i$ for $i=0,1$ is an unknot disjoint from $\partial D_i$, and moreover $(K_i,\partial D_i)$ 
is the Hopf link.

\begin{lemma} \label{complement-dr}
$X:=S^3\x I\but\bigcup_{t\in I} K_t\x\{t\}$ deformation retracts onto $(S^3\but K_{\onehalf})\x\{\frac12\}$.
\end{lemma}

\begin{proof} In the notation of the proof of Lemma \ref{Rolfsen-assertion}, it suffices to construct 
a deformation retraction $r_t$ of $B^3\x I\but\bigcup_{t\in I} J_t\x\{t\}$ onto $(B^3\but J_0)\x\{0\}$ 
which restricts to a $(\rho\x\id_I)$-equivariant deformation retraction of
$(\partial B^3\but\partial J_0)\x I$ onto $(\partial B^3\but\partial J_0)\x\{0\}$.
(Note that $\partial J_t=\partial J_0$ for all $t\in I$.)
It can be defined by $r_t(v,1-s)=\big(\min(\frac{s_t}{s},\frac{1}{||v||})v,\,1-s_t\big)$, where $s_t$ denotes
the linear motion from $s_0=s$ to $s_1=1$, that is, $s_t=(1-t)s+t$.
Let us note that $r_t$ is the restriction of a deformation retraction $R_t$ of $B^3\x I\but\{0\}\x\{1\}$ onto 
$B^3\x\{0\}$ which is defined by the same formula.
This $R_t$ also restricts to a deformation retraction of the cone $C:=\bigcup_{t\in [0,1)} (1-t)\mathring B^3\x\{t\}$ 
onto the open unit ball $\mathring B^3\x\{0\}$, and to a deformation retraction of $B^3\x I\but(C\cup \{0\}\x\{1\})$
onto $\partial B^3\x\{0\}$.
\end{proof}

Suppose that there exists a homotopy $J_t$ of a knot $J_0$ such that $\lk(J_0,K_0)=1$ and 
$J_t\cap K_t=\emptyset$ for each $t\in I$.
Since $K_0$ and $K_1$ are unknots, we may assume that $J_0=\partial D_0$ and $J_1=\partial D_1$.
The homotopy $J_t\x\{t\}$ between $J_0\x\{0\}$ and $J_1\x\{1\}$ has its values in $X$ and
extends by using the cylinders $J_0\x[0,\frac12]$ and $J_1\x[\frac12,1]$ (which also lie in $X$)
to a homotopy $J_t^+$ with values in $X$ between $J_0\x\{\frac12\}$ and $J_1\x\{\frac12\}$.
The image of $J_t^+$ under the deformation retraction then yields a homotopy $J'_t$ 
in $S^3\but K_{\onehalf}$ between $J_0$ and $J_1$.
The image of $J_t'$, being disjoint from $K_{\onehalf}$, is in particular disjoint from its subset $C$, so it
must also be disjoint from $K_{\onehalf}\cup N_n$ for some neighborhood $N_n$ of $C$, which can be chosen
to be a union of $2^n$ disjoint $3$-balls.

Now let us recall that $K_{\onehalf}$ was an equatorial circle of the horned sphere, which in turn is 
the boundary of Alexander's horned ball $B$.
If $H$ is an unknotted $1$-handle in $B$ connecting its polar caps and disjoint from the ball $N_0$,
then it is easy to see that $(B\cup N_n)\but H$ deformation retracts onto $K_{\onehalf}\cup N_n$.
Moreover, the deformation retraction extends to a homotopy of $S^3$ which restricts to an isotopy
on the complement of $(B\cup N_n)\but H$.
Hence $\pi_1\big(S^3\but(K_{\onehalf}\cup N_n)\big)\simeq\pi_1\Big(S^3\but\big((B\cup N_n)\but H)\Big)
\simeq\Z*\pi_1\big(S^3\but(B\cup N_n)\big)$, where the free factor $\Z$ is generated by the class $x$ 
of a loop passing through the hole $H$ drilled in the horned ball $B$.
Let also $m\in\pi_1\big(S^3\but(B\cup N_n)\big)$ be the image of the class of a meridian of 
the solid torus $B\cup N_0$.
Then $J_0$ and $J_1$ are in the free homotopy classes of $x$ and $xm$ respectively.
If they are freely homotopic, then $xm=x^g$ for some $g\in\pi_1\big(S^3\but(B\cup N_n)\big)$,
or equivalently $m=x^{-1}x^g$.
Thus $m$ belongs to the normal closure of the free factor $\Z$ of
$\Z*\pi_1\big(S^3\but(B\cup N_n)\big)$.
On the other hand, $m$ in fact belongs to the other free factor $\pi_1\big(S^3\but(B\cup N_n)\big)$.
It follows from the normal form theorem for free products (see e.g.\ \cite{Hat}*{\S1.2} or \cite{MKS}*{Theorem 4.1}) 
that this can happen only if $m=1$ \cite{MKS}*{Exercise 4.1.4}.
But it is well-known that the image of $m$ in $\pi_1(S^3\but B)$ is nontrivial \cite{BF}.
\end{example}

\begin{example} It is not hard to show that every PL isotopy of the first component of the Hopf link 
extends to a PL isotopy of the entire link (see \cite{HZ}*{Theorem 2}; compare \cite{Zast}*{Example 2.2}).
On the other hand:
\end{example}

\begin{proposition} For $n=0,1$ there exists a PL link of linking number $n$ and 
a PL isotopy of its first component which does not extend to a topological isotopy
of the link. 
\end{proposition}

\begin{proof}[Proof for $n=0$] 
The desired link consists of any nontrivial PL knot $K_0$ (for example, the trefoil)
and any PL knot $J_0$ in $S^3\but K_0$ which is null-homologous but not null-homotopic 
in $S^3\but K_0$.
Let $K_t$ be the usual PL isotopy from $K_0$ to an unknot $K_1$ (given, for instance, 
by a PL version of the proof of Lemma \ref{Rolfsen-assertion}, where the round ball 
$B^3$ is replaced by the cube $[-1,1]^3$).
Suppose that there exists a homotopy $J_t$ such that $J_t\cap K_t=\emptyset$ for each $t\in I$.
Since $K_1$ is the unknot and $J_1$ is null-homologous in $S^3\but K_1$, it is
null-homotopic in $S^3\but K_1$. Hence $J_0\x\{0\}$ is null-homotopic in
$X:=S^3\x I\but\bigcup_{t\in[0,1]} K_t\x\{t\}$. 
But it is easy to see that $X$ deformation retracts onto $(S^3\but K_0)\x\{0\}$
(by a version of the construction in the proof of Lemma \ref{complement-dr} where $B^3$ 
is replaced by $[-1,1]^3$ and hence the Euclidean norm by the $l_\infty$ norm). 
Hence $J_0$ is null-homotopic in $S^3\but K_0$, which is a contradiction.
\end{proof}

\begin{proof}[Proof for $n=1$] This is similar to the proof for $n=0$, if we additionally
use the following well-known lemma.

\begin{lemma} \label{non-meridian}
There exists a PL knot $K$ and a PL knot $J$ in $S^3\but K$ which has linking number 1
with $K$ but is not freely homotopic to a meridian%
\footnote{A meridian of $K$ is the boundary of a small embedded PL 2-disk which intersects $K$ in precisely one 
point and locally separates $K$ at this point.} 
of $K$. 
\end{lemma}

There are several proofs of this Lemma in the literature: a geometric proof in \cite{Ts}*{Theorem 3.9}
and several algebraic proofs in \cite{SWW}.
(See also Example \ref{2.5b} below, which provides a not so easy proof of a stronger assertion.)
For completeness, here is a simplified exposition of a proof found in \cite{SWW}*{p.\ 5 of the arXiv version}.

\begin{proof} Let $K$ be the trefoil knot. A Wirtinger presentation of $\pi_1(S^3\but K)$
is \[\left< x,y,z\mid z=x^{-1}yx,\ y=z^{-1}xz,\ x=y^{-1}zy \right>.\] As usual, one of the relations
is redundant. By expressing $z$ from the first relation and substituting for $z$ in
the second relation, we simplify the presentation as $\left< x,y\mid xyx=yxy \right>$. From
the latter relation we get $(xy)(xy)(xy)=(xyx)(yxy)=(xyx)(xyx)$, or $p^3=q^2$, where
$p=xy$ and $q=xyx$. Conversely, we can recover $x$ and $y$ as $x=p^{-1}q$ and $y=q^{-1}p^2$,
so the presentation further simplifies as $\left< p,q \mid p^3=q^2 \right>$.

This group surjects onto free product $(\Z/p)*(\Z/q)=\left< p,q \mid p^3=q^2=1 \right>$. 
We will use the well-known theorem about normal form of cyclically reduced elements of a free product 
(see e.g.\ \cite{MKS}*{Theorem 4.2}). 
The element $x^2y^{-1}=p^{-1}qp^{-1}qp^{-2}q$ of the free product is cyclically reduced and
alternates 6 times between the free factors $\Z/p$ and $\Z/q$. Hence it is not conjugate to the meridian
$x=p^{-1}q$ (which alternates only 2 times between the free factors),
and therefore it is not conjugate to any meridian (because all meridians are
conjugate to each other). Hence any knot $J$ in the complement of the trefoil $K$
representing the conjugacy class of $x^2y^{-1}$ is not freely homotopic to a meridian of $K$.
\end{proof}
\end{proof}

\begin{example}
The following theorem was proved in \cite{MR1}*{Theorem 1.3 + the proof of Theorem 2.2}: 
there exists a PL knot in the solid torus $S^1\times D^2$ which is homotopic, but not topologically 
isotopic to the core circle (by an isotopy within the solid torus).
\end{example}

\section{Proof of Theorem \ref{main5'}}\label{d+}

Theorem \ref{main5'} follows from Lemma \ref{dense} and Theorem \ref{second commutator}.

\begin{lemma} \label{dense} Let $X_t$ be an isotopy of a knot, $t\in[0,1]$.
Then there exist numbers $0=t_0<\dots<t_n=1$ and PL knots $K_0,\dots,K_{n-1}$ such that each $K_i$ 
is disjoint from $X_t$ for all $t\in [t_i,t_{i+1}]$, and also $\lk(K_i,X_{t_i})=1$.

Moreover, the $t_i$ may be chosen to belong to a given dense subset $S$ of $[0,1]$.
\end{lemma}

\begin{proof}
If $S$ is not given, let $S=\Q\cap[0,1]$.
For each $s\in S$ let $Q_s$ be some PL knot disjoint from $X_s$ and such that $\lk(Q_s,X_s)=1$.
Then there exists an $\eps_s>0$ such that $Q_s$ is disjoint from $X_t$ for each $t\in (s-\eps_s,\,s+\eps_s)\cap[0,1]$.
Since $S$ is dense in $[0,1]$, the intervals $J_s:=(s-\eps_s,\,s+\eps_s)\cap[0,1]$ for $s\in S$ form 
an open cover of $[0,1]$.
Then it has a finite subcover $J_{s_1},\dots,J_{s_n}$.
By discarding those of the $J_{s_i}$ that can be discarded we may assume that none of the $J_{s_i}$ can be discarded.
Then each $J_{s_i}\cap J_{s_{i+1}}$ is nonempty; let $t_i$ be some point in it.
It remains to set $t_0=1$, $t_n=1$ and $K_i=Q_{s_i}$.
\end{proof}

\begin{theorem} \label{second commutator}
Let $\K$ be a knot in $S^3$ and let $Q_1$, $Q_2$ be PL knots in $S^3\but\K$ such that $\lk(Q_i,K)=1$.
Then $Q_1$ and $Q_2$ represent the same conjugacy class in the quotient of $\pi_1(S^3\but\K)$ modulo
the second commutator subgroup.
\end{theorem}

Geometrically, this theorem is very surprising.%
\footnote{The author has tried to disprove it for a week for a certain wild knot $K$
by finding an appropriate epimorphism of $\pi_1(S^3\but\K)$ onto a finite group, and exhausted 
all finite groups of order $\le 100$ in this search.}
In the geometric language it says that the connected sum $Q_1\#_b (-Q_2)$ along some band disjoint from $\K$ 
bounds a singular grope of height $2$ in $S^3\but\K$.
It is not hard to construct a wild knot $\K$, using Fox--Artin wild arcs or Fox's 
``remarkable simple closed curve'', which has two wild points, and its two 
tame arcs $A$, $B$ are ``linked'' with each other so that (in contrast to 
Zastrow's example, see Example \ref{zastrow}) any meridian of $A$ visually appears 
to be highly unlikely to cobound a grope of height 2 with any meridian of $B$.
But Theorem \ref{metabelian} says that such a grope does exist!%
\footnote{A. Zastrow was able to explicitly construct such a grope for a certain wild knot 
satisfying the above specifications (private communication, August 2023).
It is different from the wild knot of the previous footnote.
This had been before I found the geometric proof of Theorem \ref{second commutator} (see Theorem \ref{boundary}).}

\begin{lemma} \label{metabelian} If $G$ is a metabelian group with $G/G'\simeq\Z$, 
and $a,b\in G$ are such that the abelianization map $G\to G/G'$ sends $a$ and $b$
to the same generator of $\Z$, then $a$ is conjugate to $b$.
\end{lemma}

\begin{proof} The idea of this proof comes from \cite{BCR}*{proof of Theorem 2.3}.

Let us consider the subgroup $[a,G']$ of $G$ generated by all 
commutators of the form $[a,c]$ where $c\in G'$.%
\footnote{Hereafter $[x,y]=x^{-1}y^{-1}xy$ and $x^y=y^{-1}xy$.}
Using the identity $[x,yz]=[x,z][x,y]\big[[x,y],z\big]$ and the hypothesis that $G'$ is abelian 
it is easy to see that $[a,c][a,d]=[a,cd]$ and $[a,c]^{-1}=[a,c^{-1}]$ for $c,d\in G'$. 
Hence every element of $[a,G']$ (and not just its generators) can be represented 
in the form $[a,c]$ for some $c\in G'$. 
On the other hand, using the identities $x^y=x[x,y]$ and $[x,y]=[y,x]^{-1}$ and 
the metabelian Hall--Witt identity $\big[[x,y],z\big]\big[[y,z],x\big]\big[[z,x],y\big]=1$,%
\footnote{The usual Hall--Witt identity $\big[[x,y],z^x\big]\big[[y,z],x^y\big]\big[[z,x],y^z\big]=1$ 
holds in all groups; the first factor $\big[[x,y],z^x\big]=\big[[x,y],z[z,x]\big]=
\big[[x,y],[z,x]\big]\big[[x,y],z\big]\Big[\big[[x,y],z\big],[z,x]\Big]$
simplifies as $\big[[x,y],z\big]$ in metabelian groups, and similarly for the second and third ones.}
we get $[a,c]^g=[a,c]\big[[a,c],g\big]=[a,c]\big[[g,a],c\big]^{-1}\big[[c,g],a\big]^{-1}=[a,c]\big[a,[c,g]\big]$.
Hence $[a,G']$ is normal in $G$.

Given $g,h\in G$, we have $g=a^m c$ and $h=a^n d$ for some $m,n\in\Z$ and $c,d\in G'$. 
Using the identities $[x,yz]=[x,z][x,y]\big[[x,y],z\big]$ and $[xy,z]=[x,z]\big[[x,z],y\big][y,z]$ we get
$[g,h]=[a^m,d][c,a^n]=[a^m,d][a^n,c]^{-1}$. 
Then using the identity $[xy,z]=[x,z]^y[y,z]$ and the normality of $[a,G']$ we get that 
$[g,h]\in [a,G']$. 
Hence $G'=[a,G']$. 
Since $a^{-1}b\in G'=[a,G']$, we have $a^{-1}b=[a,c]=a^{-1}a^c$ for some $c\in G'$. 
Hence $b=a^c$.
\end{proof}

Theorem \ref{second commutator} is a special case of Lemma \ref{metabelian}.
But in the geometric context there is also a more illuminating proof of Theorem \ref{second commutator},
based on the knot module.

Let $K$ be a knot in $S^3$.
Let $X=S^3\but K$ and let $p\:\tilde X\to X$ be the infinite cyclic covering.
The group $H_1(X)\simeq\Z\simeq\left<t\mid\,\right>$ acts on $\tilde X$ by covering transformations.

\begin{lemma} \label{torsion} 
The action by $t-1$ on $H_1(\tilde X)$ and on $H_2(\tilde X)$ is by automorphisms.
\end{lemma}

\begin{proof} The proof for PL knots found in \cite{M24-2}*{Lemma \ref{part2:torsion}(a)}
also works for topological knots, since so does the Alexander duality (using \v Cech cohomology;
see for instance \cite{M2}*{Theorem 4.3} or \cite{M00}*{Theorem 27.1(a)}).
\end{proof}

\begin{proof}[Proof of Theorem \ref{second commutator}]
Let $G=\pi/\pi''$, where $\pi=\pi_1(S^3\but K)$.
Let $a,b\in G$ be some members of the conjugacy classes represented by $[Q_1]$ and $[Q_2]$.
Let $A=\pi'/\pi''$.
Since $\pi/\pi'\simeq\Z$ is a free group, the short exact sequence
$1\to A\to G\to\Z\to 1$ splits, and so $G$ is a semidirect product $G\simeq A\rtimes_\phi\Z$ 
with respect to some automorphism $\phi$ of $A$; namely, $\phi(g)=g^a$.
Clearly, $A$ with this action of $\Z$ is the knot module $H_1(\tilde X)$; 
more precisely, $\phi(g)=t\cdot g$ in the additive notation, where $t$ is 
the image of $a$ in $G/G'\simeq\Z$.
Moreover, the action by $1-t$ is an automorphism of $A$ (see Lemma \ref{torsion}(a)).

We want to show that $a^g=b$ for some $g\in A$.
This can be rewritten as $a^{-1}g^{-1}ag=a^{-1}b$, or $(g^{-1})^ag=c$, where $c=a^{-1}b\in A$.
In the additive notation, this becomes $g-t\cdot g=c$, or $(1-t)g=c$.
But since the action by $1-t$ is invertible, such a $g$ does exist.
\end{proof}

The second proof of Theorem \ref{second commutator} can actually be turned into a geometric proof
(see Theorem \ref{boundary} below) by using the techniques of the following lemma.

\begin{lemma}\cite{AADG} \label{C-complex}
Let $L=(K_1,K_2)$ be a PL link with $\lk(L)=1$.
Then there exist Seifert surfaces $F_i$ for $K_i$, intersecting transversely and such that 
$F_1\cap F_2$ is a single arc with one endpoint in $K_1$ and another in $K_2$.
\end{lemma}

The proof is quite simple, so we include it for the reader's convenience.

\begin{proof} Let $F_i$ be an arbitrary Seifert surface for $K_i$.
We may assume that $F_1$ and $F_2$ intersect transversely.
Since $\lk(L)=1$, there are $n$ negative and $n+1$ positive intersection points in $F_1\cap K_2$ 
for some non-negative integer $n$.
If $n>0$, at least one of the arcs into which these intersection points split $K_2$ has endpoints of opposite signs.
By attaching a thin tube to $F_1$ going along this arc we eliminate these two intersection points.
By continuing in this fashion we eliminate all of the intersection points in $F_1\cap K_2$ but one.
By similarly attaching tubes to $F_2$ going along arcs of $K_1$ we eliminate all intersection points
in $F_2\cap K_1$ except one.
The two remaining points in $\partial(F_1\cap F_2)$ must cobound an arc $J$ in $F_1\cap F_2$.
If $F_1\cap F_2$ also contains closed components, then there is one, call it $C$, such that some 
point of $C$ can be joined to an interior point of $J$ by an arc $A$ in $F_1$ which is otherwise 
disjoint from $F_1\cap F_2$ and from $K_1$.
Since $J$ is non-separating in $F_1$, one can choose $A$ so that both its ends approach $F_2$
from the same side.
Then by attaching a thin tube to $F_2$ going along $A$ we will replace $J$ and $C$ by their connected sum.
By continuing in this fashion we will eliminate all closed components of $F_1\cap F_2$.
\end{proof}

\begin{theorem} \label{boundary}
Let $\K$ be a knot in $S^3$ and let $Q_1$, $Q_2$ be smooth knots in $S^3\but\K$ such that $\lk(Q_i,K)=1$.
Then the connected sum $Q_1\#_b(-Q_2)$ along some band disjoint from $\K$ bounds a Seifert surface $F$
in $S^3\but\K$ such that each knot in $F$ is null-homologous in $S^3\but\K$.
\end{theorem}

\begin{proof} Since the link $L:=Q_1\cup(-Q_2)$ in null-homologous in $S^3\but\K$, it bounds 
a smooth Seifert surface $S$ in $S^3\but\K$.%
\footnote{This follows similarly to \cite{M21}*{Lemma 2.1} and \cite{Sa}*{\S1}. 
In more detail, let $N$ be a tubular neighborhood of $L$ in $S^3\but\K$ and let $L_+\subset\partial N$ 
be a zero pushoff of $L$, which is null-homologous in $S^3\but L$.
Since $L_+$ is also null-homologous in $S^3\but\K$, it is null-homologous in $S^3\but(\K\cup L)$ 
(by the Mayer--Vietoris exact sequence), and hence also in $\overline{S^3\but N}\but\K$.
Then by Lemma \ref{steenrod-realization} $L_+$ bounds an orientable smooth surface in 
$\overline{S^3\but N}\but\K$.
Using the obvious pair of annuli cobounded by $L$ and $L_+$ in $N$, this extends to a smooth
Seifert surface of $L$ in $S^3\but\K$.}
Let $\Sigma$ be a smooth $h$-Seifert surface bounded by $\K$ in $S^3$ (see \cite{M21}*{Lemma 2.1}).
We may assume that $S$ and $\Sigma$ intersect transversely. 
Since $\lk(Q_i,\K)=1$, by attaching tubes to $\Sigma$ as in the proof of Lemma \ref{C-complex} 
we can eliminate all its intersections with $Q_i$ except one, call it $p_i$.
The two remaining intersection points $p_1$ and $p_2$ must cobound an arc $J$ in $\Sigma\cap S$.
If $\Sigma\cap S$ also contains closed components, then by attaching tubes to $\Sigma$ along arcs in $S$ 
as in the proof of Lemma \ref{C-complex}, using that $J$ is non-separating in $S$, we can make them 
absorbed into $J$.
Now $S\cap\Sigma$ consists of a single arc $J$.
Let $b$ be its regular neighborhood in $S$.
Then $Q:=Q_1\#_b(-Q_2)$ bounds the Seifert surface $F:=\overline{S\but b}$.
By construction $F$ is disjoint from $\Sigma$.
Hence every knot in $F$ has zero linking number with $\K$.
\end{proof}

\section{Proof of Theorem \ref{main5}}

We recall that a group $G$ is called $n$-solvable if $G^{(n)}=1$.
Let us note that $1$-solvable, $1.5$-solvable and $2$-solvable groups are also known as abelian, nilponent of class 2 and metabelian, respectively.

\begin{example} \label{2.5a}
There exist a finitely presented $2.5$-solvable group $G$ such that $G/G'\simeq\Z$
and $a,b\in G$ such that the abelianization map $G\to G/G'$ sends $a$ and $b$ to the same generator of $\Z$,
but $a$ is not conjugate to $b$.

First suppose that $G$ is the desired group.
Since $G/G'\simeq\Z$, we have a short exact sequence $1\to G'\to G\to\Z\to 1$, which splits since $\Z$ is free.
Hence $G\simeq G'\rtimes_\phi\Z$ for some automorphism $\phi$ of $G'$.
On the other hand, saying that $G$ is $2.5$-solvable, that is, $[G',G'']=1$, amounts to saying that 
$H:=G'$ is $1.5$-solvable, that is, $[H,H']=1$ (in other words, $H$ nilpotent of class $2$).

Thus it makes sense to seek $G$ in the form $G=H\rtimes_\phi\Z$ for some automorphism $\phi$ of $H$, where
$H$ is nilpotent of class $2$.
If we show that $G'=H$, then we will get that $[G',G'']=[H,H']=1$, that is, $G$ is $2.5$-solvable.
Let $\bar\phi$ be the automorphism of the abelianization $H/H'$ determined by $\phi$.

\begin{lemma} If $\bar\phi-\id\:H/H'\to H/H'$ is surjective, then $G'=H$.
\end{lemma}

\begin{proof}
Geometrically, $G=\pi_1(MT,pt)$, where $MT$ is the mapping torus of a self-map 
$\Phi\:(BH,pt)\to (BH,pt)$ of the classifying space $BH=K(H,1)$ representing the automorphism $\phi$
(see Lemma \ref{mapping-torus}(b)).
Using the isomorphism $H_i(MT,BH)\simeq H_i(BH\x I,\,BH\x\partial I)\simeq H_{i-1}(BH)$,
the exact sequence of the pair $(MT,BH)$ can be written as
\[H_1(BH)\xr{\bar\phi-\id} H_1(BH)\xr{\subset_*}H_1(MT)\to H_0(BH)\xr{\id-\id} H_0(BH).\]
Since $\bar\phi-\id$ is surjective, the map $H_1(BH)\to H_1(MT)$ is trivial, and it follows that 
$H_1(MT)\simeq\Z$.
In other words, the map $H/H'\to G/G'$ is trivial, and also $G/G'\simeq\Z$.%
\footnote{Algebraically, the same conclusions can be similarly derived from Bieri's long exact sequence 
\cite{Bie}: $H_1(H)\xr{\bar\phi-\id}H_1(H)\to H_1(G)\to H_0(H)\xr{\id-\id} H_0(H)$.
In fact, the two long exact sequences are the same.
Indeed, since $\Phi$ is a homotopy equivalence, $MT$ is the total space of 
an approximate fibration over $S^1$ \cite{CD}.
Consequently, since $S^1$ is aspherical, so is $MT$.
Alternatively, with the of choice $H$ that will interest us below, $BH$ can be chosen to be a closed 
$3$-manifold (the quotient of the continuous Heisenberg group by the discrete Heisenberg group), and 
$\Phi$ can be chosen to be a homeomorphism \cite{KrL}*{Theorem 0.7}.
In this case $MT$ will be the total space of a (genuine) fibration over $S^1$.}
It follows from the first observation that $H$ lies in $G'$.
Then $G/G'$ is a quotient of $G/H$; but both are isomorphic to $\Z$, which is a Hopfian group, 
so we get that $H=G'$.
\end{proof}

\begin{lemma} \label{mapping-torus}
(a) If $\left<x_1,\dots,x_n\mid r_1,\dots,r_m\right>$ is a presentation of a group $H$ and $\phi$
is an automorphism of $H$, then a presentation of $H\rtimes_\phi\Z$ is given by
\[\left<x_1,\dots,x_n,t\mid r_1=1,\dots,r_m=1,\,x_1^t=\phi(x_1),\dots,x_n^t=\phi(x_n)\right>.\]

(b) Let $X$ be a CW complex, $b\in X$ a point and $f\:(X,b)\to (X,b)$ a map such that $f_*\:\pi_1(X,b)\to\pi_1(X,b)$ is an isomorphism.
Let $T_f$ be the mapping torus of $f$, that is, $T_f=X\x [0,1]/\sim$, where $(x,0)\sim \big(f(x),1\big)$.
Then $\pi_1(T_f,b)\simeq\pi_1(X,b)\rtimes_{f_*}\Z$.
\end{lemma}

This is well-known. Part (a) says that ``a semidirect product with $\Z$ is special kind of an HNN-extension'' and 
part (b) is related to the assertion that ``an HNN-extension is the fundamental group of a certain graph of groups''.

\begin{proof}[Proof. (a)] It follows from the ``inner'' and ``outer'' definitions of a semidirect product that $H$ and 
$\Z=\left<t\mid\ \right>$ can be identified with subgroups of $H\rtimes_\phi\Z$, that every element of $H\rtimes_\phi\Z$ 
can be expressed in the form $t^kh$ for some $h\in H$ and $k\in\Z$, and that $t^kh\cdot t^lh'=t^{k+l}\phi^l(h)h'$ for 
$h,h'\in H$ and $k,l\in\Z$.
In particular, $h^t=\phi(h)$ for all $h\in H$, and it follows that the relations $x_i^t=\phi(x_i)$ and $r_j=1$ 
do hold in $H\rtimes_\phi\Z$.
Conversely, if $h^t=\phi(h)$ holds for $h=x_i$, then it clearly holds for all $h\in H$, which in turn implies
the multiplication law, $t^kh\cdot t^lh'=t^{k+l}\phi^l(h)h'$.
\end{proof}

\begin{proof}[(b)] We may assume that $b$ is a $0$-cell of $X$ and there are no other $0$-cells, and that $f$ is a cellular map.
Then $T_f$ can be identified with the CW complex obtained from $X$ by attaching one $(n+1)$-cell corresponding to 
each $n$-cell of $X$ in the obvious way; namely, $D^n\xr{e}X$ yields $D^n\x I\xr{e\x\id_I}X\x I\to T_f$.
It is a well-known consequence of the Seifert--van Kampen theorem that a presentation of $\pi_1(X,b)$ is determined by 
the $2$-skeleton of $X$ (with generators given by the $1$-cells and relators by the $2$-cells).
A presentation of $\pi_1(T_f,b)$ is similarly determined by the $2$-skeleton of $T_f$, and now the assertion follows from (a).
\end{proof}

Now let, in fact, $H$ be the {\it free} nilpotent group of class $2$ on two generators, which has the presentation
$\left<x,y\mid \big[[x,y],x\big]=\big[[x,y],y\big]=1\right>$ and is also known as the discrete Heisenberg group.
By using the commutator collecting process (see \cite{Hall}*{\S11}) it is easy to see that every element of $H$ 
can be uniquely written as $x^ly^m[y,x]^n$ for some $l,m,n\in\Z$, and 
$x^ly^m[y,x]^n\cdot x^{l'}y^{m'}[y,x]^{n'}=x^{l+l'}y^{m+m'}[y,x]^{n+n'+ml'}$.
It follows that $\big(x^ly^m[y,x]^n\big)^{-1}=x^{-l}y^{-m}[y,x]^{-n+lm}$.

Let $\phi\:H\to H$ be given%
\footnote{Since $H$ is a free nilpotent group of class $2$, every map of the set of its generators to a class $2$ nilpotent group $K$
extends to a homomorphism $H\to K$.}
by $\phi(x)=y^{-1}$ and $\phi(y)=yx$.
Let us note that $\phi$ is an isomorphism, with $\phi^{-1}$ given by $\phi^{-1}(y)=x^{-1}$ and $\phi^{-1}(x)=xy$.
The abelianization $\bar\phi\:H/H'\to H/H'$ of $\phi$ is given by the matrix
$\big(\begin{smallmatrix}0 & 1\\ -1 & 1\end{smallmatrix}\big)$.
Therefore $\bar\phi-\id\:H/H'\to H/H'$ is given by the matrix
$\big(\begin{smallmatrix}-1 & 1\\ -1 & 0\end{smallmatrix}\big)$, and so is also an isomorphism.

It is easy to see that $\phi([y,x])=[yx,y^{-1}]=[x,y^{-1}]=[y,x]^{y^{-1}}$, which equals $[y,x]$ in $H$.
Hence $\phi\big(x^ly^m[y,x]^n\big)=y^{-l}(yx)^m[y,x]^n=x^my^{m-l}[y,x]^{n-lm+m(m+1)/2}$, 
where $(yx)^k=x^ky^k[y,x]^{k(k+1)/2}$ by using the commutator collecting process.

The semidirect product $G=H\rtimes_\phi\Z$ consists of all products of the form $t^kh$, where $h\in H$ and $k\in\Z$, 
and satisfies the relation $h^t=\phi(h)$, or equivalently $ht=t\phi(h)$ for all $h\in H$ (see the proof of 
Lemma \ref{mapping-torus}(a)).

Finally, let us show that $a:=t$ and $b:=txy$ are not conjugate by any $g\in G$.
We have $g=t^kx^ly^m[y,x]^n$ for some $k,l,m,n\in\Z$.
Then 
\begin{multline*}
a^g=t^{t^kx^ly^m[y,x]^n}=x^{-l}y^{-m}[y,x]^{-n+ml}\,t\,x^ly^m[y,x]^n=t\phi\big(x^{-l}y^{-m}[y,x]^{-n+lm}\big)x^ly^m[y,x]^n\\
=tx^{-m}y^{-m+l}[y,x]^{-n+lm-(-l)(-m)+(-m)(-m+1)/2}x^ly^m[y,x]^n=tx^{l-m}y^l[y,x]^{l(l-m)+m(m-1)/2}.
\end{multline*}
If $a^g=b$, then $l-m=1$, $l=1$ and $l(l-m)+m(m-1)/2=0$.
The first two equations imply $m=0$, and then the third equation implies $1=0$.
\end{example}

\begin{figure}[h]
\includegraphics{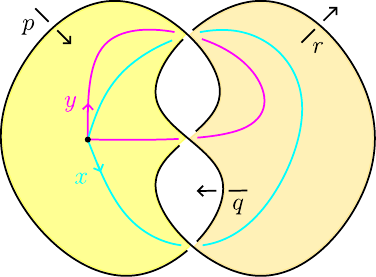}
\caption{}
\label{trefoil}
\end{figure}

\begin{example} \label{2.5b}
The fundamental group $\pi:=\pi_1(S^3\but Q)$ of the trefoil knot $Q$ contains elements $a,b$ which map 
to the same generator of $H_1(S^3\but Q)\simeq\Z$ but their images in $\pi/\pi^{(2.5)}$ are not conjugate.

Indeed, the trefoil is well-known to be a fibered knot, more precisely $S^3\but Q$ fibers over $S^1$ with 
fiber being the punctured torus (see \cite{Ro-book}*{pp.\ 327--332}, see also \cite{BZ}*{5.12}).
Hence $\pi\simeq\pi'\rtimes_\psi\Z$, where the commutator subgroup $\pi'$ is isomorphic to the free group 
$F_2=\left<x,y\mid\right>$, and $\psi$ is its automorphism induced by the monodromy self-homeomorphism 
of the fiber (see \cite{BZ}*{5.4}, see also \cite{Ro-book}*{p.\ 332}).
In order to prepare for the next example, we recall a computation of the monodromy $\psi$ (this computation
is a slight variation of \cite{BZ}*{5.15}).
Figure \ref{trefoil} shows a Seifert surface $S$ of the minimal genus for the trefoil $K$; loops $x$, $y$ in $S$ 
whose homotopy classes freely generate $\pi_1(S)$; and the generators of the Wirtinger presentation 
$\left<p,q,r\mid pq^{-1}=r^{-1}p,\, qr^{-1}=p^{-1}q,\, rp^{-1}=q^{-1}r\right>$ for $\pi$.
(The first relation corresponds to the upper crossing, the second relation to the middle crossing and
the third relation to the lower crossing in the figure.
As is well-known, any one of the relations is redundant.)
It is easy to see from the figure that $x=r^{-1}p$ and $y=p^{-1}q$.
By using the relations we also get $x=pq^{-1}$ and $y=qr^{-1}$.
Let us note that $\psi$ depends on an identification of $\Z$ with a subgroup of $\pi$; 
if, for instance, the generator $t$ of $\Z$ is identified with $p$, then $\psi(g)=g^p$.
We have $x^p=p^{-1}pq^{-1}p=q^{-1}p=y^{-1}$ and $y^p=p^{-1}qr^{-1}p=yx$.
Thus $\psi(x)=y^{-1}$ and $\psi(y)=yx$.

Since the automorphism $\phi$ of the previous example is given by the same formula, we obtain a homomorphism
$f\:\pi\to G$, where $G$ is the group of the previous example.
Clearly $f$ is surjective.
Since $G$ is $2.5$-solvable, $f$ factors through the quotient $\pi/\pi^{(2.5)}$.%
\footnote{In fact, the epimorphism $g\:\pi/\pi^{(2.5)}\to G$ is an isomorphism.
Indeed, since $\pi^{(2.5)}=[\pi'',\pi']=[F_2',F_2]$, the image of $F_2$ in $\pi/\pi^{(2.5)}$
is the free nilpotent group $H$ of class $2$, and it follows that $\ker g=1$.}
Now the assertion follows from the previous example.
\end{example}

\begin{figure}[h]
\includegraphics[width=\linewidth]{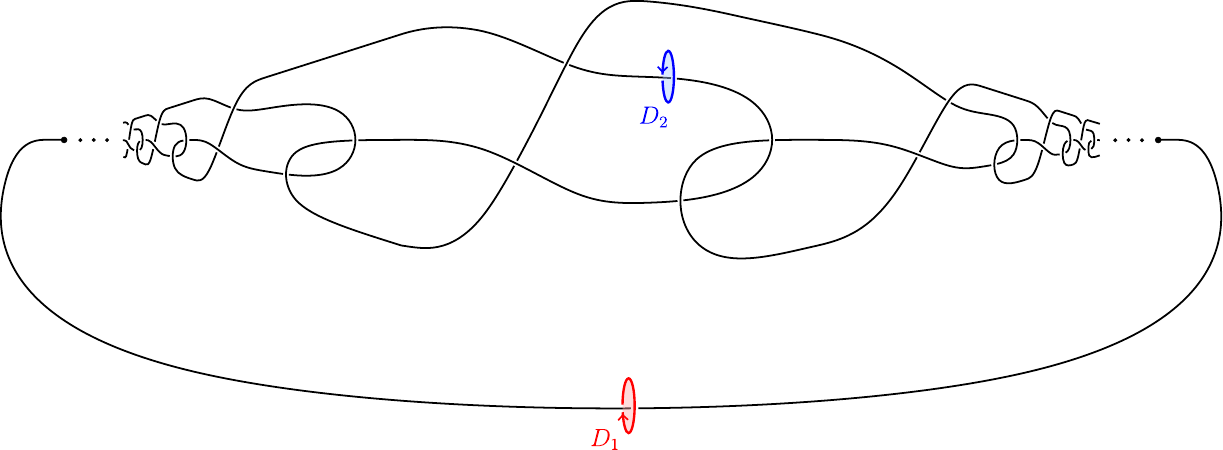}
\caption{}
\label{fox-artin-trefoil}
\end{figure}

\begin{example} \label{2.5c}
There exists a wild knot $\K\subset S^3$ and its meridians $\partial D_1$, $\partial D_2$%
\footnote{That is, each $\partial D_i$ bounds an embedded PL disk $D_i\subset S^3$ which meets $\K$ in 
a single point and locally separates it at that point.} 
such that the conjugacy classes in $\pi:=\pi_1(S^3\but\K)$ corresponding to
$\partial D_1$ and $\partial D_2$ have different images in $\pi/\pi^{(2.5)}$.

The wild knot $\K$ is shown in Figure \ref{fox-artin-trefoil}.
It is the union of two closed arcs with common endpoints: a tame arc $A$ (which intersects the disk $D_1$) 
and a wild arc $A'$ (which intersects the disk $D_2$).
Since $A$ is tame, $S^3\but A$ is homeomorphic to $\R^3$, which is in turn homeomorphic to the infinite cyclic cover 
of $S^3\but B$, where $B$ is an unknot.
Moreover, it is easy to see that the composite homeomorphism can be chosen so that $\K\but A$ gets identified
with the preimage under the covering map $S^3\but A\to S^3\but B$ of the knot $Q\subset S^3\but B$ which 
is shown on the left hand side of Figure \ref{mazur-trefoil}.
Let $L=B\cup Q$.
The images of $\partial D_i$ under the covering map are the knots $J_i$ in $S^3\but L$, which
are also shown on the left hand side of Figure \ref{mazur-trefoil}.
Hence it suffices to show that the conjugacy classes in $\bar\pi:=\pi_1(S^3\but L)$
corresponding to $J_1$ and $J_2$ have different images in $\bar\pi/\bar\pi^{(2.5)}$.
For this it in turn suffices to show that the conjugacy classes in $\hat\pi:=\pi_1(S^3\but Q)$
corresponding to $J_1$ and $J_2$ have different images in $\hat\pi/\hat\pi^{(2.5)}$.
The knot $Q$ is easily seen to be equivalent to the trefoil, and the right hand side of Figure \ref{mazur-trefoil} 
shows the images of $J_1$ and $J_2$ under the equivalence.
Then in the notation of the previous example (Example \ref{2.5b}), $J_2$ can be represented by $p$ and $J_1$ by $p^2r^{-1}$ in
$\pi_1(S^3\but Q)$.
Since $p^2r^{-1}=pxy$, in the notation of Example \ref{2.5a} these correspond to $t$ and $txy$.
Thus the assertion follows from the previous two examples.
\end{example}

\begin{figure}[h]
\includegraphics[width=\linewidth]{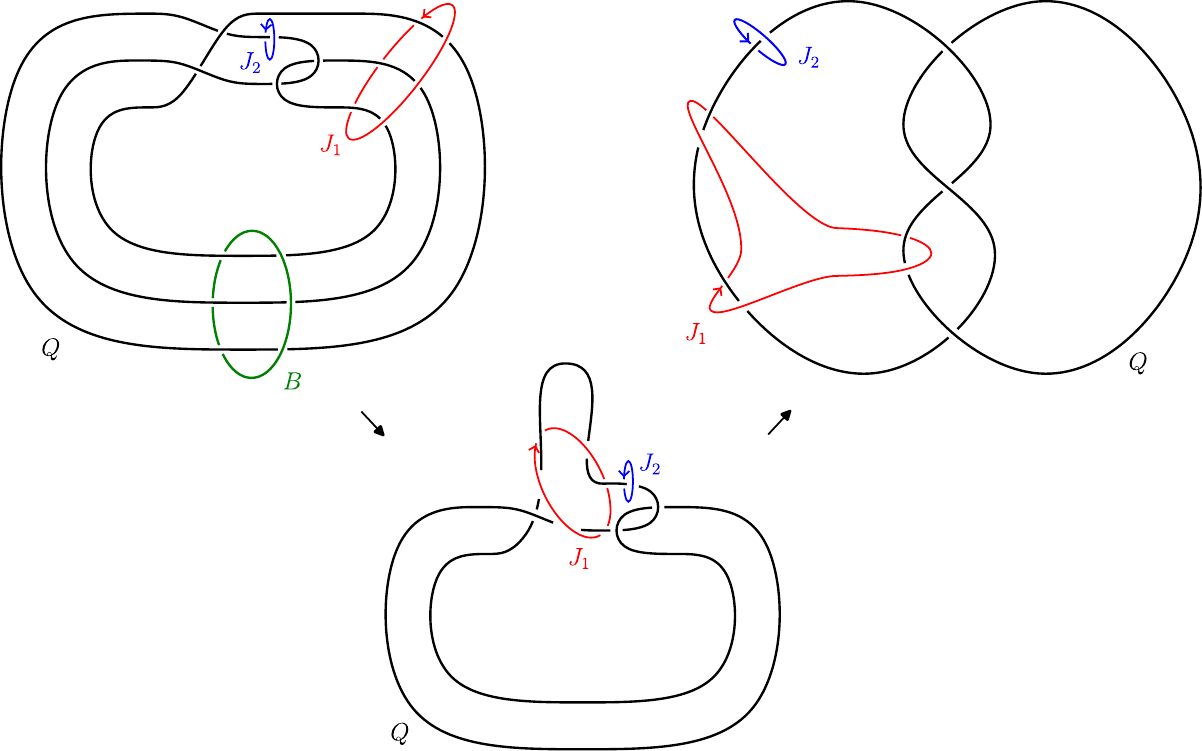}
\caption{}
\label{mazur-trefoil}
\end{figure}

\begin{proof}[Proof of Theorem \ref{main5}] Theorem \ref{main5} is proved by the construction of Zastrow's Example 
(see Example \ref{zastrow}) if instead of Rolfsen's wild knot and its meridional disks $D_0$ and $D_1$ we use the wild knot 
of Example \ref{2.5c} and its meridional disks $D_1$ and $D_2$.
\end{proof}

The strengthening of Theorem \ref{main5} stated in Remark \ref{torsion-remark} follows from the proof of Theorem ~\ref{main5},
taking into account the following observation.

\begin{proposition} \label{torsion-free}
The knot module $H_1(\widetilde{S^3\but\K})$ of the wild knot $\K$ shown in Figure \ref{fox-artin-trefoil}
is $\Z[t^{\pm1}]$-torsion free.
\end{proposition}

\begin{proof} As explained in Example \ref{2.5c}, $\K$ contains a closed tame arc $A$ such that the pair $(S^3\but A,\,\K\but A)$
is homeomorphic to $\big(\widetilde{S^3\but B},\,p^{-1}(Q)\big)$, where $L=B\cup Q$ is the $2$-component link in $S^3$ shown in 
Figure \ref{mazur-trefoil} (on the left) and $p\:\widetilde{S^3\but B}\to S^3\but B$ is the infinite cyclic covering.
It follows that the infinite cyclic covering $q\:\widetilde{S^3\but\K}\to S^3\but\K$ is fiberwise equivalent to the projection
$\widetilde{S^3\but L}\to\big(\widetilde{S^3\but B}\big)\but p^{-1}(Q)$ from the universal abelian covering of $S^3\but L$ to its 
infinite cyclic covering which extends to the infinite cyclic covering of $S^3\but B$.
Thus $H_1(\widetilde{S^3\but\K})$ is a $\Z[s^{\pm1},t^{\pm1}]$-module, where $s$ and $t$ are the positive generators of
the deck transformation groups of the coverings $p$ and $q$, respectively.

The infinite cyclic covering $\widetilde{S^3\but B}$ can be regarded as a copy of $\R^3$, which contains 
the periodic open knot $p^{-1}(Q)$.
Figure \ref{fox-artin-torsion} shows a periodic Seifert surface $F$ of this open knot in this $\R^3$, such that 
$F=p^{-1}(\Sigma\but B)$, where $\Sigma$ is a Seifert surface for $Q$ which meets $B$ transversely in one point.
The identifying homeomorphism between $S^3\but A$ and $\widetilde{S^3\but B}$ can be chosen so that $\F:=F\cup A$ is 
a Seifert surface for $\K$ in $S^3$.
Let us note that $S^3\but\F$ is identified with $\R^3\but F$.

\begin{figure}[h]
\includegraphics[width=\linewidth]{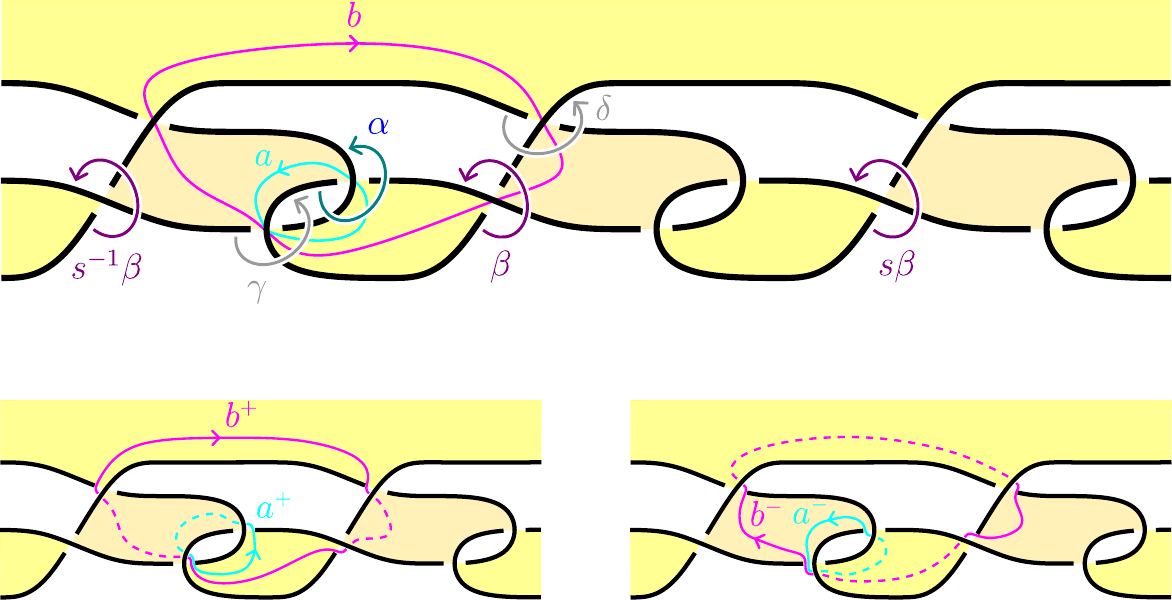}
\caption{}
\label{fox-artin-torsion}
\end{figure}

The upper half of Figure \ref{fox-artin-torsion} shows representative cycles for classes $a,b\in H_1(F)$ and for classes 
$\alpha,\beta,\gamma,\delta\in H_1(\R^3\but F)$.
Clearly, a basis in $H_1(F)$ over $\Z[s^{\pm1}]$ is given by $a$ and $b$.
It follows from the Alexander duality (see \cite{M24-2}*{\S\ref{part2:przytycki-yasuhara}}) that a basis of $H_1(\R^3\but F)$ over $\Z[s^{\pm1}]$ 
is given by $\alpha$ and $\beta$.
It is easy to see that
\[\begin{aligned}
\alpha+\gamma&=\beta\\
\beta+\delta&=s\beta.
\end{aligned}\tag{$*$}\]
Now let $a^\pm,b^\pm\in H_1(\R^3\but F)$ be the classes of the parallel $\pm$-pushoffs of representative cycles of $a$ and $b$.
Their representative cycles are shown in the lower half of Figure \ref{fox-artin-torsion}.
It is easy to see that
\[\begin{aligned}
a^+&=\alpha\\
a^-&=-\gamma
\end{aligned}\qquad\qquad\qquad\qquad
\begin{aligned}
b^+&=\beta+\gamma-s^{-1}\beta\\
b^-&=\beta-\delta.
\end{aligned}\]
Using ($*$), we can eliminate $\gamma$ and $\delta$:
\[\quad\begin{aligned}
a^+&=\alpha\\
a^-&=\alpha-\beta
\end{aligned}\qquad\qquad\qquad\quad
\begin{aligned}
b^+&=(2-s^{-1})\beta-\alpha\\
b^-&=(2-s)\beta.
\end{aligned}\tag{$**$}\]
Let $\tilde\F$ be some fixed lift of $\F\but\K$ in the infinite cyclic cover $\widetilde{S^3\but\K}$
and let $\tilde a^\pm$, $\tilde b^\pm$, $\tilde\alpha$ and $\tilde\beta$ be the classes in $H_1(\widetilde{S^3\but\K})$
of the lifts of some representative cycles of $a^\pm$, $b^\pm$, $\alpha$ and $\beta$ respectively lying between 
$\tilde\F$ and $t\tilde\F$, where $t$ is the positive generator of the group of deck transformations of $\widetilde{S^3\but\K}$.
Then $\tilde a^\pm$ and $\tilde b^\pm$ can be expressed in terms of $\tilde\alpha$ and $\tilde\beta$ similarly to ($**$),
and a presentation of $H_1(\widetilde{S^3\but\K})$ as a $\Z[s^{\pm1},t^{\pm1}]$-module is given by
$\big\langle\tilde\alpha,\tilde\beta\ \big|\ \tilde a^-=t\tilde a^+,\ \tilde b^-=t\tilde b^+\big\rangle$;
that is,
\[\big\langle\tilde\alpha,\tilde\beta\ \big|\ (\tilde\alpha-\tilde\beta)=t\tilde\alpha,\ 
(2-s)\tilde\beta=t(2-s^{-1})\tilde\beta-t\tilde\alpha\big\rangle.\]
From the first relation we have $\tilde\beta=(1-t)\tilde\alpha$.
Substituting this into the second relation and dividing by $t$, we get
$(1-t^{-1})(s-2)\tilde\alpha+(1-t)(s^{-1}-2)\tilde\alpha+\tilde\alpha=0$.
Expanding out, we arrive at the following presentation of $H_1(\widetilde{S^3\but\K})$ as a $\Z[s^{\pm1},t^{\pm1}]$-module:
\[\big\langle\tilde\alpha\ \big|\ (1-t^{-1})s\tilde\alpha+(1-t)s^{-1}\tilde\alpha+(2t+2t^{-1}-3)\tilde\alpha=0\big\rangle.\]
Writing $x_i=s^i\tilde\alpha$, we get a presentation of $H_1(\widetilde{S^3\but\K})$ as a $\Z[t^{\pm1}]$-module:
\[\big\langle x_i,\ i\in\Z\ \big|\ (1-t^{-1})x_{i+1}+(1-t)x_{i-1}+(2t+2t^{-1}-3)x_i=0,\ i\in\Z\big\rangle.\tag{$*{*}*$}\]
Since the multiplication by $1-t$ in $H_1(\widetilde{S^3\but\K})$ is an isomorphism (see Lemma \ref{torsion}(a)), this abelian group
is also a module over the ring $\Z[t^{\pm}][\frac1{1-t}]$ (that is, the localization of $\Z[t^{\pm1}]$ with respect to 
the multiplicatively closed set $\{(1-t)^i\mid i=0,1,2,\dots\}$).
The same presentation ($*{*}*$) works for $H_1(\widetilde{S^3\but\K})$ as a module over $\Z[t^{\pm}][\frac1{1-t}]$.
But over this ring the relation can be rewritten as
\[x_{i+1}=-(1-t^{-1})\big((1-t)x_{i-1}+(2t+2t^{-1}-3)x_i\big)\] and also as
\[x_{i-1}=-(1-t)\big((1-t^{-1})x_{i+1}+(2t+2t^{-1}-3)x_i\big),\] and it follows that
$H_1(\widetilde{S^3\but\K})$ is a free module of rank two over $\Z[t^{\pm}][\frac1{1-t}]$, freely generated
by $x_0$ and $x_1$ (for instance).
In particular, this implies that $H_1(\widetilde{S^3\but\K})$ contains no $\Z[t^{\pm1}]$-torsion.
\end{proof}

\section{Exactly what is the Bing sling?} \label{whatisbingsling}

Let $K_0$ be an unknot is $S^3$, and let $T_1$ be its regular neighborhood in $S^3$.
Let $J$ be a copy of the fake%
\footnote{\label{zeeman}It differs from the original Mazur curve \cite{Maz} in one crossing, the one which is marked by an asterisk in Figure \ref{mazur-curve}.
See \cite{MR1}*{\S2.2} for a discussion of one curve versus another.
See also \cite{Akb} for a discussion of the Mazur curve versus Zeeman's curve, which is a mirror image of the fake Mazur curve.} 
Mazur curve in the unknotted solid torus $T_1$ (see Figure \ref{mazur-curve}).
Given an $i\in\{1,2,\dots\}$ and assuming that a solid torus $T_i$ in $S^3$ has been defined, let $K_i$ be the preimage of $J$ under 
an $r_i$-fold cover $p_i\:T_i\to T_1$ and let $T_{i+1}$ be a regular neighborhood $K_i$ in $T_i$.
(In Figure ~\ref{bingsling}, $r_1=6$.)
We will always assume the following hypothesis:
\begin{enumerate}
\item[(0)]
each $p_i$ is untwisted, in the sense of being regular homotopic (with values in $T_1$) to the pullback 
$T_1\to T_1$ of the orientation-preserving $r_i$-fold cover $S^1\to S^1$ along the projection $T_1\simeq S^1\x D^2\to S^1$.
\end{enumerate}

\begin{figure}[h]
\includegraphics[width=0.33\linewidth]{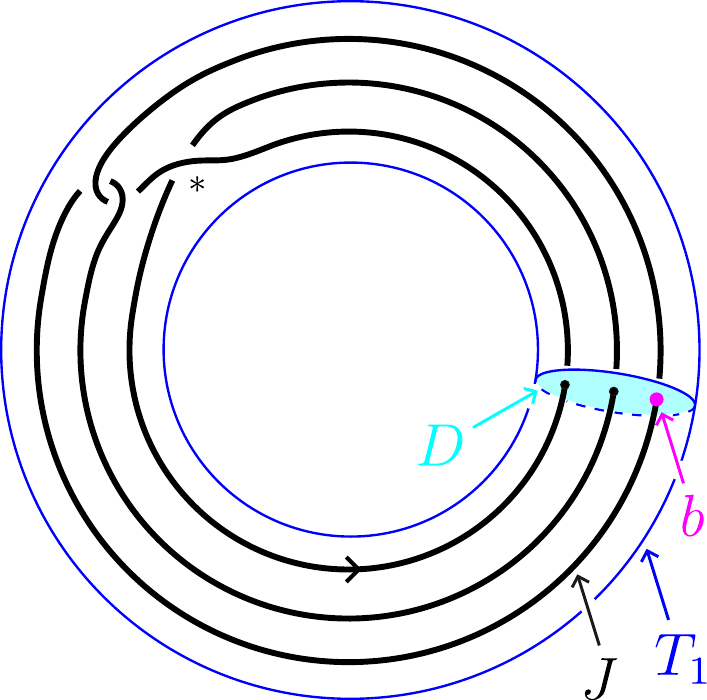}
\caption{The fake Mazur curve.}
\label{mazur-curve}
\end{figure}

While Bing himself did not explicitly discuss in his paper \cite{Bi} any of the choices involved, it is not hard to see, 
as explained in the book of R. Daverman and G. Venema \cite{DV}*{\S2.8.5}, that the numbers $r_i$, the covers $p_i$ 
(as specific maps, not just up to a fiberwise homeomorphism) and the regular neighborhoods $T_i$ can be chosen so that 
$\bigcap_{i=1}^\infty T_i$ is homeomorphic to $S^1$ and thus we get an embedding $\B$ of $S^1$.
However, with their approach (which is based on a general lemma \cite{DV}*{Proposition 2.2.2} 
that the uniform limit of a sequence of embeddings will be an embedding if they converge rapidly enough)
it seems to be challenging to describe any one resulting sequence $r_1,r_2,\dots$ in closed terms,
and it appears that the equivalence class of $\B$ (up to ambient isotopy) will depend not only on 
the integers $r_i$, but also on metric properties of the maps $p_i$.
Which could lead some readers to ask: ``OK, but exactly what do you mean by the Bing sling?
Is it a well-defined topological object? Can it be described explicitly enough that one can evaluate
at least one algebraic invariant on it and obtain an answer in closed form?''

We will take a different, more tailor-made approach, which presumably settles these questions in a satisfactory way.
Let us note that the meridional disk $D$ of $T_1$ shown in Figure \ref{mazur-curve} has $r_i$ distinct lifts $D_{ij}$ 
into $T_i$, which partition $T_i$ into $r_i$ (closed) fundamental domains $\Phi_{ij}$ of the group of deck transformations 
of the cover $p_i$.
(Figure \ref{bingsling} hints at the six disks $D_{1j}$ and the six cylinders $\Phi_{1j}$, as well as five of the disks $D_{2j}$ and
four of the cylinders $\Phi_{2j}$.)
Suppose that 
\begin{enumerate}
\item each $\Phi_{ij}$ is of diameter $\le\epsilon_i$, where $\epsilon_n\to 0$ as $n\to\infty$;
\item each $\Phi_{ij}\cap T_{i+1}$ has only $3$ connected components and each of them is a union of $q_i$ 
of the $\Phi_{i+1,\,k}$'s, where $q_i:=\frac{r_{i+1}}{3r_i}$ is required to be an integer.
\end{enumerate}
(In Figure \ref{bingsling}, $q_1=4$.)

\begin{theorem} \label{bingslinglemma} Suppose that conditions (0), (1) and (2) hold.
Then

(a) the knots $K_i$, when suitably parameterized, uniformly converge to an embedding $\B$ of $S^1$ whose image is $\bigcap_{i=1}^\infty T_i$;

(b) the equivalence class of $\B$ (up to ambient isotopy) depends only on the integers $r_i$.

\noindent
Moreover,

(c) given a sequence of positive integers $r_1,r_2,\dots$ such that each $\frac{r_{i+1}}{3r_i}$ is an integer greater than $1$,
there exists a sequence of $K_i$ and $T_i$ as above, satisfying (0), (1) and (2).
\end{theorem}

The approach in the proof of (a) will still work, albeit with more tedious details, if instead of (2) we assume the following 
weaker hypothesis:

\begin{enumerate}
\item[(2$'$)] each $\Phi_{ij}\cap T_{i+1}$ has only $3$ connected components and each of them is a union of some of
the $\Phi_{i+1,\,k}$'s.
\end{enumerate}

However, it is expected that (b) will be false if we replace (2) with (2$'$).%
\footnote{Let us note a potential  connection with Rolfsen's Problem: if (b) were true when (2) is replaced with (2$'$), then 
this would provide, in particular, an isotopy from the Bing sling (as defined more precisely below) to a somewhat 
differently looking wild knot, which might be easier to isotope to the unknot by some geometric construction.}

Although (c) tells us that for Bing's construction to work, the sequence $r_1,r_2,\dots$
need not be particularly rapidly growing after all (for instance, $r_i=6^{i-1}$ will do), we will nevertheless need to assume in 
the present paper (apart form the present section) that it is quite rapidly growing, because of our use of Theorem \ref{convergence}.
To be specific, the proofs of Theorems \ref{main1} and \ref{main2} work under the hypothesis that each $r_{i+1}\ge (r_i^2+3r_i)^2$ 
(see Remark \ref{growth1}), and the proofs of Theorems \ref{main3}, \ref{main4} and \ref{main6} work under the slightly 
weaker hypothesis $r_{i+1}\ge(r_i^2-r_i+1)^2+1$ (see Remark \ref{growth2}).
For instance, $r_i=3^{5^{i-1}}$ satisfies all our hypotheses.

Despite the fact that the wild knot $\B$ produced by Bing's construction depends on the choice of the integers $r_i$, it is customary 
in the literature to speak of {\it the} Bing sling%
\footnote{The first appearance of this terminology that I'm aware of is in R. H. Fox's list 
of problems in knot theory \cite{Fox2}*{p.~171}.} 
(rather than many different Bing slings) and we will keep with this tradition, in order to avoid peripheral details in statements of theorems.
But this forces us to include the condition of rapid growth of the $r_i$ in the definition of ``the Bing sling''.
Thus {\it the Bing sling} is defined for the purposes of this paper by conditions (0), (1), (2) and the condition that
each $r_{i+1}\ge (r_i^2+3r_i)^2$.

\begin{proof}[Proof of Theorem \ref{bingslinglemma}. (a)] 
Let us make the notation $D_{ij}$, $\Phi_{ij}$ more precise (so far they are defined only up to a permutation
of the second indices).
First we choose an arbitrary lift $D_1$ of $D$ in $T_1$.
To continue, we need the basepoint $b$ shown in Figure \ref{mazur-curve} (one of the three intersection points 
of $D$ with the fake Mazur curve $J$).
Assuming that a lift $D_i$ of $D$ in $T_i$ has been defined, it contains a unique lift $b_i$ of $b$ in $T_i$.
This $b_i$ lies in $K_i$ and hence also in $D_i\cap T_{i+1}$.
Then by (2) there is a unique lift $D_{i+1}$ of $D$ in $T_{i+1}$ which contains $b_i$.
Let us note that the sequence of disks defined in this way is nested: $D_1\supset D_2\supset\dots$.
Moreover, since by (1) their diameters tend to zero, $\bigcap_{i=1}^\infty D_i$ consists of a single point $b_\infty$, which is also the limit
of the sequence $b_1,b_2,\dots$.
Now let $D_{i0}=D_i$.
Assuming that a lift $D_{ij}$ of $D$ in $T_i$ has been defined, we define $D_{i,\,j+1}$ 
to be the next one in the positive direction (determined by the orientation of the fake Mazur curve $J$).
Thus $D_{ij}=D_{i,\,j+r_i}$, and using this formula we also define $D_{ij}$ for all $j<0$.
Also let $\Phi_{ij}$ be the fundamental domain bounded by $D_{ij}$ and $D_{i,\,j+1}$,
and let $b_{ij}$ be the unique lift of $b$ in $T_i$ that lies in $D_{ij}$.
(Thus $b_{i0}=b_i$.)

The fake Mazur curve can be parameterized by an embedding $k\:S^1\to T_1$, which sends $0\in S^1=\R/\Z$ onto $b$.
Let $\tilde k$ be the composition $\R\xr{q}\R/\Z=S^1\xr{k} T_1$.
We may assume that $\tilde k$ sends $\frac13$ and $\frac23$ onto the other two points of intersection of $J$ and $D$.
The embedding $k$, when precomposed with the $r_i$-fold cover $S^1\to S^1$, $s\mapsto r_is$, lifts to an embedding 
$k_i\:S^1\to T_i$ which parameterizes $K_i$. 
We may assume that $k_i(0)=b_i$.
Let $\tilde k_i$ be the composition $\R\xr{q}S^1\xr{k_i} T_1$.
Then $\tilde k_i\big(\frac{j}{r_i}\big)=b_{ij}$.
Moreover, it is clear from Figure ~\ref{mazur-curve} that
$\tilde k_i\big(\frac{j+\text{\sfrac13}}{r_i}\big)$ lies in the same disk $D_{ij}$ which contains $b_{ij}$,
whereas $\tilde k_i\big(\frac{j+\text{\sfrac23}}{r_i}\big)$ lies in the disk $D_{i,\,j+1}$ which contains 
$b_{i,\,j+1}=\tilde k_i\big(\frac{j+1}{r_i}\big)$.
Hence $\tilde k_i$ sends $\big[\frac{j+\text{\sfrac13}}{r_i},\,\frac{j+\text{\sfrac23}}{r_i}\big]$ onto an arc in $\Phi_{ij}$ with 
one endpoint in $D_{ij}$ and another in $D_{i,\,j+1}$.
Also, $\big[\frac{j}{r_i},\,\frac{j+\text{\sfrac13}}{r_i}\big]$ must go onto an arc in $\Phi_{i,\,j-1}$ with both endpoints in $D_{ij}$;
and $\big[\frac{j+\text{\sfrac23}}{r_i},\,\frac{j+1}{r_i}\big]$ onto an arc in $\Phi_{i,\,j+1}$ with both endpoints in $D_{i,\,j+1}$.
Thus $\tilde k_i$ sends each arc $J_{ij}:=\big[\frac{j}{r_i},\,\frac{j+1}{r_i}\big]$ onto an arc in the interior of
$\Psi_{ij}:=\Phi_{i,\,j-1}\cup\Phi_{ij}\cup\Phi_{i,\,j+1}$.

Next, we have $\frac{j}{r_i}=\frac{3jq_i}{r_{i+1}}$, and consequently
$J_{ij}=\big[\frac{j}{r_i},\,\frac{j+1}{r_i}\big]=
\big[\frac{3jq_i}{r_{i+1}},\,\frac{3(j+1)q_i}{r_{i+1}}\big]=
J_{i+1,\,3jq_i}\cup\dots\cup J_{i+1,\,3(j+1)q_i-1}$.
Hence $\tilde k_{i+1}(J_{ij})$ lies in $\Psi_{i+1,\,3jq_i}\cup\dots\cup\Psi_{i+1,\,3(j+1)q_i-1}$, which
is the same as $\Phi_{i+1,\,3jq_i-1}\cup\dots\cup\Phi_{i+1,\,3(j+1)q_i}$.
But it follows from (2) and from our choice of the lift $D_{0,\,i+1}$ that the latter union 
lies in $\Psi_{ij}$.%
\footnote{In more detail, $\Phi_{i+1,\,k}$ lies in $\Phi_{i,\,j-1}$ if $3jq_i\le k<(3j+1)q_i$;
in $\Phi_{i,\,j+1}$ if $(3j+2)q_i\le k<(3j+3)q_i$; and in $\Phi_{ij}$ if
$(3j+1)q_i\le k<(3j+2)q_i$ as well as for $k=3jq_i-1$ and $k=(3j+3)q_i$.}
Arguing by induction, it is easy to get from this that each $\tilde k_l$ for $l>i$ also sends $J_{ij}$ 
into $\Psi_{ij}$.
Hence $d(k_i,k_l)$ for $l>i$ is bounded above by the maximum of the diameters of $\Psi_{i1},\dots,\Psi_{ir_i}$, which by (1)
does not exceed $3\epsilon_i$.
Since $C^0(S^1,T_1)$ is complete, $k_1,k_2,\dots$ uniformly converge to a continuous map $\B\:S^1\to T_1$.

Let us show that $\B$ is injective (since $S^1$ is compact, this will imply that $\B$ is a homeomorphism onto its image).
First let us note that $r_i\ge 3^{i-1}$ by (2).
Therefore, given distinct points $x,y\in [0,1)$, there exists an $i$ such that $x\in J_{ij}$ and $y\in J_{ik}$ where 
$j-k\not\equiv 0,\pm1,\pm2,\pm 3\pmod{r_i}$.
Then $\Psi_{ij}$ and $\Psi_{ik}$ are disjoint, but $\tilde k_l(x)\in\tilde k_l(J_{ij})\subset\Psi_{ij}$ and 
$\tilde k_l(y)\in\tilde k_l(J_{ik})\subset\Psi_{ik}$ for each $l\ge i$.
Writing $\bar x=q(x)\in S^1$, we get $\B(\bar x)\in\Psi_{ij}$ and $\B(\bar y)\in\Psi_{ik}$.
Thus $\B(\bar x)\ne\B(\bar y)$.

Let us show that the image of $\B$ equals $\bigcap_{i=1}^\infty T_i$.
The images of $k_i,k_{i+1},\dots$ lie in $T_i$, and hence so does that of $\B$.
Therefore the image of $\B$ lies in $\bigcap_{i=1}^\infty T_i$.
Let us prove the converse inclusion.
Since each $T_i$ is a union of the fundamental domains $\Phi_{ij}$, every point $y\in\bigcap_{i=1}^\infty T_i$ 
is the intersection of some sequence of the form $\Phi_{1j_1}\supset\Phi_{2j_2}\supset\dots$.
Let us pick some points $x_i\in J_{ij_i}$.
Then $\tilde k_l(x_i)\in\Psi_{ij_i}$ for each $l\ge i$ and hence $\B(\bar x_i)\in\Psi_{ij_i}$ for each $i$.
Also $y\in\bigcap_{i=1}^\infty\Psi_{ij_i}$, and it follows from (1) that $\B(\bar x_i)\to y$.
But the image of $\B$ is compact (it is homeomorphic to $S^1$) and hence closed.
Thus it contains $y$.
\end{proof}

\begin{proof}[(b)] We assume familiarity with the proof of (a).
Let $\B$ and $\B'$ be two embeddings as in (a) corresponding to the same sequence $r_1,r_2,\dots$.
Let $T_i$, $p_i$, $\Phi_{ij}$ and $\Psi_{ij}$ be as above, and $T'_i$, $p'_i$, $\Phi'_{ij}$ and $\Psi'_{ij}$
be defined similarly.
Let us note that $T_1'=T_1$ by construction.
Since both $p_1$ and $p'_1$ are untwisted $r_1$-fold covers, there is a homeomorphism $h_1\:T_1\to T_1$ 
which sends each $\Phi'_{1j}$ onto $\Phi_{1j}$ and extends to an orientation preserving PL homeomorphism 
$S^3\to S^3$.
Next, it follows from (2), using that both $p_2$ and $p'_2$ are untwisted $r_2$-fold covers,
that there is a homeomorphism $h_2\:T_1\to T_1$ which keeps $\partial T_1$ fixed, preserves each $\Phi_{1j}$ 
and sends each $h_1(\Phi'_{2j})$ onto $\Phi_{2j}$.
Similarly, there is a homeomorphism $h_3\:T_1\to T_1$ which keeps $\overline{T_1\but T_2}$ fixed, preserves
each $\Phi_{2j}$ and sends each $h_2\circ h_1(\Phi'_{3j})$ onto $\Phi_{3j}$.
Continuing in the same fashion, we obtain homeomorphisms $h_4,h_5,\dots\:T_1\to T_1$.
Let $H_k=h_k\circ\dots\circ h_1$.
Then $H_k(\Phi'_{ij})=\Phi_{ij}$ for all $i\le k$ and $H_{k+1}$ disagrees with $H_k$ only
on the interior of $T'_k$.
Also $h_l\circ\dots\circ h_{k+1}$ preserves each $\Phi_{kj}$ for $l>k$, and hence by (1) 
$H_l$ is $\eps_k$-close to $H_k$ for all $l>k$.
Since $C^0(T_1,T_1)$ is complete, $H_1,H_2,\dots$ uniformly converge to a continuous map $H\:T_1\to T_1$.

For each $k$ the restriction of $H$ to $\overline{T_1\but T'_k}$ is a finite composition of homeomorphisms, 
hence a homeomorphism onto $\overline{T_1\but T_k}$.
Therefore $H$ restricts to a homeomorphism between $T_1\but\B'(S^1)$ and $T_1\but\B(S^1)$.
Also we have $H(\Phi'_{ij})=\Phi_{ij}$ for all $i$ and $j$.
It follows that $H(\Psi'_{ij})=\Psi_{ij}$ for all $i$ and $j$.
On the other hand, since each $r_i\ge 3^{i-1}$, every point $x\in [0,1)$ is the intersection of 
some sequence of the form $J_{1,j_1}\supset J_{2,j_2}\supset\dots$.
Then $\B(\bar x)$ lies in the intersection of the sequence $\Psi_{1,j_1}\supset\Psi_{1,j_2}\supset\dots$
and in fact equals this intersection by (1).
Similarly $\B'(\bar x)=\bigcap_{i=1}^\infty\Psi'_{ij_i}$.
Hence $H$ sends $\B'(\bar x)$ onto $\B(\bar x)$. 
Thus $H\circ\B'=\B$.
In particular, this implies that $H$ restricts to an injective map $\B'(S^1)\to\B(S^1)$.
Since it also restricts to an injective map $T_1\but\B'(S^1)\to T_1\but\B(S^1)$, we get
that $H$ is injective.
Since $T_1$ is compact, $H$ is a homeomorphism.
Since $h_1$ extends to an orientation preserving PL homeomorphism of $S^3$ and each
$h_i$ for $i>1$ keeps $\partial T_1$ fixed, $H$ also extends to an orientation preserving 
homeomorphism $\bar H\:S^3\to S^3$, which is PL outside $T_1$.
Then $\bar H$ is isotopic to the identity by the Alexander trick.
\end{proof}

\begin{proof}[(c)] 
Let us fix some $\eps>0$.
By scaling we may assume that the length of the unknot $K_0$ equals $1$.
Then the untwisted $r_1$-fold cover $p_1\:T_1\to T_1$ may be chosen so that each $\Phi_{1j}\cap K_0$ 
is an arc of length $\frac{1}{r_1}$.
Hence, as long as the regular neighborhood $T_1$ of $K_0$ is chosen to be thin enough, 
the diameter of each $\Phi_{1j}$ will be bounded above by $\frac{1+\epsilon}{r_1}$.

On the other hand, if the regular neighborhood $T_1$ of $K_0$ is chosen to be thin enough and 
the fake Mazur curve $J$ is drawn appropriately in $T_1$ (with a sufficiently small clasp), 
then we may assume that its parameterization $k\:S^1\to J\subset T_1$ (see the proof of (a)) 
is $\epsilon$-close to the map $S^1\xr{f}S^1=K_0\subset T_1$, where $f$ is defined as follows.
Let $\tilde f\:[0,1]\to\R$ be defined by $f(0)=0$, $f(\frac16)=-\frac12$, $f(\frac56)=\frac32$,
$f(1)=1$ and extending linearly.
(Thus $f(\frac13)=0$, $f(\frac12)=\frac12$ and $f(\frac23)=1$.)
Then the composition $[0,1]\xr{\tilde f}\R\xr{q}\R/\Z=S^1$ factors as 
$[0,1]\subset\R\xr{q}S^1\xr{f}S^1$ for a unique map $f$.
The identification $S^1=K_0$ is chosen so that $f(0)=b$ (the point $b$ is shown in 
Figure \ref{mazur-curve}).
Clearly, the length of the loop $S^1\xr{f}S^1=K_0\subset T_1$ equals $3$.
While this does not formally imply any bound on the length of $J$, it follows that
$J$ {\it can} be drawn so that its length does not exceed $3+\epsilon$.

The map $S^1\xr{f}S^1=K_0\subset T_1$ precomposed with the $r_1$-fold cover $S^1\to S^1$, 
$s\mapsto r_1s$, lifts with respect to the cover $p_1$ to a map $f_1\:S^1\to K_0\subset T_1$.
We may assume that the parameterization $k_1\:S^1\to K_1\subset T_1$ of $K_1$ (see the proof of (a)) 
is $\epsilon$-close to this $f_1$.
The loop $f_1$ is of length $3$, and the three paths that each $\Phi_{1j}$ 
cuts out of this loop are of total length $\frac{3}{r_1}$.
(These three paths have equal lengths.)
It follows that $p_1$ can be chosen so that the three arcs in $\Phi_{1j}\cap K_1$ 
are of total length $\le\frac{3+\epsilon}{r_1}$.
Now the untwisted $r_2$-fold cover $p_2\:T_2\to T_1$ may certainly be chosen to satisfy
the case $i=1$ of condition (2), i.e.\ so that each $\Phi_{1j}\cap T_2$ has only three
components and each of them is a union of $q_1$ of $\Phi_{2,k}$'s.
Then we may as well choose $p_2$ so that each $\Phi_{2,k}\cap K_1$ is an arc of length
$\le\frac{3+\epsilon}{3q_1r_1}=\frac{3+\epsilon}{r_2}$.
Hence, as long as the regular neighborhood $T_2$ of $K_1$ is chosen to be thin enough, 
the diameter of each $\Phi_{2,k}$ will be bounded above by $\frac{(3+\epsilon)+\epsilon}{r_2}$.

Now suppose inductively that for some $i$ each $\Phi_{ij}\cap K_{i-1}$ is an arc of length
$\le\frac{(3+\epsilon)^{i-1}}{r_i}$.
The map $S^1\xr{f}S^1=K_0\subset T_1$ precomposed with the $r_i$-fold cover $S^1\to S^1$, 
$s\mapsto r_is$, lifts with respect to the untwisted $r_i$-fold cover $p_i\:T_i\to T_1$ to a map 
$f_i\:S^1\to K_{i-1}\subset T_i$.
Since the arc $\Phi_{ij}\cap K_{i-1}$ is of length $\le\frac{(3+\epsilon)^{i-1}}{r_i}$, 
the three paths that $\Phi_{ij}$ cuts out of $f_i$ are of total length 
$\le\frac{3(3+\epsilon)^{i-1}}{r_i}$.
It follows that $p_i$ can be chosen so that the three arcs in $\Phi_{ij}\cap K_i$ 
are of total length $\le\frac{3(3+\epsilon)^{i-1}+\epsilon}{r_i}\le\frac{(3+\epsilon)^i}{r_i}$.
Arguing as before, we get that $T_{i+1}$ and $p_{i+1}$ can be chosen so that each
$\Phi_{i+1,\,j}\cap K_i$ is an arc of length 
$\le\frac{(3+\epsilon)^i}{3q_ir_i}=\frac{(3+\epsilon)^i}{3r_{i+1}}$
and the diameter of each $\Phi_{i+1,\,j}$ is bounded above by $\frac{(3+\epsilon)^i+\epsilon}{r_{i+1}}$.

Now setting $\epsilon=0.5$, we have proved that there exists a sequence of $K_i$ and $T_i$ with each
$\Phi_{ij}$ of diameter $\le\frac{3.5^{i-1}+0.5}{r_i}\le\frac{4^{i-1}}{r_i}$.
Since $r_1\ge 1$ and by the hypothesis each $\frac{r_{i+1}}{3r_i}\ge 2$, each $r_i\ge 6^{i-1}$.
Then condition (1) holds with $\epsilon_i=(4/6)^{i-1}$.
\end{proof}

\section{Conway potential function}\label{conway}

For an $m$-component PL link $L$ in $S^3$ let $\Omega_L(x_1,\dots,x_m)$ denote its Conway potential function.
Several explicit constructions of $\Omega_L$ are known (see references in the introduction of \cite{M24-2}).
We have $\Omega_L(x_1,\dots,x_m)\in\Z[x_1^{\pm1},\dots,x_m^{\pm1}]$ for $m>1$. 
Also $\Omega_L(x,\dots,x)=\dfrac{\nabla_K(x-x^{-1})}{x-x^{-1}}$, where $\nabla_K(z)\in\Z[z]$ is called the Conway polynomial.
If the components of $L$ are colored in $n$ colors according to a coloring function $c\:\{1,\dots,m\}\to\{1,\dots,n\}$, 
then $\Omega_L(x_1,\dots,x_n)$ denotes $\Omega_L(x_{c(1)},\dots,x_{c(m)})$.

Let us recall Conway's identities I (involving strands of the same color) and 
II (involving strands of possibly distinct colors):%
\footnote{To be precise, (\ref{conwayI}) appears in the papers by Kauffman \cite{Kau} and 
Hartley \cite{Ha} (and in a preliminary form also in the paper by Alexander \cite{Al}) and differs in a sign 
from the identity announced by Conway \cite{Con}.}
\[\Omega_{\includegraphics[width=0.6cm]{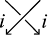}}-\Omega_{\includegraphics[width=0.6cm]{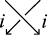}}
=(x_i-x_i^{-1})\,\Omega_{\includegraphics[width=0.6cm]{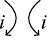}}
\tag{I}\label{conwayI}\]
\[\Omega_{\includegraphics[width=0.6cm]{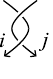}}+\Omega_{\includegraphics[width=0.6cm]{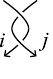}}
=(x_ix_j+x_i^{-1}x_j^{-1})\,\Omega_{\includegraphics[width=0.6cm]{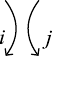}}
\tag{II}\label{conwayII}\]
We will also frequently use the formula%
\footnote{Here the component that is shown entirely need not be the only component of color $j$.
In fact, the case of the formula where it is the only component of color $j$ easily implies the general case.}
for a connected sum with the Hopf link:
\[\Omega_{\includegraphics[height=0.9cm]{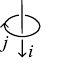}}
=(x_i-x_i^{-1})\cdot\Omega_{\ \includegraphics[height=0.9cm]{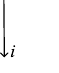}}
\tag{$\Phi$}\label{conwayPhi}\]
A basis for computation is given by%
\footnote{A link $L$ in $S^3$ is called {\it split} if there exists a PL $3$-ball in $S^3$ containing a nonempty 
collection of components of $L$ but not all of them, and disjoint from the remaining components of $L$.} 
\[\Omega_L=0\text{ for every split link $L$}\tag{S}\label{split}\] and by the value of $\Omega$ on the Hopf link:
\[\Omega_{\includegraphics[height=0.9cm]{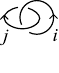}}=1.\tag{H}\label{hopf}\]
We refer to \cite{Ci0} or \cite{Ha} for the proofs of the relations 
(\ref{conwayI}), (\ref{conwayII}),  (\ref{conwayPhi}), (\ref{split}) and (\ref{hopf}).%
\footnote{Proofs of (\ref{conwayI}) and (\ref{conwayII}) from the definition of \cite{Ci0} can also be found in \cite{Coo2}*{\S7.9}.
The relation (\ref{split}) is proved only in a special case in \cite{Ci0}, but the general case is proved similarly.
The proof of (\ref{split}) from the definition of \cite{Ha}, which is omitted in \cite{Ha}, can be based on the fact that one
of the relations in the Wirtinger presentation of a link group is redundant (see \cite{Ha}*{(2.2)}),
and consequently for a suitable plane diagram of a split link there are two redundant relations.
The relation (\ref{conwayPhi}) is not proved in \cite{Ha} but can be easily verified by the methods of \cite{Ha}.}
In the case of two colors, it is easy to calculate $\Omega$ by using these five relations as axioms \cite{Na}*{pp.\ 166--168}.
The case of more than two colors is more complicated.%
\footnote{In the case of three colors $\Omega$ can be calculated by using these five axioms, Conway's third identity (which 
involves four terms, each with three strands) and the value of $\Omega$ on the Borromean rings \cite{Na}.
In the general case several axiomatic characterizations of $\Omega$ are known, but they are either non-local \cite{Tu0}
or contain a local identity that involves six or more terms \cite{Jia}, \cite{MuJ}.}

Let us note that in presence of (\ref{conwayPhi}), an equivalent form of (\ref{hopf}) is given by the value of $\Omega$ on the unknot:
\[\Omega_{\includegraphics[height=0.65cm]{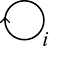}}=\frac{1}{x_i-x_i^{-1}}.\]
Also (\ref{hopf}), (\ref{conwayII}) and (\ref{split}) easily imply
\[\Omega_{\includegraphics[height=0.9cm]{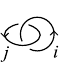}}=-1\tag{H$'$}\label{hopf'}\]
and similarly (\ref{conwayPhi}), (\ref{conwayII}) and (\ref{split}) imply
\[\Omega_{\includegraphics[height=0.9cm]{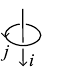}}
=-(x_i-x_i^{-1})\cdot\Omega_{\ \includegraphics[height=0.9cm]{Phip.pdf}}
\tag{$\Phi'$}\label{conwayPhi'}\]

\begin{example} \label{mazur-conway} 
Let us compute $\Omega_M$, where $M=(J,B)$ is the fake
Mazur link (see Figure \ref{mazur}).
\begin{figure}[h]
\includegraphics{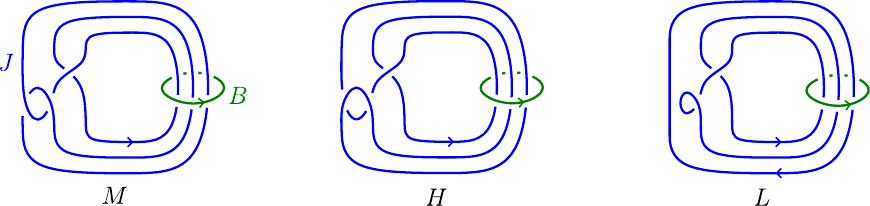}
\caption{The fake Mazur link $M$ and a computation of $\Omega_M$ (beginning).}
\label{mazur}
\end{figure}

$M$ can be made into the Hopf link $H$ by one crossing change (see Figure \ref{mazur}).
Then by (\ref{conwayI}) we have $\Omega_M-\Omega_H=-(x-x^{-1})\Omega_L$, where $L$ is the $3$-component link shown
in Figure~ \ref{mazur}.
Hence $\Omega_M=1-(x-x^{-1})\Omega_L$.
By (\ref{conwayPhi'}) we have $\Omega_L=-(y-y^{-1})\Omega_{H_2}$, where $H_2$ is shown in Figure \ref{mazur2}.
By (\ref{conwayII}) we have $\Omega_{H_2}+\Omega_U=(xy+x^{-1}y^{-1})\Omega_H$, where $U$ is the $2$-component unlink 
(see Figure \ref{mazur2}), whence $\Omega_{H_2}=xy+x^{-1}y^{-1}$. 
Thus we obtain \[\Omega_M(x,y)=1+(xy+x^{-1}y^{-1})(x-x^{-1})(y-y^{-1}).\]
\end{example}
\begin{figure}[h]
\includegraphics{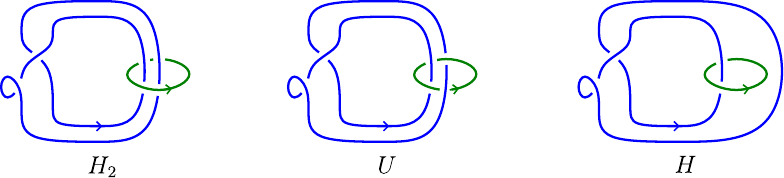}
\caption{The computation of $\Omega_M$ (conclusion).}
\label{mazur2}
\end{figure}

\begin{remark}\label{conway-sign}
Up to a sign, $\Omega_L$ is determined by the multi-variable Alexander polynomial $\Delta_L$.%
\footnote{Namely, $\Omega_L(x_1,\dots,x_m)=\pm x_1^{k_1}\dots x_m^{k_m}\Delta_L(x_1^2,\dots,x_m^2)$ for $m>1$ and
$(x-x^{-1})\Omega_K(x)=\pm x^k\Delta_K(x^2)$ for a knot $K$,
where the integers $k,k_1,\dots,k_m$ are uniquely determined by the symmetry relation
$\Omega_L(x_1,\dots,x_m)=(-1)^m\Omega_L(x_1^{-1},\dots,x_m^{-1})$.}
The sign can be recovered for knots (using Theorem \ref{alexander}); for $2$-component links it is determined by 
the sign of the linking number when it is nonzero (using Corollary \ref{lk}), and in general by the sign of any 
shortest non-vanishing $\bar\mu$-invariant of $L$ of the type $\bar\mu(1\dots 12\dots 2)$ \cite{Tr2}*{8.3}.
However no such description of the sign of $\Omega_L$ is known for links of $3$ or more components 
(compare \cite{Le2}*{1.2, 1.6}, \cite{Tr2}*{8.2}, \cite{TY}*{1.8}, \cite{KLW}).

This sign might seem to be a minor issue, but it is indispensable for a recursive computation of $\Omega_L$ 
(see Example \ref{mazur-conway}), and at the same time it is, in a sense, a much subtler invariant than 
the multi-variable Alexander polynomial.
Indeed, $\Delta_L$ is an invariant of the $\Z[\pi_L/\pi'_L]$-module $\pi'_L/\pi''_L$, where 
$\pi_L=\pi_1(S^3\but L)$ and $G'=[G,G]$.%
\footnote{Namely, $\Delta_L$ is the gcd of all $n\x n$ minors of a presentation matrix of this module, 
where $n$ is the number of generators.
Note that $\Z[\pi_L/\pi'_L]\simeq\Z[H_1(S^3\but L)]\simeq\Z[t_1^{\pm1},\dots,t_m^{\pm1}]$ and 
$\pi'_L/\pi''_L\simeq H_1(\widetilde{S^3\but L})$, 
where $\tilde X$ denotes the universal abelian cover of $X$.}
By contrast, the sign of $\Omega_L$ is not an invariant of $\pi_L$, and not even of $S^3\but L$, because it 
generally depends on the orientations of $S^3$ and of the components of $L$ (see \cite{Ha}*{5.6, 5.7}).
Nevertheless, the definition of $\Omega_L$ in terms of Seifert surfaces \cite{Ci0} deals (as explained 
in \cite{CiF}) with a certain presentation of the same module $G'/G''$; and the two definitions of $\Omega_L$ 
in terms of the Fox calculus \cite{Ha}, \cite{BC}*{\S2} deal with certain presentations of a closely 
related module.%
\footnote{Namely, $H_1(\widetilde{S^3\but L},\,p^{-1}(b))$, where $b\in S^3\but L$ and 
$p\:\widetilde{S^3\but L}\to S^3\but L$ is the covering map.
See \cite{Kaw} concerning the relation between the two modules.}
All three presentations have to be chosen very carefully, because interchanging any two relators would have
the effect of changing the sign of $\Omega_L$; and if adding a redundant relator $1$ is allowed,
then any two relators can be interchanged, as rows in the presentation matrix, by a repeated application 
of the operation of adding a multiple of a row to another row (see \cite{Za}*{p.\ 118}).
Therefore it is not surprising that the definition of $\Omega_L$ that appears to be the closest to being
``conceptual'' \cite{Tu0} proceeds by redoing everything in terms of determinants.
\end{remark}

\begin{problem} (a) Can $\Omega_L$ be constructed as an invariant of the $\Z[\pi_L/\pi'_L]$-module 
$\pi'_L/\pi''_L$ endowed with some additional structure?

(b) Is there a ``conceptual'' proof (as opposed to a verification of the axioms) that the two definitions 
of $\Omega_L$ in terms of the Fox calculus \cite{Ha}, \cite{BC}*{\S2} and the third one in terms of 
Seifert surfaces \cite{Ci0} are all equivalent to each other?
\end{problem}

\section{Some known tools for computing $\Omega$}

\begin{theorem}[Torres--Hartley \cite{Ha}] \label{torres}
Let $L=(K_1,\dots,K_m,K)$ be a PL link, where $m\ge 1$, and let $l_i=\lk(K,K_i)$. 
Let $L'=(K_1,\dots,K_m)$.
Then \[(x_1^{l_1}\cdots x_m^{l_m}-x_1^{-l_1}\cdots x_m^{-l_m})\,\Omega_{L'}(x_1,\dots,x_m)=
\Omega_L(x_1,\dots,x_m,1).\]
In particular, if all $l_i=0$, then $\Omega_L(x_1\dots,x_m,1)=0$; and if $m=1$ and $l_1=1$, 
then $\Omega_L(x,1)=\nabla_{K_1}(x-x^{-1})$.
\end{theorem}

In the case $m=0$ we have

\begin{theorem}[Alexander \cite{Al}] \label{alexander} If $K$ is a PL knot, then $\nabla_K(0)=1$.
\end{theorem}

\begin{proof} (\ref{conwayI}) implies that $\nabla_K(0)$ is the same for all knots $K$.
By (\ref{hopf}) and (\ref{conwayPhi}) $\nabla_{\includegraphics[height=0.65cm]{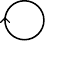}}=1$.
\end{proof}

\begin{example}\label{fibonacci-lucas}
Let $L=(K_1,K_2)$ be a $2$-component link and $n=\lk(L)$.
By Theorem~ \ref{torres} \[\Omega_L(x,1)=
\frac{x^n-x^{-n}}{x-x^{-1}}\nabla_{K_1}(x-x^{-1}).\]
Here \[x^n-x^{-n}=\begin{cases}
(x+x^{-1})F_n(x-x^{-1})&\text{if $n$ is even,}\\
L_n(x-x^{-1})&\text{if $n$ is odd,}
\end{cases}\]
where $F_n(z)$ is the $n$th Fibonacci polynomial, defined for $n\ge 0$ by $F_0(z)=0$, $F_1(z)=1$ and 
$F_{n+1}(z)=zF_n(z)+F_{n-1}(z)$, and $L_n(z)$ is the $n$th Lucas polynomial, defined for $n\ge 0$ by 
the same recursion $L_{n+1}(z)=zL_n(z)+L_{n-1}(z)$ with different initial values $L_0(z)=2$, $L_1(z)=z$.
The definitions are extended for negative indexes by $F_{-n}(z)=(-1)^{n+1}F_n(z)$ and $L_{-n}(z)=(-1)^nL_n(z)$.
\end{example}

\begin{corollary} \label{lk} If $L$ is a $2$-component link, then $\Omega(1,1)=\lk(L)$.
\end{corollary}

\begin{proof} 
$\dfrac{x^n-x^{-n}}{x-x^{-1}}=x^{-(n-1)}+x^{-(n-3)}+\dots+x^{n-1}$, which contains $n=\lk(L)$ terms.
\end{proof}

Let $L=(K_1,\dots,K_m,K)$ and $L'=(Q,Q_1,\dots,Q_n)$ be PL links in $S^3$, where $Q$ is unknotted.
Let $N$ and $N'$ be regular neighborhoods of $K$ and $Q$ in the complements to the other components,
and let $E'$ be the solid torus $\overline{S^3\but N'}$.
Let $h\:E'\to N$ be a homeomorphism that sends a longitude of $E'$ (which is a meridian of $N'$) 
onto a longitude of $N$.
Then $\Lambda:=\big(K_1,\dots,K_m,h(Q_1),\dots,h(Q_n)\big)$ is called the {\it splice} of $L$ and $L'$.

\begin{theorem}[Cimasoni \cite{Ci}] \label{splice} Let $\Lambda$ be the splice of 
$L=(K_1,\dots,K_m,K)$ and $L'=(Q,Q_1,\dots,Q_n)$, and let $\mu_i=\lk(K,K_i)$ and $\nu_i=\lk(Q,Q_i)$.
Assume that $n\ge 1$, and either $m\ge 1$ or at least one of the $\nu_i$ is nonzero.
Then \[\Omega_\Lambda(x_1,\dots,x_m,y_1,\dots,y_n)=
\Omega_L(x_1,\dots,x_m,y_1^{\nu_1}\cdots y_n^{\nu_n})\,\Omega_{L'}(x_1^{\mu_1}\cdots x_m^{\mu_m},y_1,\dots,y_n).\]
\end{theorem}

Theorem \ref{splice} can be used to reprove the relation (\ref{conwayPhi}), by representing the connected sum 
of a link $L$ with the Hopf link as a splice of $L$ with a connected sum of two Hopf links (whose $\Omega$
can be computed directly). 
In turn, (\ref{conwayPhi}) and Theorem \ref{splice} easily yield a formula for an arbitrary connected sum:

\begin{corollary}[Cimasoni \cite{Ci}, \cite{Ci0}] \label{connected sum}
Let $\Lambda=(K_1,\dots,K_m,\,K\#Q,\,Q_1,\dots,Q_n)$ be the connected sum of PL links $L=(K_1,\dots,K_m,K)$ and $L'=(Q,Q_1,\dots,Q_n)$,
where $L'$ is disjoint from some $3$-ball containing $L$, along an arbitrary band joining $K$ and $Q$ and otherwise disjoint from $L$ and $L'$.
Then \[\Omega_\Lambda(x_1,\dots,x_m,y,y_1,\dots,y_n)=(y-y^{-1})\Omega_L(x_1,\dots,x_m,y)\,\Omega_{L'}(y,y_1,\dots,y_n).\]
In particular, when $n=0$, we have 
$\Omega_\Lambda(x_1,\dots,x_n,y)=\Omega_L(x_1,\dots,x_n,y)\,\nabla_Q(y-y^{-1})$ and $\nabla_\Lambda(z)=\nabla_L(z)\,\nabla_Q(z)$.
\end{corollary}

\begin{proof} Let $L''=(B,Q,Q_1,\dots,Q_m)$, where $B$ is the boundary of a small embedded disk intersecting $Q$ transversely 
in one interior point with positive sign (that is, so that $\lk(B,Q)=+1$) and otherwise disjoint from $L'$.
By (\ref{conwayPhi}) $\Omega_{L''}(x,y,y_1,\dots,y_n,y)=(y-y^{-1})\Omega_{L'}(y,y_1,\dots,y_n)$.
On the other hand, $\Lambda$ is easily seen to be the splice of $L$ and $L''$, and the assertion follows from Theorem \ref{splice}.
\end{proof}

Theorem \ref{splice} can also be used to obtain a formula for adding several push-offs of a component to the link 
(compare related formulas in \cite{Tu0}, \cite{Ci}):

\begin{corollary} \label{pushoff}
Let $L=(K_1,\dots,K_m,K)$ be a PL link and let $l_i=\lk(K,K_i)$. 
Let $L_+=(K_1,\dots,K_m,K_+^1,\dots,K_+^n)$, where $n\ge 1$ and $K_+^1,\dots,K_+^n$ are pairwise disjoint 
parallel pushoffs of $K$.
Then \[\Omega_{L_+}(x_1,\dots,x_m,y_1,\dots,y_n)=(x_1^{l_1}\cdots x_m^{l_m}-x_1^{-l_1}\cdots x_m^{-l_m})^{n-1}\,
\Omega_L(x_1,\dots,x_m,y_1\cdots y_n).\]
\end{corollary}

\begin{proof} Let $H=(C,B)$ be the Hopf link, and let $H_+=(C,B_+^1,\dots,B_+^n)$, where $B_+^1,\dots,B_+^n$ 
are pairwise disjoint parallel pushoffs of $B$.
Then $L_+$ is the splice of $L$ and $H_+$, and $\Omega_{H_+}(u,v_1,\dots,v_n)=(u-u^{-1})^{n-1}$ by using (\ref{conwayPhi}).
\end{proof}

Let us note that for $n=0$ Corollary \ref{pushoff} does not make sense as stated (because of a potential division by $0$),
but this case is addressed by Theorem \ref{torres}.

\begin{proposition}[Hartley \cite{Ha}]\label{orientation}
If $L'$ is obtained from $L$ by reversing the orientation of the first component, then
$\Omega_{L'}(x_1,x_2,\dots,x_n)=-\Omega_L(x_1^{-1},\dots,x_n)$.
\end{proposition}

Corollary \ref{pushoff}, Theorem \ref{torres} and Proposition \ref{orientation} imply

\begin{corollary} \label{annulus}
Let $L=(K_1,\dots,K_m,K)$ be a PL link, let $L'=(K_1,\dots,K_m,K,-K_+)$, where $-K_+$ is a parallel
pushoff of $K$ with orientation reversed, and let $l_i=\lk(K,K_i)$. 
Then \[\Omega_{L'}(x_1,\dots,x_m,y,y)=-(x_1^{l_1}\cdots x_m^{l_m}-x_1^{-l_1}\cdots x_m^{-l_m})^2\,\Omega_{L\but K}(x_1,\dots,x_m).\]
\end{corollary}

Corollary \ref{annulus} and Conway's identity \ref{conwayI} imply

\begin{corollary} \label{whitehead}
Let $L=(K_1,\dots,K_m,K)$ be a PL link, let $L'=(K_1,\dots,K_m,W)$ be its possibly twisted Whitehead
doubling along the last component, and let $l_i=\lk(K,K_i)$.
Then \[\Omega_{L'}(x_1,\dots,x_m,y)=\pm(y-y^{-1})(x_1^{l_1}\cdots x_m^{l_m}-x_1^{-l_1}\cdots x_m^{-l_m})^2\,\Omega_{L\but K}(x_1,\dots,x_m).\]
\end{corollary}

\begin{theorem}[Murasugi--$\dots$--Miyazawa \cite{Mi}]\label{miyazawa} Let $(B,K_1,\dots,K_m)$ be an $(m+1)$-component PL link, 
where $B$ is the unknot, and let $\tilde K_j$ denote the link $p_r^{-1}(K_j)$,
where $p_r\:(S^3,B)\to (S^3,B)$ is the $r$-fold cover ramified over $B$.
Then \[\Omega_{(\tilde K_1,\dots,\tilde K_m)}(x_1,\dots,x_m)=i^{(r-1)(1-\lambda)}
\Omega_{(K_1,\dots,K_m)}(x_1,\dots,x_m)\prod_{j=1}^{r-1}\Omega_{(B,K_1,\dots,K_m)}(\xi^j,x_1,\dots,x_m),\]
where $i=\sqrt{-1}$, $\xi=e^{\pi i/r}$ and $\lambda=\lk(B,K_1)+\dots+\lk(B,K_m)$.
\end{theorem}

\begin{corollary} \label{periodic}
Let $M=(J,B)$ be a $2$-component link with linking number $1$ and let
$p_r\:(S^3,B)\to (S^3,B)$ be the $r$-fold cover ramified over $B$.
Let $J_r=p_r^{-1}(J)$ and $M_r=(J_r,B)$.
Let $\xi=e^{\pi i/r}$.
Then

(a) $\Omega_{M_r}(x,y)=\prod_{j=0}^{r-1}\Omega_M(x,y^{\frac1r}\xi^j)$;

(b) $\Omega_{J_r}(x)=\frac{1}{x-x^{-1}}\prod_{j=0}^{r-1}\Omega_M(x,\xi^j)$.
\end{corollary}

\begin{proof}[Proof. (a)]
Let $M_+=(J,B,B_+)$ be obtained from $M$ by adding a parallel pushoff $B_+$ of $B$.
By Corollary \ref{pushoff} $\Omega_{M_+}(x,y,z)=(x-x^{-1})\Omega_M(x,yz)$.
Let $M_r^+$ be the $2$-colored link $(J_r,B_r)$, where $B_r=p_r^{-1}(B_+)$.
Then by Theorem \ref{miyazawa}
$\Omega_{M_r^+}(x,y)=\Omega_M(x,y)\prod_{j=1}^{r-1}\Omega_{M_+}(x,\xi^j,y)$.
Hence $\Omega_{M_r^+}(x,y)=(x-x^{-1})^{r-1}\prod_{j=0}^{r-1}\Omega_M(x,y\xi^j)$. 
But clearly $B_r=(B_+^1,\dots,B_+^r)$, the link of $r$ parallel pushoffs of $B$.
Hence by Corollary \ref{pushoff} $\Omega_{M_r^+}(x,y)=(x-x^{-1})^{r-1}\Omega_{M_r}(x,y^r)$.
Thus we get that $\Omega_{M_r}(x,y)=\prod_{j=0}^{r-1}\Omega_M(x,y^{\frac1r}\xi^j)$.
\end{proof}

\begin{proof}[(b)] Although (b) follows from (a) and Theorem \ref{torres}, it is easier to prove it
by observing that it is a direct consequence of Theorems \ref{miyazawa} and \ref{torres}.
\end{proof}

\begin{remark} Explicit computations may be aided by the following congruence \cite{Ms2}*{Proposition 4.2}.
If $P(x,y)\in\Z[x^{\pm1},y^{\pm1}]$%
\footnote{$P$ has no negative exponents in \cite{Ms2}, but it is easy to see that this hypothesis is redundant.}
and $r=p^nm$, where $p$ is a prime, $n>0$ and $(m,p)=1$, then
$\prod_{j=0}^{r-1}P(x,e^{2\pi j/r})\equiv\Big(\prod_{k=0}^{m-1}P(x,e^{2\pi k/m})\Big)^{p^n}\pmod p$.
\end{remark}

\section{Computing $\Omega$ for ramified covers of the fake Mazur link} \label{geometry}

The following notation will be used throughout this section.
Let $M=(J,B)$ be the fake Mazur link (see Figure \ref{mazur}) and let $p_r\:(S^3,B)\to (S^3,B)$ be the $r$-fold cover ramified over $B$.
Let $J_r=p_r^{-1}(J)$ and $M_r=(J_r,B)$.

\begin{lemma} \label{conway-MJ} Let $\zeta=e^{2i\pi/r}$.
Then

(a) $\Omega_{M_r}(x,y)=\prod_{j=0}^{r-1}\Big(1+(xy^{\frac2r}\zeta^j-x^{-1}y^{-\frac2r}\zeta^{-j})(x-x^{-1})-(x-x^{-1})^2\Big)$;

(b) $\Omega_{J_r}(x)=\frac{1}{x-x^{-1}}\prod_{j=0}^{r-1}\Big(1+(x\zeta^j-x^{-1}\zeta^{-j})(x-x^{-1})-(x-x^{-1})^2\Big)$.
\end{lemma}

\begin{proof} By Example \ref{mazur-conway} 
\[\Omega_M(x,y)=1+(xy+x^{-1}y^{-1})(x-x^{-1})(y-y^{-1})=1+(xy^2-x^{-1}y^{-2})(x-x^{-1})-(x-x^{-1})^2.\]
Now the assertion follows from Corollary \ref{periodic}.
\end{proof}

\begin{proposition}\label{conway-MJ'} 
(a) $\nabla_{J_r}(z)=r\prod_{k=1}^{r-1}\Big(\frac1{\zeta^{2k}-\zeta^k}-z^2\Big)$, where $\zeta=e^{2\pi i/r}$.
\medskip

(b) $\Omega_{M_r}(x,y)=\nabla_{J_r}(x-x^{-1})+(x-x^{-1})^r(y-y^{-1})\big((-1)^{r+1}x^ry+x^{-r}y^{-1}\big)$.
\medskip

(c) $x^ry+(-1)^{r+1}x^{-r}y^{-1}=(y-y^{-1})F_{r+1}(x-x^{-1})+(x^{-1}y+xy^{-1})F_r(x-x^{-1})$.
\medskip

(d) $\nabla_{M_r}(z)=z\nabla_{J_r}(z)+(-1)^{r+1}z^{r+2}L_{r+1}(z)$.
\end{proposition}

Here $F_r$ and $L_r$ are the Fibonacci and Lucas polynomials (see Example \ref{fibonacci-lucas}).

\begin{proof}[Proof. (a)] We start from the expression of Lemma \ref{conway-MJ}(b) and multiply out the factors corresponding 
to $j$ and to $r-j$, for $j=1,\dots,\lfloor\frac{r-1}2\rfloor$.
There remain: 
\begin{itemize}
\item an additional factor corresponding to $j=0$, but it equals 1; and
\item an additional factor corresponding to $\frac{r}{2}$ when $r$ is even.
\end{itemize}
The result is:
\[\nabla_{J_r}(z)=
\begin{cases}\phantom{(1-2z^2)}\prod_{k=1}^n\big(1+2z^2(1+c_k-2c_k^2)+2z^4(1-c_k)\big)&\text{ if }r=2n+1;\\
(1-2z^2)\prod_{k=1}^{n-1}\big(1+2z^2(1+c_k-2c_k^2)+2z^4(1-c_k)\big)&\text{ if }r=2n,
\end{cases}\]
where $c_k=\cos\frac{2\pi k}{r}$. 
Using the trigonometric identities $\cos(2\phi)=1-2\sin^2\phi$
and $\sin(3\phi)=3\sin\phi-4\sin^3\phi$ this formula can be simplified as follows:
\[\nabla_{J_r}(z)=
\begin{cases}\phantom{(1-2z^2)}\prod_{k=1}^n\big(1+4\sin(k\alpha)\sin(3k\alpha)z^2 +4\sin^2(k\alpha)z^4\big)&\text{ if }r=2n+1;\\
(1-2z^2)\prod_{k=1}^{n-1}\big(1+4\sin(k\alpha)\sin(3k\alpha)z^2 +4\sin^2(k\alpha)z^4\big)&\text{ if }r=2n,
\end{cases}\]
where $\alpha=\pi/r$.
Let $x=2\sin(k\alpha)z^2$, $\lambda=\sin(3k\alpha)$.
Then $1+2\lambda x+x^2=(r_+-x)(r_-+x)$, where $r_\pm=-\lambda\pm\sqrt{\lambda^2-1}=-\sin(3k\alpha)\pm i\cos(3k\alpha)=
i\big(i\sin(3k\alpha)\pm\cos(3k\alpha)\big)$.
Thus $r_+=ie^{3k\alpha i}$ and $r_-=ie^{(\pi-3k\alpha)i}=ie^{3(\pi-k\alpha)i}=ie^{3(r-k)\alpha i}$.
On the other hand, $\sin(k\alpha)=\sin(\pi-k\alpha)=\sin\big((r-k)\alpha\big)$.
It follows that
\[\nabla_{J_r}(z)=\prod_{k=1}^{r-1}\big(ie^{3k\alpha i}-2\sin(k\alpha)z^2\big).\]
Now using the trigonometric identity%
\footnote{Let $\zeta=e^{2\pi i/r}$ and $\xi=e^{\pi i/r}$.
Then
$\prod_{k=1}^{r-1}\sin(k\alpha)=\prod_{k=1}^{r-1}\frac{\xi^k-\xi^{-k}}{2i}
=(2i)^{1-r}\xi^{r(r-1)/2}\prod_{k=1}^{r-1}(1-\xi^{-2k})=(2i)^{1-r}e^{(r-1)\pi i/2}\prod_{k=1}^{r-1}(1-\zeta^{-k})
=2^{1-r}\prod_{k=1}^{r-1}(1-\zeta^k)$.
On the other hand, by dividing both sides of $x^r-1=\prod_{k=0}^{r-1}(x-\zeta^k)$ by $x-1$ we get
$1+x+\dots+x^{r-1}=\prod_{k=1}^{r-1}(x-\zeta^k)$, and by substituting $x=1$ we get
$r=\prod_{k=1}^{r-1}(1-\zeta^k)$.}
$\prod_{k=1}^{r-1}\sin(k\alpha)=\frac{r}{2^{r-1}}$ we get
\[\nabla_{J_r}(z)=r\prod_{k=1}^{r-1}\Big(\frac{ie^{3k\alpha i}}{2\sin(k\alpha)}-z^2\Big).\]
Finally, $\frac{ie^{3k\alpha i}}{2\sin(k\alpha)}=\frac{-e^{3k\alpha i}}{e^{k\alpha i}-e^{-k\alpha i}}=\frac1{e^{-4k\alpha i}-e^{-2k\alpha i}}$.
We have $e^{-2k\alpha i}=e^{(2\pi-2k\alpha)i}=e^{2(r-k)\alpha i}$ and similarly $e^{-2k\alpha i}=e^{4(r-k)\alpha i}$.
Hence 
\[\nabla_{J_r}(z)=r\prod_{k=1}^{r-1}\Big(\frac1{e^{4k\alpha i}-e^{2k\alpha i}}-z^2\Big).\]
\end{proof}

\begin{proof}[(b)] Since $\lk(M_r)=1$, the exponent of $y$ in every term of $\Omega_{M_r}(x,y)$
must be even \cite{M24-2}*{Lemma \ref{part2:4.1}}. 
Each of the $r$ factors of the product in Lemma \ref{conway-MJ}(a) is a sum of monomials of three types: 
$cx^k$, $cx^ky^{\frac2r}$ and $cx^ky^{-\frac2r}$, where $k\in\Z$ and $c\in\C$. 
A product of $r$ such monomials $m_0\cdots m_{r-1}$ can have a nonzero even integer exponent of $y$ 
only when all $m_i$ are of the second type or all $m_i$ are of the third type.
Since $\zeta^0\zeta^1\cdots\zeta^{r-1}=\zeta^{r(r-1)/2}=(-1)^{r+1}$, the sum of all such products is
$(-1)^{r+1}x^ry^2(x-x^{-1})^r$ in the first case and $-x^{-r}y^{-2}(x-x^{-1})^r$ 
in the second case.
Thus $\Omega_{M_r}(x,y)=\big((-1)^{r+1}x^ry^2-x^{-r}y^{-2}\big)(x-x^{-1})^r+L(x)$ for some $L(x)\in\Z[x^{\pm1}]$.
To eliminate $L(x)$ let us note that $\Omega_{M_r}(x,1)=\big((-1)^{r+1}x^r-x^{-r}\big)(x-x^{-1})^r+L(x)$.
Hence we get
$\Omega_{M_r}(x,y)-\Omega_{M_r}(x,1)=(x-x^{-1})^r\big((-1)^{r+1}x^r(y^2-1)+x^{-r}(1-y^{-2})\big)=
(x-x^{-1})^r(y-y^{-1})((-1)^{r+1}x^ry+x^{-r}y^{-1})$.
Now the assertion follows from Theorem \ref{torres}.
\end{proof}

\begin{proof}[(c)]
This follows from the formula $\{x^nM\}=F_n(\{x\})\{xM\}+F_{n-1}(\{x\})\{M\}$, which in turn follows from \cite{M24-2}*{Lemma \ref{part2:identities}(a)}.
\end{proof}

\begin{proof}[(d)] This follows from (b) and (c) using that $L_{r+1}(z)=zF_{r+1}(z)+2F_r(z)$.
\end{proof}

\begin{corollary} \label{conway-MJ-degrees}
$\nabla_{M_r}\ne z\nabla_{J_r}$.
\end{corollary}

\begin{proof} It is easy to see from Proposition \ref{conway-MJ'} that $\nabla_{M_r}(z)$ 
is a polynomial of degree $2r+3$, and $\nabla_{J_r}(z)$ is a polynomial of degree $2r-2$.
\end{proof}

Proposition \ref{conway-MJ'}(a) gives a description of the polynomial $\nabla_{J_r}$ in terms of its roots.
But we are more interested in its coefficients, which do not seem to be easily obtainable from
this description (except for specific low values of $r$).
To get closer to understanding the coefficients of $\nabla_{J_r}$ we will take a different approach.

\begin{proposition}
\[\sum_{r=1}^\infty\nabla_{J_r}(z)x^r=\frac{x}{1+z^2(x-x^2)}\Big(\frac{x}{1-x}+\frac1{1+z^2x}\Big).\]
\end{proposition}

This result is not used in the present paper, but it might be useful for further attacks on the Rolfsen Problem.

\begin{figure}[h]
\includegraphics[width=\linewidth]{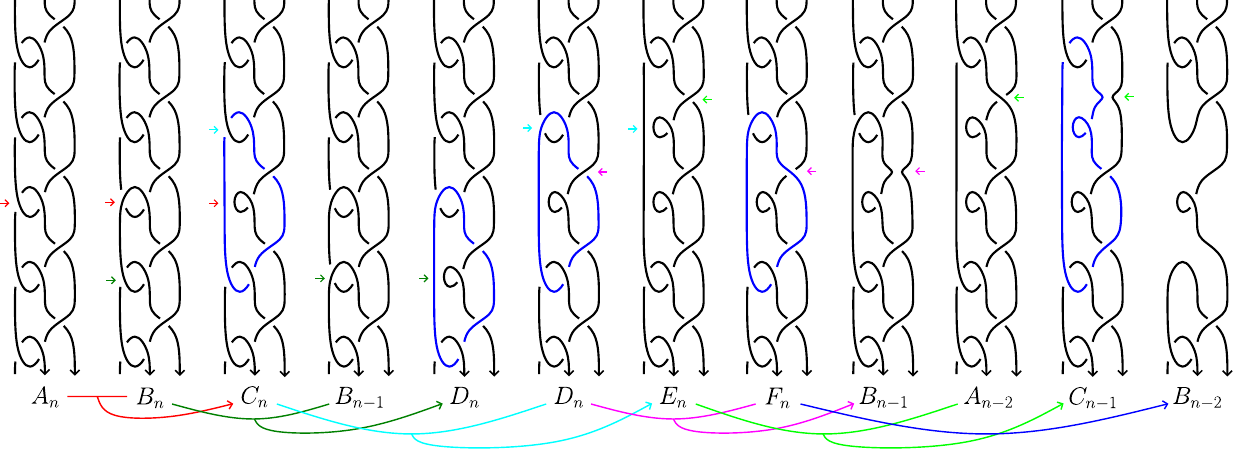}
\caption{Recursion for $\nabla_{A_n}$.}
\label{recursion}
\end{figure}
\begin{proof} Let $A_n=J_n$ and let $B_n$, $C_n$, $D_n$, $E_n$ and $F_n$ be its modifications shown in Figure~ \ref{recursion}.
Here $A_n$ involves $n$ repeats of the Fox--Artin tangle:
\begin{center}
\includegraphics{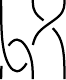}
\end{center}
and so is defined for $n\ge 0$; $B_n$ and $C_n$ each involve $n-1$ repeats of the same tangle, and so are defined for $n\ge 1$; and
$D_n$, $E_n$ and $F_n$ each involve $n-2$ repeats of the same tangle, and so are defined for $n\ge 2$.
From Conway's first relation (\ref{conwayI}) and the relation (\ref{conwayPhi}) we obtain
\[\begin{aligned}
\nabla_{A_n}&=\nabla_{B_n}-z\nabla_{C_n}, &n\ge 1\\
\nabla_{B_n}&=\nabla_{B_{n-1}}-z\nabla_{D_n}, &n\ge 2\\
\nabla_{C_n}&=\nabla_{D_n}-z\nabla_{E_n}, &n\ge 2\\
\nabla_{D_n}&=\nabla_{F_n}+z\nabla_{B_{n-1}}, &n\ge 2\\
\nabla_{E_n}&=\nabla_{A_{n-2}}+z\nabla_{C_{n-1}}, &n\ge 2\\\
\nabla_{F_n}&=-z\nabla_{B_{n-2}}, &n\ge 3.\!\!
\end{aligned}\]
It is easy to see from Figure \ref{recursion-start} that $B_1$ is the unknot (whence $\nabla_{B_1}=1$), 
$C_1$ is the two-component unlink (whence $\nabla_{C_2}=0$) and each of $C_2$ and $D_2$ is the Hopf link (whence $\nabla_{C_2}=z$ and
$\nabla_{B_2}=\nabla_{B_1}-z\nabla_{D_2}=1-z^2$).
\begin{figure}[h]
\includegraphics[width=\linewidth]{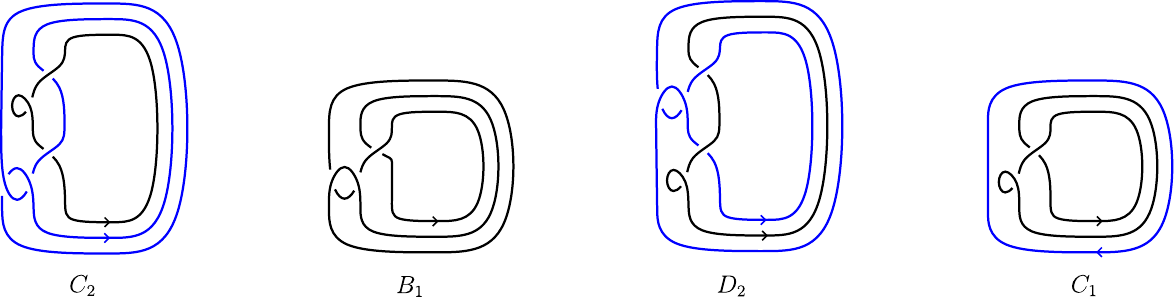}
\caption{Beginning of the recursion.}
\label{recursion-start}
\end{figure}

Setting $a_n=\nabla_{A_n}$, $b_n=\nabla_{B_n}$ and $c_n=\nabla_{C_n}$, we obtain a system of recurrences
\[\begin{aligned}
a_n&=b_n-zc_n, &n\ge 1\\
b_n&=(1-z^2)b_{n-1}+zb_{n-2}, &n\ge 3\\
c_n&=z(b_{n-1}-2b_{n-2})-z^2(c_{n-1}-c_{n-2}), &n\ge 3
\end{aligned}\]
with initial conditions $b_1=1$, $b_2=1-z^2$, $c_1=0$ and $c_2=z$.%
\footnote{Note that the same system of recurrences with $z=x-x^{-1}$ and with a different set 
of initial conditions applies also to $a_n=\Omega_{M_n}(x,y)$.}
Solving the second and third equations, we obtain 
\[\begin{aligned}
b_n&=1-z^2+z^4-z^6+\dots+(-z^2)^{n-1}\\
\sum_{n=1}^\infty c_nx^n&=\frac{1}{1+z^2(x-x^2)}\sum_{n=2}^\infty z\big(b_{n-1}-2b_{n-2}\big)x^n,
\end{aligned}\] 
where $b_0=0$.
Using that $2b_{n-2}-b_{n-1}=b_{n-1}-2(-z^2)^{n-2}$ for $n\ge 2$, we get
\[\sum_{n=1}^\infty a_nx^n=
\frac{1}{1+z^2(x-x^2)}\Bigg(\sum_{n=1}^\infty b_n x^n\Big(1+z^2(x-x^2)\Big)+
z^2\sum_{n=2}^\infty\Big(b_{n-1}-2(-z^2)^{n-2}\Big)x^n\Bigg).\]
Now observe that $b_n+z^2b_{n-1}=1$ and $z^2\sum_{n=1}^\infty b_nx^n(x-x^2)=
z^2\sum_{n=1}^\infty(b_n-b_{n-1})x^{n+1}=z^2\sum_{n=1}^\infty(-z^2)^{n-1}x^{n+1}=-\sum_{n=2}^\infty(-z^2)^{n-1}x^n$.
Hence we get
\[\sum_{n=1}^\infty a_nx^n=\frac{1}{1+z^2(x-x^2)}
\Bigg(\sum_{n=1}^\infty x^n-\sum_{n=2}^\infty(-z^2)^{n-1}x^n+2\sum_{n=2}^\infty(-z^2)^{n-1}x^n\Bigg).\]
Thus
\[\sum_{n=1}^\infty\nabla_{J_n}(z)x^n=\frac{x}{1+z^2(x-x^2)}\Big(\frac{1}{1-x}+\frac1{1+z^2x}-1\Big).\]
\end{proof}

\section{Formal limits of rational power series}

The following is well-known, though not necessarily in the form of this particular assertion.

\begin{lemma} \label{recurrence} A formal power series $R=r_0+r_1x+r_2x^2+\dots\in\Z[[x]]$ satisfies a linear 
recurrence of the form $r_k=q_1r_{k-1}+\dots+q_nr_{k-n}$ for each $k>m$, where $q_1,\dots,q_n\in\Z$
(here in the case $n>m+1$ we set $r_{-1}=r_{-2}=\dots=r_{m+1-n}=0$), if and only if $R=P/Q$, where 
$P,Q\in\Z[x]$, $\deg Q\le n$ and $\deg P\le m$.
\end{lemma}

\begin{proof} Suppose that $R$ is of the form $P/Q$, where $P$ and $Q$ are as above.
Since $Q$ is invertible in $\Z[[x]]$, we have $Q(0)=\pm1$; by changing the signs of $P$ and $Q$ 
we may assume that $Q(0)=1$.
Let us write $Q=1-q_1x-\dots-q_nx^n$.
Then 
\begin{multline*}
P=QR=(1-q_1x-\dots-q_nx^n)(r_0+r_1x+\dots)=r_0+(r_1-q_1r_0)x+(r_2-q_1r_1-q_2r_0)x^2\\
+\dots+(r_n-q_1r_{n-1}\dots-q_nr_0)x^n+(r_{n+1}-q_1r_{n+2}-\dots-q_nr_1)x^{n+1}+\dots.
\tag{$*$}
\end{multline*}
Since $\deg P\le m$, the coefficients at $x^i$ for $i>m$ must be zero, so we obtain the desired recurrence relation.

Conversely, suppose that the said recurrence relation is satisfied for $k>m$, and let $Q=1-q_1x-\dots-q_nx^n$.
Then $P:=QR$ is a polynomial of degree $\le m$ (see ($*$)).
\end{proof}

\begin{theorem} \label{convergence}
Let $P_1,Q_1,P_2,Q_2,\dots\in\Z[x]$, where each $Q_n\in\Z[[x]]^\times$ (in other words, $Q_n(0)=\pm1$).
Suppose that the $P_i$ converge%
\footnote{In the usual topology on $\Z[[x]]$, that is, its topology as the inverse limit of 
the discrete rings $\Z[x]/(x^i)$.
More explicitly, the condition $P_i\to 0$ is saying the for each $k$ there exists an $n$
such that for each $i>n$, $P_i$ is divisible by $x^k$.} 
to $0$ in $\Z[[x]]$ but each $P_i\ne 0$.
Then there exists a function $f\:\N\to\N$ such that for any increasing sequence $r_1,r_2,\dots$
satisfying $r_{i+1}\ge f(r_i)$ for each $i$, 

(a) the infinite product%
\footnote{That is, the limit of the finite products in the same topology on $\Z[[x]]$. 
Its existence is clearly guaranteed by the condition $P_i/Q_i\to 0$.} 
$R:=\Bigg(1+\dfrac{P_{r_1}}{Q_{r_1}}\Bigg)\Bigg(1+\dfrac{P_{r_2}}{Q_{r_2}}\Bigg)\cdots$;

(b) the infinite sum%
\footnote{That is, the limit of the finite sums in the same topology on $\Z[[x]]$.
Its existence is clearly guaranteed by the condition  $P_i/Q_i\to 0$.} 
$R:=\dfrac{P_{r_1}}{Q_{r_1}}+\dfrac{P_{r_2}}{Q_{r_2}}+\dots$%

\noindent
is not a rational power series.
\end{theorem}

One cannot do without selecting a subsequence in general.
For instance, \[S:=\prod_{n=1}^\infty\left(1+\frac{x^{n-1}-x^{n+1}}{1-x^{2n}}\right)\] 
is a rational power series.
(Indeed, $S=\prod_{n=1}^\infty\frac{(1+x^{n-1})(1-x^{n+1})}{(1-x^n)(1+x^n)}=\frac{2}{1-x}$.)
Also, if $c_n=\sum_{d\mid n}\mu(d)2^{n/d}$, where $\mu(n)$ is the M\"obius function%
\footnote{$\mu(n)=(-1)^k$ if $n$ is a product of $k$ distinct primes and $\mu(n)=0$ 
if $n$ divisible by a square of a prime.},
then \[S:=\sum_{n=1}^\infty\dfrac{c_nx^n}{1-x^n}\] is a rational power series.
Indeed, by the M\"obius inversion formula $\sum_{d\mid n}c_d=2^n$, whence 
$S=\sum_{n=1}^\infty c_n(x^n+x^{2n}+x^{3n}+\dots)=2x+2^2x^2+2^3x^3+{\dots}=\frac{2x}{1-2x}$.

\begin{proof} We will prove (a) and (b) simultaneously.
Let $n_i=\deg Q_1\cdots Q_i$ and
\[m_i=\begin{cases}\deg (P_1+Q_1)\cdots(P_i+Q_i)&\text{in (a),}\\
\max\{\deg Q_1\cdots Q_{j-1}P_jQ_{j+1}\cdots Q_i\mid j=1,\dots,i\}&\text{in (b).}
\end{cases}\]
Let us define $f\:\N\to\N$ as follows: let $f(i)$ be such that $P_j$ is divisible by $x^{\max(2n_i,m_i^2)+1}$ for all $j\ge f(i)$.
Let $r_1,r_2,\dots$ be any increasing sequence satisfying $r_{i+1}\ge f(r_i)$ for each $i$.
We have $R=\lim R_i$, where 
\[R_i=\begin{cases}
\Big(1+\dfrac{P_{r_1}}{Q_{r_1}}\Big)\cdots\Big(1+\dfrac{P_{r_i}}{Q_{r_i}}\Big)
=\dfrac{(P_{r_1}+Q_{r_1})\cdots(P_{r_i}+Q_{r_i})}{Q_{r_1}\cdots Q_{r_i}}&\text{in (a)}\\[15pt] 
\dfrac{P_{r_1}}{Q_{r_1}}+{\dots}+\dfrac{P_{r_i}}{Q_{r_i}}=\dfrac{P_{r_1}Q_{r_2}\cdots Q_{r_i}+
{\dots}+Q_{r_1}\cdots Q_{r_{i-1}}P_{r_i}}{Q_{r_1}\cdots Q_{r_i}}&\text{in (b).}
\end{cases}\]

Suppose that $R=P/Q$, where $\deg P=M$ and $\deg Q=N$.
Choose $i_0$ so that $R-R_i$ is divisible by $x^{\max(2M,N^2)+1}$ for all $i\ge i_0$.
Let us fix some $i\ge i_0$.
For each $j>i$ we have $r_j\ge r_{i+1}\ge f(r_i)$ and hence $P_{r_j}$ is divisible by $x^{\max(2n_{r_i},m_{r_i}^2)+1}$.
Let $n=\deg Q_{r_1}\cdots Q_{r_i}$ and let
\[m=\begin{cases}\deg (P_{r_1}+Q_{r_1})\cdots(P_{r_i}+Q_{r_i})&\text{in (a),}\\
\max\{\deg Q_{r_1}\cdots Q_{r_{j-1}}P_{r_j}Q_{r_{j+1}}\cdots Q_{r_i}\mid j=1,\dots,i\}&\text{in (b).}
\end{cases}\]
Then $n\le n_{r_i}$ and $m\le m_{r_i}$, and hence $P_{r_j}$ is divisible by $x^{\max(2n,m^2)+1}$ for each $j>i$.
It easily follows from this that $R-R_i$ is also divisible by $x^{\max(2n,m^2)+1}$. In more detail,
in (a) we note that $\Big(1+\dfrac{P_{r_{i+1}}}{Q_{r_{i+1}}}\Big)\dots\Big(1+\dfrac{P_{r_j}}{Q_{r_j}}\Big)-1$
is divisible by $x^{\max(2n,m^2)+1}$ for each $j>i$, and hence so is
$\Big(1+\dfrac{P_{r_{i+1}}}{Q_{r_{i+1}}}\Big)\Big(1+\dfrac{P_{r_{i+2}}}{Q_{r_{i+2}}}\Big)\cdots-1
=\dfrac{R}{R_i}-1=\dfrac{R-R_i}{R_i}$.
In (b), $\dfrac{P_{r_{i+1}}}{Q_{r_{i+1}}}+{\dots}+\dfrac{P_{r_j}}{Q_{r_j}}$
is divisible by $x^{\max(2n,m^2)+1}$ for each $j>i$, and hence so is
$\dfrac{P_{r_{i+1}}}{Q_{r_{i+1}}}+\dfrac{P_{r_{i+2}}}{Q_{r_{i+2}}}+{\dots}=R-R_i$.

Thus $R-R_i$ is divisible by $x^s$, where $s=\max(2M,N^2,2n,m^2)+1$.
Then $QR-QR_i$ is also divisible by $x^s$.
Let us write $QR_i=a_0+a_1x+a_2x^2+\dots$.
Since $QR=P$ is a polynomial of degree $M<s$, we obtain that $a_0+a_1x+{\dots}+a_Mx^M=P$
and $a_{M+1}={\dots}=a_{s-1}=0$.
On the other hand, by Lemma \ref{recurrence} $QR_i$ satisfies 
a linear recurrence of the form $a_j=c_1a_{j-1}+{\dots}+c_na_{j-n}$ for each $j>Nm$.
Since $s>\max(2M,2n)\ge M+n$ and $s>\max(N^2,m^2)\ge Nm$, we obtain that $0=a_s=a_{s+1}=\dots$.
Thus $QR_i=P$.
In other words, $R_i=P/Q=R$.
Now $i$ was an arbitrary integer such that $i\ge i_0$.
Thus $R_{i_0}=R=R_{i_0+1}$.
Hence $P_{i_0+1}=0$, which contradicts the hypothesis.
\end{proof}

\section{Proof of Theorem \ref{main1}} 

For a two-component PL link $L$ it will be convenient to use the shifted Conway polynomial $\varnabla_L(z):=z^{-1}\nabla_L(z)$.
Thus $\varnabla_L(x-x^{-1})=\Omega_L(x,x)$.
If we divide the shifted Conway polynomial $\varnabla_L(z)$ of $L=(K_1,K_2)$ by the product of the non-shifted Conway polynomials 
of the two components:
\[\bar\varnabla_L(z)=\frac{\varnabla_L(z)}{\nabla_{K_1}(z)\nabla_{K_2}(z)},\]
then we get an invariant of PL isotopy of $L$ (see Corollary \ref{connected sum}).%
\footnote{We do need to use both shifted and non-shifted Conway polynomials in the same fraction.
If we use the shifted Conway polynomials in the denominator, then we do not get the PL isotopy invariance.
If use the non-shifted Conway polynomial in the numerator, then we do not get the formula of Theorem \ref{splice2}.
Of course one can take the view that both the numerator and the denominator use the polynomial
$z^{m-1}\nabla_L(z)$, where $m$ is the number of components of $L$; but this approach does not apply 
in the multi-variable setting, which will be used to prove Theorems \ref{main1'} and \ref{main3}.}
Its coefficients are finite type invariants (see \cite{M24-1}*{Lemma \ref{part1:reduced-coefficients2}(a)}).
Hence by Theorem \ref{finite-type} it extends to an invariant of isotopy of topological links.

\begin{theorem} \label{splice2}
Let $\Lambda$ be the splice of two-component PL links $L$ and $L'$ with $\lk(L)=\lk(L')=1$. 
Then $\bar\Omega_\Lambda(x,y)=\bar\Omega_L(x,y)\bar\Omega_{L'}(x,y)$, where%
\footnote{Although $\bar\varnabla_L(z)$ may be understood either as a rational function or a formal power series,
$\bar\Omega_L(x,y)$ will be understood as a rational function only.
One may, of course, identify $\bar\Omega_L(x,y)$ with a formal Laurent series $\Lambda(x,y)$, but this is not
the only possibility.
If $\Omega_L(x,y)$ happens to be a polynomial in $u:=x-x^{-1}$ and $v:=y-y^{-1}$, then another possibility 
is to identify $\bar\Omega_L$ with a formal power series $\Phi(u,v)$.
In this case one cannot in general recover $\Lambda(x,y)$ from $\Phi(u,v)$ by substituting $x-x^{-1}$ for $u$ and
$y-y^{-1}$ for $v$.
For instance, if $(1-x+x^{-1})^{-1}$ is treated as a Laurent series, it equals 
$x(1-x^2+x)^{-1}=x\big(1+(x^2-x)+(x^2-x)^2+\dots\big)$, which cannot be obtained from
$(1-z)^{-1}=(1+z+z^2+\dots)$ by substituting $x-x^{-1}$ for $z$.}
\[\bar\Omega_{(K,K')}(x,y)=\frac{\Omega_{(K,K')}(x,y)}{\nabla_{K}(x-x^{-1})\nabla_{K'}(y-y^{-1})}.\]
In particular, 
$\bar\varnabla_\Lambda(z)=\bar\varnabla_L(z)\,\bar\varnabla_{L'}(z)$.
\end{theorem}

\begin{proof} By Theorem \ref{splice} $\Omega_\Lambda(x,y)=\Omega_L(x,y)\Omega_{L'}(x,y)$.
Let $\Lambda=(K_1,K_2)$, $L=(K_1,J)$ and $L'=(B,Q)$.
The knot $K_2$ is itself the splice of $J$ with $(B,Q)$.
Then by Theorem~ \ref{splice} $\Omega_{K_2}(y)=\Omega_J(y)\,\Omega_{(B,Q)}(1,y)$.
Also by Theorem \ref{torres} $\Omega_{(B,Q)}(1,y)=\nabla_Q(y-y^{-1})$.
Hence $\nabla_{K_2}(z)=\nabla_J(z)\,\nabla_Q(z)$.
Since $\nabla_B(z)=1$, the assertion follows.
\end{proof}

Let $X$ be the Bing sling contained in the solid torus $T:=S^1\x D^2$. 
Let us assume that $T$ is standardly embedded in $S^3=S^1\x D^2\cup D^2\x S^1$, and let $B$ be a meridian of $T$
and $C$ the core circle of $T$.

The Bing sling $X$ is defined as the limit (in the space of maps) of PL knots $X_0,X_1,\dots$, all contained in $T$, 
where $X_0=C$ and each $X_{i+1}$ is the image of a periodic PL knot $J_{r_{i+1}}$ contained in $T$ under 
an untwisted homeomorphism $h_i$ between $T$ and a regular neighborhood $T_i$ of $X_i$.
Here the knots $J_r$ are as in \S\ref{geometry}, and, as discussed in \S\ref{intro}, the sequence $r_1,r_2,\dots$
needs to be chosen appropriately in order for $X$ to be homeomorphic to $S^1$.

{\bf Step I.} Let us first show that $Y:=(B,X)$ is not isotopic to any PL link, provided that the $r_i$ increase 
sufficiently rapidly. 

Let $Y_i=(B,X_i)$ and $L_i=(B,J_{r_i})$.%
\footnote{Let us note that the unknotted component $B$ comes first here, and therefore will correspond 
to the variable $x$; whereas in \S\ref{geometry} it came last in $M_r=(J_r,B)$ and therefore 
corresponded to the variable $y$.
There are reasons for both choices, and it seems that reversing either of them would result in a greater
inconvenience than the present inconvenience of switching from one choice to another.}
Each $Y_{i+1}$ is clearly the splice of $Y_i$ with $L_{i+1}$.
Then by Theorem \ref{splice2} we get
$\bar\varnabla_{Y_{i+1}}(z)=\bar\varnabla_{Y_i}(z)\,\bar\varnabla_{L_{i+1}}(z)$.
Since $Y_0=(B,C)$ is the Hopf link, $\bar\varnabla_{Y_0}=1$, and we get by induction that 
$\bar\varnabla_{Y_n}=\bar\varnabla_{L_1}\cdots\bar\varnabla_{L_n}$.

Since the links $Y_n$ converge to $Y$ in the $C^0$ topology, their reduced shifted Conway polynomials
$\bar\varnabla_{Y_0},\bar\varnabla_{Y_1},\dots$ converge to $\bar\varnabla_Y$ 
in the ring of formal power series.
Since each $\bar\varnabla_{Y_n}$ is the finite product $\bar\varnabla_{L_1}\cdots\bar\varnabla_{L_n}$,
their limit $\bar\varnabla_Y$ is the infinite product $\bar\varnabla_{L_1}\bar\varnabla_{L_2}\cdots$
whose factors $\bar\varnabla_{L_i}$ converge to $1$.%
\footnote{There is also a more geometric way to see that $\bar\varnabla_{L_i}$ converge to $1$.
The knots $J_{r_1},J_{r_2},\dots$, or rather their equivalent versions in thinner neighborhoods of $C$ in $T$, 
have to converge to $C$ in order for $X_1,X_2,\dots$ to converge to a topological embedding of a circle. 
Since $Y_0=(B,C)$ is the Hopf link, each $L_i=(B,J_{r_i})$ is $k_i$-quasi-isotopic to the Hopf link, where $k_i\to\infty$ 
as $i\to\infty$.
Hence the coefficients of $\bar\varnabla_{L_i}$ at degrees $\ge 2k_i$ are zero.}
By Corollary \ref{conway-MJ-degrees} $\bar\varnabla_{L_i}\ne 1$, and hence by
Theorem \ref{convergence}(a) this infinite product is not a rational power series,
as long as the sequence $r_i$ increases sufficiently rapidly.
Hence $Y$ is not isotopic to any PL link.

{\bf Step II.} Next we prove Theorem \ref{main1} in the general case.
Thus we are going to show that $Z:=(K,X)$ is not isotopic to any PL link.

Since $K$ lies in the complement to the Bing sling $X$, it is disjoint from one of the regular neighborhoods $T_i$ 
of the knots $X_i$ such that $X=\bigcap_{i=0}^\infty T_i$.
Say $K\cap T_k=0$.
Then $Z_i:=(K,X_i)$ is a PL link for each $i\ge k$.
Moreover, each $Z_{i+1}$ is the splice of $Z_i$ and $L_{i+1}$ for all $i\ge k$.
Since $\lk Z_i=\lk Z=1$, arguing as before we obtain that 
$\bar\varnabla_{Z_i}=\bar\varnabla_{Z_k}\bar\varnabla_{L_{k+1}}\cdots\bar\varnabla_{L_i}$,
and consequently $\bar\varnabla_Z=\bar\varnabla_{Z_k}\bar\varnabla_{L_{k+1}}\bar\varnabla_{L_{k+2}}\cdots$.
Like before, $\bar\varnabla_{L_{k+1}}\bar\varnabla_{L_{k+2}}\cdots$ is not a rational power series,
as long as the sequence $r_i$ increases sufficiently rapidly.
Also $\bar\varnabla_{Z_k}$ is a rational power series, and it is not identically zero due to $\lk Z_k=\lk Z=1$.
Hence the power series $\bar\varnabla_Z$ cannot be rational and thus $Z$ is not isotopic to any PL link.

\begin{remark} \label{growth1}
Let us write down an explicit condition on the $r_i$ that was used in the proof.
We have $\bar\varnabla_{L_i}=\frac{z^{-1}\nabla_{M_r}}{\nabla_{J_r}}$, where $r=r_i$.
Here $\deg\nabla_{J_r}=2r-2$ and $\deg z^{-1}\nabla_{M_r}=2r+2$ (see the proof of Corollary \ref{conway-MJ-degrees}).
If $n_i$ and $m_i$ are as in the proof of Theorem \ref{convergence}(a),
then $n_i=\sum_{j=1}^i(2j-2)=i(i-1)$ and $m_i=\sum_{j=1}^i(2j+2)=i(i+3)$, and hence $\max(2n_i,m_i^2)+1=i^2(i+3)^2+1$.
The polynomial $P_i$ in the proof of Theorem \ref{convergence}(a) is $z^{-1}\nabla_{M_i}-\nabla_{J_i}=(-1)^{i+1}z^{i+1}L_{i+1}(z)$ 
(see Proposition \ref{conway-MJ'}(d)).
The function $f$ in the proof of Theorem \ref{convergence}(a) is defined by: $f(i)$ is such that $P_j$ is divisible by 
$x^{\max(2n_i,m_i^2)+1}$ for all $j\ge f(i)$.
Since $L_{i+1}$ contains a constant term for odd $i$ (and a linear term for even $i$, so we will not get much by omitting the odd $i$),
the divisibility condition is saying that $j+1\ge i^2(i+3)^2+1$.
Thus $f(i)=i^2(i+3)^2$.
That is, the condition used in the proof that $\bar\varnabla_{L_1}\bar\varnabla_{L_2}\cdots$ is not rational
is that each $r_{i+1}\ge r_i^2(r_i+3)^2$.
Since $\bar\varnabla_{L_1}\cdots\bar\varnabla_{L_k}$ is rational, to prove that $\bar\varnabla_{L_{k+1}}\bar\varnabla_{L_{k+2}}\cdots$ 
is not rational we may use the same condition.
\end{remark}

\section{Proof of Theorem \ref{main2}}

\begin{lemma} \label{anti-F-isotopy}
Let $X$ be a knot and let $K$ and $K'$ be PL knots each linking $X$ with $\lk=1$.
If there exists a PL solid torus $T$, disjoint from both $K$ and $K'$ and containing $X$ so that 
$[X]$ is a generator of $H_1(T)$, then $\bar\varnabla_{(K',X)}=R\bar\varnabla_{(K,X)}$ for some rational power series $R$.
\end{lemma}

\begin{proof} Let $X_0$ be a core circle of $T$, and $X_1,X_2,\dots$ be PL knots in the interior of $T$ converging to $X$.
Let $h$ be a homemorphism between $T$ and the standard unknotted solid torus $S^1\x D^2\subset S^3$.
Let $M\subset S^3$ be a meridian $pt\x\partial D^2$ of the standard unknotted solid torus, $Q=h(X)$ and $Q_n=h(X_n)$.
Then each link $Z_n:=(K,X_n)$ is a splice of $Z_0$ and $L_n:=(M,Q_n)$.
Hence $\bar\varnabla_{Z_n}=\bar\varnabla_{Z_0}\bar\varnabla_{L_n}$.
Since $Z_n\to Z:=(K,X)$ and $L_n\to L:=(M,Q)$, we have $\bar\varnabla_{Z_n}\to\bar\varnabla_Z$ and $\bar\varnabla_{L_n}\to\bar\varnabla_L$.
The latter clearly implies $\bar\varnabla_{Z_0}\bar\varnabla_{L_n}\to\bar\varnabla_{Z_0}\bar\varnabla_L$.
Hence $\bar\varnabla_Z=\bar\varnabla_{Z_0}\bar\varnabla_L$.
Similarly $\bar\varnabla_Y=\bar\varnabla_{Y_0}\bar\varnabla_L$, where $Y=(K',X)$ and $Y_0=(K',X_0)$.
Since $\lk(Z_0)=\lk(Z)=1$ and $\lk(Z_0)$ is the free term of $\bar\varnabla_{Z_0}$, we have $\bar\varnabla_{Z_0}\ne 0$.
Then $\bar\varnabla_Y=\bar\varnabla_{Y_0}\bar\varnabla_{Z_0}^{-1}\bar\varnabla_Z$, where $\bar\varnabla_{Y_0}\bar\varnabla_{Z_0}^{-1}$
is a rational power series.
\end{proof}

\begin{lemma} \label{colimit}
Let $\K$ be a knot in $S^3$ which is the intersection of solid tori $T_1\supset T_2\supset\dots$.
Then $\K$ represents a generator of $H_1(T_i)$ for all but finitely many $i$.
\end{lemma}

\begin{proof}
By the Alexander duality $H_1(S^3\but\K)\simeq\Z$ and each $H_1(S^3\but T_i)\simeq\Z$.
Since $S^3\but\K=\bigcup_i S^3\but T_i$, we have $H_1(S^3\but\K)\simeq\colim H_1(S^3\but T_i)$.
It is easy to see that $\colim(\Z\xr{f_1}\Z\xr{f_2}\dots)\not\simeq\Z$ if infinitely many of the $f_i$ 
are not isomorphisms.
Hence only finitely many of the $f_i$ are not isomorphisms, and then without loss of generality all of them 
are isomorphisms.
Then the inclusion induced maps $H_1(T_i)\to H_1(T_{i+1})$ are also isomorphisms.
Then $\K$ represents a generator of $H_1(T_i)$ for each $i$.
\end{proof}

Theorem \ref{main2} is a consequence of the following

\begin{theorem} Suppose that the Bing sling $X_0$ is isotopic to a knot $X_1$ through 
knots $X_t$ such that when $t$ belongs to a dense subset $S$ of $[0,1]$,
$X_t$ is the intersection of a nested sequence of PL solid tori.
Then $X_1$ is a wild knot.
\end{theorem}

\begin{proof} Let $0=t_0<\dots<t_n=1$ and $K_0,\dots,K_{n-1}$ be given by Lemma \ref{dense}.
By the proof of Theorem \ref{main1}, $\bar\varnabla_{(K_0,X_0)}$ is not rational.
Then it follows from Lemmas \ref{colimit} and \ref{anti-F-isotopy} that $\bar\varnabla_{(K_{n-1},X_1)}$ 
is also not rational.
Hence $X_1$ is wild.
\end{proof}

\section{Two-variable Conway polynomial} \label{uvw}

For a $2$-component PL link $L=(K_1,K_2)$ of linking number $\lambda$ it is not hard to show that $\Omega_L(x,y)$ can be
uniquely rewritten in the form 
\[\varnabla_L(u,v,w):=\sum_{i\equiv j\not\equiv\lambda\,(\text{mod }2)}a_{ij}u^iv^j+\sum_{i\equiv j\equiv\lambda\,(\text{mod }2)}a_{ij}u^iv^jw,\]
where $u=x-x^{-1}$, $v=y-y^{-1}$ and $w=xy^{-1}+x^{-1}y$ and each $a_{ij}\in\Z$ (see \cite{M24-2}). 

\begin{remark} Let $\bar w=xy+x^{-1}y^{-1}$.
Then $w-\bar w=-uv$ and $w\bar w=u^2+v^2+4$.
So by Vieta's theorem $w$ and $-\bar w$ are the roots of the equation $w^2+uvw-(u^2+v^2+4)=0$.
Thus $w=-\frac{uv}2+\sqrt{\frac{u^2v^2}4+u^2+v^2+4}$, in the sense that
$(w+\frac{uv}2)^2=\frac{u^2v^2}4+u^2+v^2+4$ in $\Q[x^{\pm1},y^{\pm1}]$.
In fact, $w+\frac{uv}2=\frac12(x+x^{-1})(y+y^{-1})$.
Thus the $\Q$-subalgebra of $\Q[x^{\pm1},y^{\pm1}]$ generated by $u$, $v$ and $w$
can be described as $\Q[u,v]\Big[\sqrt{u^2v^2+4(u^2+v^2+4)}\Big]$, where
$\sqrt{u^2v^2+4(u^2+v^2+4)}$ can be identified with the element 
$(x+x^{-1})(y+y^{-1})$ of $\Q[x^{\pm1},y^{\pm1}]$. 
(But there is no such description over $\Z$.)
Of course, $\Q[x^{\pm1},y^{\pm1}]$ itself is generated by its elements 
$u$, $v$, $x+x^{-1}$ and $y+y^{-1}$.
\end{remark}

Let us note that upon equating $x$ and $y$ we get $u=v$ and $w=2$, with $\varnabla_L(z)=\varnabla_L(z,z,2)$.
In general, we will call polynomials (resp.\ formal power series) in $u,v,w$ that are linear in $w$ and include only 
those monomials where the degrees of all three variables have the same parity {\it $w$-polynomials} 
(resp.\ $w$-series).
Thus when $\lk(L)$ is odd, $\varnabla_L$ is a $w$-polynomial.
When multiplying two or more $w$-polynomials we need to remember that there is an algebraic relation between 
$u$, $v$ and $w$, namely, $w^2=u^2+v^2+4-uvw$.
Therefore a product of finitely many $w$-polynomials will again be a $w$-polynomial (if we take the product 
in $\Z[u,v,w]/(w^2+uvw-u^2-v^2-4)$), and a product of finitely many $w$-series will again be a $w$-series 
(if we take the product in $\Z[[u,v,w]]/(w^2+uvw-u^2-v^2-4)$).
This also works for infinitely many $w$-series as long as they converge to $1$ in the usual topology 
on $\Z[[u,v,w]]$.
It is easy to see that the following are equivalent for a $w$-series $R$:
\begin{enumerate}
\item $R$ is rational (as a power series in $\Z[[u,v,w]]$);%
\footnote{A power series in $\Z[[x_1,\dots,x_n]]$ is called {\it rational} if it is of the form $P/Q$, where 
$P,Q\in\Z[x_1,\dots,x_n]$ and $Q\in\Z[[x_1,\dots,x_n]]^\times$, i.e.\ $Q(0,\dots,0)=\pm1$; compare \cite{HLP}.}
\item $R$ is of the form $P(u,v,w)/Q(u,v)$, where $P$ and $Q$ are $w$-polynomials with $Q$ constant in $w$ and 
$Q(0,0)=1$;
\item $R$ is of the form $R'(u,v)+wR''(u,v)$, where $R'$ and $R''$ are rational power series in $u$ and $v$.
\end{enumerate}

\section{Proof of Theorem \ref{main1'}} \label{proof of th2}

For a $2$-component PL link $L=(K_1,K_2)$ let \[\bar\varnabla_L(u,v,w)=\frac{\varnabla_L(u,v,w)}{\nabla_{K_1}(u)\,\nabla_{K_2}(v)}.\]
Then $\bar\varnabla_L(u,v,w)$ is a rational $w$-series if $\lk(L)$ is odd.
By \cite{M24-2}*{Theorem \ref{part2:extension}(a)} each coefficient of $\bar\varnabla_L$ assumes the same value on all sufficiently
close PL approximations of a topological link; consequently $\bar\varnabla_L$ extends to an invariant of isotopy of topological links
(which is, of course, a $w$-series that no longer needs to be rational).

\begin{remark} The rational function $\bar\Omega_L$ does not seem to have a reasonable extension to an invariant 
of topological links (with respect to some completion of the field of rational functions).
On the other hand, if $L$ and $L'$ are topologically isotopic PL links, then $\bar\varnabla_L=\bar\varnabla_{L'}$,
and if we understand $\bar\varnabla_L$ and $\bar\varnabla_{L'}$ as fractions whose denominator does not involve $w$
(rather than power series) and substitute $x-x^{-1}$, $y-y^{-1}$ and $xy^{-1}+x^{-1}y$ for $u$, $v$ and $w$
separately in the numerator and in the denominator, then we obtain that $\bar\Omega_L=\bar\Omega_{L'}$.
\end{remark}

The definition of $\varnabla_L$ along with the fact that the Conway polynomial of a knot has no odd exponents imply
the following representation for $\bar\varnabla_L$:
\[\bar\varnabla_L(u,v,w)=\begin{cases}
\sum_{i=0}^\infty u^{2i+1}vP_{2i+1}(v^2)+w\sum_{i=0}^\infty u^{2i}P_{2i}(v^2),&\text{if $\lk(L)$ is even;}\\
\sum_{i=0}^\infty u^{2i}P_{2i}(v^2)+w\sum_{i=0}^\infty u^{2i+1}vP_{2i+1}(v^2),&\text{if $\lk(L)$ is odd,}
\end{cases}\]
where each $P_k$ is a rational power series.
It follows from Theorem \ref{torres} that $\bar\varnabla_L(0,v,y+y^{-1})=\bar\Omega_L(1,y)=\frac{y^{\lk(L)}-y^{-\lk(L)}}{y-y^{-1}}$.
Therefore $P_0$ is a function of $\lk(L)$.%
\footnote{In fact, it is easy to see from Example \ref{fibonacci-lucas} that when $\lk(L)$ is even, $P_0(v)=v^{-1}F_{\lk(L)}(v)$;
and when $\lk(L)$ is odd, $P_0(v)=v^{-1}L_{\lk(L)}(v)$.}

Let us denote $P_1$ by $C_L$.
When $\lk(L)=0$, $C_L(z)=\sum_{k=0}^\infty(-1)^k\beta_{k+1} z^{2k}$, where $\beta_k$ are Cochran's derived invariants 
(see \cite{M24-2}*{Theorem \ref{part2:main-jin}}).
By the above, $C_L$ extends to an invariant of isotopy of topological links and by \cite{M24-2}*{Theorem \ref{part2:extension}(b)} 
the extension invariant under sufficiently close $C^0$-perturbation of the first component.

\begin{theorem} \label{cochran-splice} Let $\Lambda$ be the splice of two-component PL links $L$ and $L'$ with $\lk L=\lk L'=1$.
Then $C_\Lambda(z)=C_L(z)+C_{L'}(z)$.
\end{theorem}

\begin{proof} Let $L=(K_1,K_2)$.
By Theorem \ref{torres} $\Omega_L(1,y)=\nabla_{K_2}(y-y^{-1})$ and $\nabla_{K_1}(0)=1$.
Hence $\bar\Omega_L(1,y)=1$.
Therefore $P_0=1$.
Thus $\bar\varnabla_L(u,v,w)\equiv 1+uvw\,C_L(v)\pmod{(u^2)}$.
And similarly for $L'$ and $\Lambda$.
By Theorem \ref{splice2} $\bar\varnabla_\Lambda(u,v,w)=\bar\varnabla_L(u,v,w)\,\bar\varnabla_{L'}(u,v,w)$, and the assertion follows.
\end{proof}

\begin{proof}[Proof of Theorem \ref{main1'}]
Suppose that $\K$ is a topological knot in the complement of the Bing sling $X$ such that $\lk (\K,X)=1$.
Since $C_L$ is invariant under sufficiently close $C^0$-perturbation of the first component, $C_{(\K,X)}=C_{(K,X)}$ 
for any PL knot $K$ sufficiently close to $\K$.
Let $Z=(K,X)$.
Since $K$ lies in $S^3\but X$, it is disjoint from one of the regular neighborhoods $T_i$ of the knots $X_i$ 
such that $X=\bigcap_{i=0}^\infty T_i$.
Say $K\cap T_k=0$.
Then $Z_i:=(K,J_{r_i})$ is a PL link for each $i\ge k$.
Then similarly to the proof of the PL case of Theorem \ref{main1}, using Theorem \ref{cochran-splice} 
instead of Theorem \ref{splice2}, we get that $C_Z=C_{Z_k}+C_{L_{k+1}}+C_{L_{k+2}}+\dots$.
By Proposition \ref{conway-MJ'}(b,c) $C_{L_i}=C_{(B,J_{r_i})}=z^rF_r(z)/\nabla_{J_r}(z)\ne 0$, where $r=r_i$.
Hence by Theorem \ref{convergence}(b) we get that $C_{L_{k+1}}+C_{L_{k+2}}+\dots$ is not a rational power series, 
as long as the sequence $r_i$ increases sufficiently rapidly.
Also $C_{Z_k}$ is a rational power series.
Hence the power series $C_{(\K,X)}=C_Z$ cannot be rational.
Since $C_L$ is invariant under isotopy of $L$, we conclude that $(K,X)$ is not isotopic to any PL link.
\end{proof}

\begin{remark} \label{growth2}
Let us write down an explicit condition on the $r_i$ that was used in the proof.
We have $C_{L_i}=z^rF_r(z)/\nabla_{J_r}(z)$, where $r=r_i$.
Here $\deg\nabla_{J_r}=2r-2$ (see the proof of Corollary \ref{conway-MJ-degrees}) and 
$\deg z^rF_r=2r-1$.
If $n_i$ and $m_i$ are as in the proof of Theorem \ref{convergence}(b),
then $n_i=\sum_{j=1}^i(2j-2)=i(i-1)$ and $m_i=n_i+1=i^2-i+1$, and hence $\max(2n_i,m_i^2)+1=(i^2-i+1)^2+1$.
The polynomial $P_i$ in the proof of Theorem \ref{convergence}(b) is $z^iF_i(z)$.
The function $f$ in the proof of Theorem \ref{convergence}(b) is defined by: $f(i)$ is such that $P_j$ is divisible by 
$x^{\max(2n_i,m_i^2)+1}$ for all $j\ge f(i)$.
Since $F_i$ contains a constant term for odd $i$ (and a linear term for even $i$, so we will not get much by omitting the odd $i$),
the divisibility condition is saying that $j\ge(i^2-i+1)^2+1$.
Thus $f(i)=(i^2-i+1)^2+1$.
That is, the condition used in the proof that $C_{L_1}+C_{L_2}+\dots$ is not rational
is that each $r_{i+1}\ge(r_i^2-r_i+1)^2+1$.
Since $C_{L_1}+\dots+C_{L_k}$ is rational, to prove that $C_{L_{k+1}}+C_{L_{k+2}}+\dots$
is not rational we may use the same condition.
\end{remark}

\begin{lemma} \label{anti-F-isotopy2} Let $(K,X)$ and $(K',X)$ be links with $\lk(K,X)=\lk(K',X)=1$.
Suppose that there exists a PL solid torus $T\subset S^3$, disjoint from both $K$ and $K'$ and containing $X$ 
so that $H_1(T)$ is generated by $[X]$.
Then $C_{(K,X)}-C_{(K',X')}$ is a rational power series.
\end{lemma}

\begin{proof} Similarly to the proof of Lemma \ref{anti-F-isotopy}.
\end{proof}

Let us call links $(K,X)$ and $(K',X)$ as in Lemma \ref{anti-F-isotopy2} {\it anti-F-isotopic}.
Then we get

\begin{theorem} \label{anti-F-isotopy3} Let $X$ be the Bing sling and $K$ be a knot linking $X$ with $\lk=1$.
Suppose that $(K,X)$ is related to a link $(K',X')$ by the equivalence relation generated by isotopy and
anti-F-isotopy.
Then $X'$ is a wild knot.
\end{theorem}

\section{A disappointing example}

The proofs of Theorem \ref{main2} and Theorem \ref{anti-F-isotopy3} raise the following two questions:

\begin{problem} Let $X$ be a knot and let $K$ and $K'$ be PL knots each linking $X$ with $\lk=1$.
Is it true that 

(a) $C_{(K',X)}(z)$ is a rational power series if and only if so is $C_{(K,X)}(z)$?

(b) $\bar\varnabla_{(K',X)}(z)$ is a rational power series if and only if so is $\bar\varnabla_{(K,X)}(z)$?
\end{problem}

A positive solution to either (a) or (b) would imply that the Bing sling is not isotopic to the unknot.

We will show, however, that (a) has a negative solution, and so does the two-variable version of (b).
This does not yet provide a solution of (b), but there does not seem to be much hope that it can be positive.

\begin{theorem} \label{non-invariance}
There exist $2$-component links $L'=(K',X)$ and $L^*=(K^*,X)$ such that $K'$ and $K^*$ are PL knots, 
$\lk(L')=\lk(L^*)=1$ and

(a) $C_{L^*}(z)$ is rational but $C_{L'}(z)$ is not rational; also,

(b) $\bar\varnabla_{L^*}(u,v,w)$ is rational but $\bar\varnabla_{L'}(u,v,w)$ is not rational.
\end{theorem}

\begin{proof} Let $L=(K,X)$ be one of the links constructed in \cite{M21}, with $\lk(L)=0$
and $C_L$ not being rational (see Figure \ref{w-infty0}).
\begin{figure}[h]
\includegraphics[width=12cm]{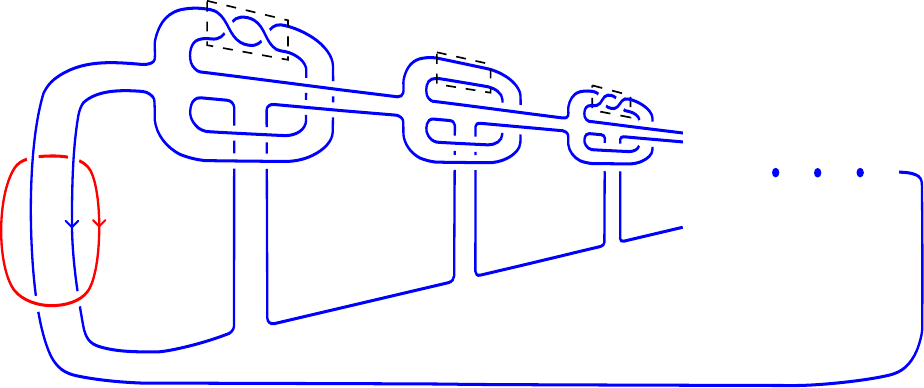}
\caption{The link $L=W_{n_1,n_2,n_3,\dots}$ where $(n_1,n_2,n_3)=(1,0,-1)$.}
\label{w-infty0}
\end{figure}
Let $K'$ and $K''$ be PL knots in the complement of $X$ such that the links $L=(K,X)$, $L':=(K',X)$ 
and $L'':=(K'',X)$ are related as in the second Conway identity; see Figure \ref{w-infty+-}.
\begin{figure}[h]
\includegraphics[width=12cm]{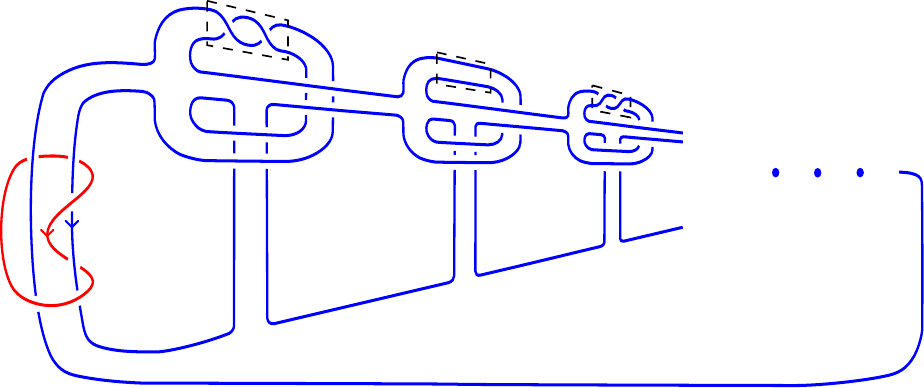}
\hspace{10pt}

\includegraphics[width=12cm]{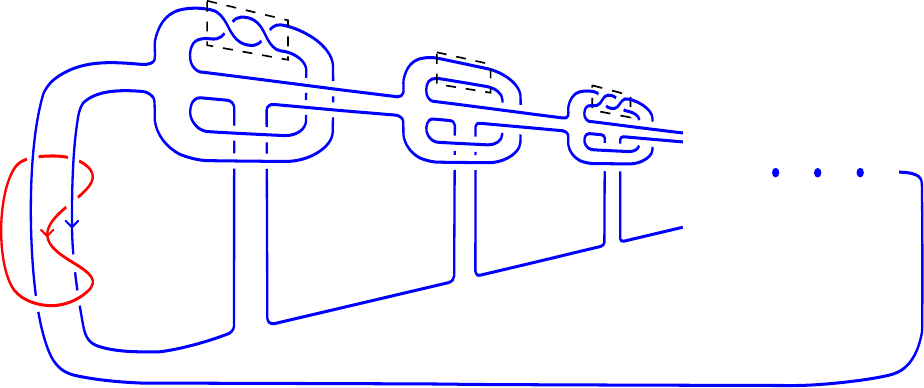}
\caption{The links $L'$ (top) and $L''$ (bottom).}
\label{w-infty+-}
\end{figure}

Since $K$ is unknotted in $S^3$ (see \cite{M21}), so are $K'$ and $K''$.
If $Q$ is a sufficiently close PL approximation of $X$, then $\Lambda:=(K,Q)$, $\Lambda':=(K',Q)$ and 
$\Lambda'':=(K'',Q)$ are also related as in the second Conway identity, and so we have
$\Omega_{\Lambda'}(x,y)+\Omega_{\Lambda''}(x,y)=(xy+x^{-1}y^{-1})\Omega_{\Lambda}(x,y)$.
Writing $\Lambda^*=(K^*,Q)$, where $K^*=-K''$, that is, $K''$ with orientation reversed, by
Proposition \ref{orientation} we get
$\Omega_{\Lambda'}(x,y)-\Omega_{\Lambda^*}(x^{-1},y)=(xy+x^{-1}y^{-1})\Omega_\Lambda(x,y)$. 
Upon substituting $x$ with $x^{-1}$ we get
$\Omega_{\Lambda'}(x^{-1},y)-\Omega_{\Lambda^*}(x,y)=(x^{-1}y+xy^{-1})\Omega_\Lambda(x^{-1},y)$ 
or equivalently
$\varnabla_{\Lambda'}(-u,v,w)-\varnabla_{\Lambda^*}(u,v,w)=w\varnabla_\Lambda(-u,v,w)$.
Since $K$, $K'$ and $K^*$ are unknotted, we get 
$\bar\varnabla_{\Lambda'}(-u,v,w)-\bar\varnabla_{\Lambda^*}(u,v,w)=w\bar\varnabla_\Lambda(-u,v,w)$.
Since $Q$ can be any sufficiently close approximation of $X$, we obtain
$\bar\varnabla_{L'}(-u,v,w)-\bar\varnabla_{L^*}(u,v,w)=w\bar\varnabla_L(-u,v,w)$, 
where $L^*=(K^*,X)$.
Now since $X$ is locally flat except at one point, it is easy to see that $L^*$ is isotopic to the Hopf link.
Hence $\bar\varnabla_{L^*}(u,v,w)=1$, and in particular $C_{L^*}=0$.
Therefore $\bar\varnabla_{L'}(u,v,w)=w\bar\varnabla_L(u,v,w)$.
Hence also $C_{L'}(z)=C_L(z)$, which is not rational.
Then in particular $\bar\varnabla_{L'}(u,v,w)$ itself cannot be rational.
\end{proof}

\begin{remark}
One might wonder whether the fact that $\bar\varnabla_{(K,X)}(z)$ and $\bar\varnabla_{(K,X)}(u,v,w)$ are represented
as infinite products of rational power series for any PL knot $K$ that links the Bing sling $X$ with linking number $1$
is remarkable as a property of the Bing sling.
It is not.
An arbitrary power series of the form $1+a_1z+a_2z^2+\dots$ factors in $\Z[[z]]$ into
a product of polynomials of the form $(1+b_1z)(1+b_2z^2)(1+b_3z^3)\cdots$, where each $b_n$ is uniquely determined,
assuming that $b_1,\dots,b_{n-1}$ have already been found, by equating the degree $n$ term of $\prod_{i=1}^n(1+b_iz^i)$ 
with that of the given power series.
Similarly, every $w$-series $\sum_{i,j=0,\,i\equiv j(\text{mod}\,2)}^\infty a_{ij}u^iv^jw^{i\bmod 2}$ with $a_{00}=1$ 
factors in $\Z[[u,v,w]]/(w^2+uvw-u^2-v^2-4)$ into a product of polynomials of the form 
$\prod_{i,j=0,\,i\equiv j(\text{mod}\,2)}^\infty (1+b_{ij}u^iv^jw^{i\bmod 2})$ with $b_{00}=0$.
\end{remark}

\section{Proofs of Theorems \ref{main3} and \ref{main4'}}

Theorem \ref{main4'} is a reformulation of the following

\begin{theorem} \label{eta-function}
There exists a $2$-component link $\L$ in $S^3$ with $\lk(\L)=0$ and second component being PL, which is not equivalent to 
any PL link under the equivalence relation generated by (i) isotopy and (ii) replacing a link $(Q,\K)$ by a link $(Q',\K)$, provided that 
$Q$ and $Q'$ are PL knots homotopic in $S^3\but\K$ or more generally represent the same conjugacy class in the quotient of $\pi_1(S^3\but\K)$ 
by the second commutator subgroup.
\end{theorem}

In fact, the proof also works for the following more general version (ii$'$) of the transformation (ii): 
replace a link $(Q,\K)$ by a link $(Q',\K)$, provided that $Q$ and $Q'$ represent the same conjugacy class 
in the quotient of $G:=\pi_1(S^3\but\K)$ by the preimage in $G'$ of the torsion submodule of the knot module $G'/G''$.%
\footnote{Equivalently, $Q'=Q\#_bR$ for some knot $R$ in $S^3\but\K$ such that the Kojima--Yamasaki $\lambda$-polynomial \cite{KY} 
$\lambda(R,\K)\ne 0$ and for some band $b\subset S^3\but\K$.}

\begin{proof}
The proof of Theorem \ref{lk0} in \cite{M21} shows that for $2$-component links $L$ with $\lk(L)=0$ the Cochran power series $C_L$ 
is invariant under isotopy of $L$, is rational for PL links, and is not rational for a certain link $\L$ whose second component is PL.
To complete the proof it suffices to show that for any transformation of the type (ii), $C_{(Q,\K)}-C_{(Q',\K)}$ is always rational.
In fact we will prove that this difference is a polynomial.
(For transformations of type (ii$'$) this difference will only be rational, not polynomial.)

Since $Q$ and $Q'$ represent the same conjugacy class in the quotient of $\pi_1(S^3\but\K)$ by 
the second commutator subgroup, they admit lifts $\tilde Q$, $\tilde Q'$ to the infinite cyclic cover 
$\X$ of $S^3\but\K$ such that the $1$-cycle $\tilde Q-\tilde Q'$ bounds a $2$-chain $\zeta$ in $\X$.
The image of $\zeta$ in $S^3\but\K$ as well as $Q$ and $Q'$ are disjoint from some polyhedral
neighborhood $N$ of $\K$.
Hence $\zeta$ as well as $\tilde Q$ and $\tilde Q'$ lie in the preimage $Y$ of $S^3\but N$ in $\X$.
This $Y$ can also be described as the cover of $S^3\but N$ corresponding to (the kernel of) the homomorphism
$\pi_1(S^3\but N)\to\Z$ given by $[Z]\mapsto\lk(Z,\K)$.
(In other words, this homomorphism is determined by the class in $H^1(S^3\but N)$ which is Alexander dual 
to $[\K]\in H_1(N)$.)
If $K$ is a sufficiently close PL approximation of $\K$, then it is homotopic to $\K$ within $N$,
and then the preimage of $S^3\but N$ in the infinite cyclic cover $X$ of $S^3\but K$ is the same cover $Y$
of $S^3\but N$.
Thus $\zeta$ as well as $\tilde Q$ and $\tilde Q'$ can be regarded as lying in $X$.

For PL links we have $C_L(z)=\eta_L(t)$, where $\eta_L$ is Kojima's $\eta$-function, $t=x^2$ and $z=x-x^{-1}$.
By definition $\eta_{(Q,K)}(t)=\big\langle\tilde Q,\tilde Q_+\big\rangle$, where $\left<\cdot,\cdot\right>$ is 
the Cochran pairing in the infinite cyclic cover $X$ of $S^3\but K$ and $\tilde Q$ and $\tilde Q_+$ are nearby lifts 
of $Q$ and of its parallel pushoff.
By \cite{M24-2}*{Lemma \ref{part2:cochran-pairing}(a)}
$\big\langle\tilde Q,\tilde Q_+\big\rangle-\big\langle\tilde Q',\tilde Q'_+\big\rangle=
\big\langle\partial\zeta,\tilde Q_+\big\rangle+\big\langle\tilde Q',\partial\zeta_+\big\rangle$, where 
$\partial\zeta=\tilde Q-\tilde Q'$ and $\partial\zeta_+=\tilde Q_+-\tilde Q'_+$.
By \cite{M24-2}*{Lemma \ref{part2:cochran-pairing}(c)}
$\big\langle\tilde Q',\partial\zeta_+\big\rangle(t)=\big\langle\partial\zeta_+,\tilde Q'\big\rangle(t^{-1})$.
Thus \[\eta_{(Q,K)}(t)-\eta_{(Q',K)}(t)=
\big\langle\partial\zeta,\,\tilde Q_+\big\rangle(t)+\big\langle\partial\zeta_+,\,\tilde Q'\big\rangle(t^{-1})=
\sum_{n=-\infty}^\infty\big(\zeta\cdot t^n\tilde Q_++\zeta_+\cdot t^{-n}\tilde Q'\big)t^n.\]
The right hand side is a Laurent polynomial (so, only finitely many summands are nonzero) which depends only on 
the mutual position of $\zeta$ and $Q_+$ and of $\zeta_+$ and $Q'$ in $Y$.
Thus it does not depend on the choice of the PL approximation $K$, as long as it is homotopic to $\K$ within $N$.
Since $\eta_L(t^{-1})=\eta_L(t)$, this Laurent polynomial is invariant under the involution $t\mapsto t^{-1}$.
So it must be of the form $\sum_{n=0}^\infty a_n(t^n+t^{-n})$.
It is not hard to check that $t^n+t^{-n}=L_{2n}(z)$ for $n\ge 0$, where $L_n$ is the Lucas polynomial, 
defined by $L_0(z)=2$, $L_1(z)=z$ and $L_{n+1}(z)=zL_n(z)+L_{n-1}(z)$.
(Thus $L_{2n}$ contains only even powers of $z$.)
Since each coefficient of $C_L(z)$ is invariant under sufficiently close approximation \cite{M21}*{Theorem 2.3},
$C_{(Q,\K)}(z)-C_{(Q',\K)}(z)=C_{(Q,K)}(z)-C_{(Q',K)}(z)=\sum_{n=0}^\infty a_nL_{2n}(z)$, which is a polynomial.
\end{proof}

Theorem \ref{main3} follows by the proof of Theorem \ref{main1'} if we additionally use the following

\begin{lemma} \label{link homotopy}
Let $\K$ be a knot in $S^3$ and let $Q_1$, $Q_2$ be PL knots in $S^3\but\K$ such that $\lk(Q_i,K)=1$.
If $Q_1$ and $Q_2$ are homotopic in $S^3\but\K$, then $C_{(Q_1,\K)}-C_{(Q_2,\K)}$ is a polynomial.
\end{lemma}

\begin{proof} Let $N$ be a polyhedral neighborhood of $\K$ such that the given homotopy takes place in $S^3\but N$.
Let $K$ be a PL knot homotopic to $\K$ within $N$.
It suffices to consider the case where $Q_1$ is obtained from $Q_2$ by a single crossing change.
Then for the $2$-variable Conway polynomial of \cite{M24-2} we have
\[\varnabla_{(Q_1,K)}(u,v)-\varnabla_{(Q_2,K)}(u,v)=u\varnabla_{(L,K)}(u,v),\] where $L=(L_1,L_2)$ is the $2$-component link
obtained by smoothing the crossing.
If we divide both sides of the equation by $\nabla_K(v)$, then differentiate both sides with respect to $u$ and
finally set $u=0$, then we get \[C_{(Q_1,K)}(v^2)-C_{(Q_2,K)}(v^2)=\frac{\varnabla_{(L,K)}(0,v)}{\nabla_K(v)}.\]
Let us note that $\lk(L_1,K)+\lk(L_2,K)=\lk(Q_i,K)=1$.
So precisely one of the integers $\lk(L_1,K)$, $\lk(L_2,K)$ is even,
and if we denote it by $n$, then the other one equals $1-n$.
Then by Theorem \ref{torres} and Example \ref{fibonacci-lucas}
\[\Omega_{(L,K)}(1,y)=\frac{(y^n-y^{-n})(y^{1-n}-y^{n-1})}{y-y^{-1}}\nabla_K(y-y^{-1})=
(y+y^{-1})F_n(v)\frac{L_{1-n}(v)}v\nabla_K(v),\]
where $v=y-y^{-1}$.
Then the part of $\Omega_{(L,K)}(x,y)$ which involves no powers of $u=x-x^{-1}$
must be $(x^{-1}y+xy^{-1})F_n(v)\frac{L_{1-n}(v)}v\nabla_K(v)$.
Therefore $\varnabla_{(L,K)}(0,v)=F_n(v)\frac{L_{1-n}(v)}v\nabla_K(v)$.
Hence \[C_{(Q_1,K)}(v^2)-C_{(Q_2,K)}(v^2)=F_n(v)\frac{L_{1-n}(v)}v.\]
Here the left hand side converges to $C_{(Q_1,\K)}-C_{(Q_2,\K)}$ as $K\to\K$ in $C^0$ topology,
whereas the right side does not depend on $K$.
Thus $C_{(Q_1,\K)}(v^2)-C_{(Q_2,\K)}(v^2)=F_n(v)\frac{L_{1-n}(v)}v$, which is a polynomial in $v^2$.
(Note that since $n$ is even, $F_n(v)$ involves only even powers of $v$ and $L_{1-n}$ only odd ones.)
\end{proof}

\section{Proof of Theorem \ref{main4}}

\begin{proposition} \label{torsion-jump}
Let $\K$ be a knot in $S^3$ and $Q$ a PL knot in $S^3\but\K$ such that $\lk(Q,\K)=1$.
Let $S$ be a Seifert surface bounded by $Q$ in $S^3$ and let $F$ be an embedded PL surface 
in $S^3\but\K$, disjoint from $S$ and such that every knot in $F$ is null-homologous in $S^3\but\K$.
Let $(A_1,\dots,A_{2g})$ be a symplectic basis for $F$, and let $\tilde A_i$ be some lifts of $A_i$ 
to the infinite cyclic cover $\X$ of $S^3\but\K$.
Finally let $Q'=Q\#_b\partial F$, where $b$ is a band in $S^3\but\K$, connecting $Q$ to $\partial F$ 
and otherwise disjoint from $S$ and $F$.
If each $[\tilde A_{2i}]$ is annihilated by some Laurent polynomial in $H_1(\X)$, then 
$C_{(Q,\K)}(z)-C_{(Q',\K)}(z)$ is a rational power series.
\end{proposition}

\begin{proof} By the hypothesis there exist $p_1,\dots,p_g\in\Z[t^{\pm1}]$ such that each 
$p_i\tilde A_{2i}$ bounds a $2$-chain $\zeta_i$ in $\X$.
The images of the chains $\zeta_i$ in $S^3\but\K$ as well as the surface $F$ and the knot $Q$ 
are disjoint from some polyhedral neighborhood $N$ of $\K$.
Hence the $\zeta_i$ and the $\tilde A_{2i-1}$ lie in the preimage $Y$ of $S^3\but N$ in $\X$.
This $Y$ can also be described as the cover of $S^3\but N$ corresponding to (the kernel of) the homomorphism
$\pi_1(S^3\but N)\to\Z$ given by $[Z]\mapsto\lk(Z,\K)$.
If $K$ is a sufficiently close PL approximation of $\K$, then it is homotopic to $\K$ within $N$,
and then the preimage of $S^3\but N$ in the infinite cyclic cover $X$ of $S^3\but K$ is the same cover $Y$
of $S^3\but N$.
Thus the $\zeta_i$ and the $A_{2i-1}$ can be regarded as lying in $X$.

Let $\Sigma$ be a Seifert surface bounded by $K$ in $S^3$. 
Since $K$ is homotopic to $\K$ within $N$, which is disjoint from $F$, $\lk(R,K)=\lk(R,\K)=0$ for every knot $R$ in $F$.
Then by the proof of (3)$\Rightarrow$(1) for $g=2$ in Theorem \ref{disjoint seifert}, we may assume that 
$\Sigma$ is disjoint from $F$.
If $\Sigma$ intersects the band $b$, we can eliminate any such intersections by pushing them into $S$.
Now by the proof of Lemma \ref{C-complex} we can eliminate all intersections between $\Sigma$ and $S$ except
for one clasp arc $J$ with one endpoint in $K$ and another in $Q$.
(This elimination consists of modifications of $\Sigma$ and $S$ which take place in a small neighborhood of $\Sigma\cup S$
and therefore create no intersections of either $\Sigma$ or $S$ with $F$.
Also the innermost arc argument which is used to eliminate the intersections between $\Sigma$ and $Q$ can be carried out 
in the arc $Q\but b$ rather than in the entire knot $Q$, and therefore will create no intersections of $\Sigma$ with $b$.)
Thus $S':=S\cup b\cup F$ is an embedded surface, and its intersection with $\Sigma$ consists only of the arc $J\subset S$.

Now \cite{M24-2}*{Theorem \ref{part2:mainmain}} applies to each of the pairs $(\Sigma,S)$ and $(\Sigma,S')$, which yields
\[C_{(Q',K)}(s)-C_{(Q,K)}(s)=(-1)^g\sum_{n=1}^\infty(-1)^n\sum_{i=1}^g\beta_\Sigma^{n-1,\,n}(A_{2i},A_{2i-1})\,s^{n-1}.\]
The left hand side is a rational power series since both $(Q,K)$ and $(Q',K)$ are PL links.
On the other hand, by \cite{M24-2}*{Theorem \ref{part2:cochran expansion}(a)} each $\beta_\Sigma^{n-1,\,n}(A_{2i},A_{2i-1})$
depends only on $\left<A_{2i},A_{2i-1}\right>_\Sigma$, which in turn depends only on 
the mutual position of the $\zeta_i$ and the $A_{2i-1}$ in $Y$.
Thus it does not depend on the choice of the PL approximation $K$, as long as it is homotopic to $\K$ within $N$.
Therefore $C_{(Q',\K)}(s)-C_{(Q,\K)}(s)$ is rational.
\end{proof}

\begin{theorem} \label{torsion-jump2}
Let $\K$ be a knot in $S^3$ and let $Q_1$, $Q_2$ be PL knots in $S^3\but\K$ such that $\lk(Q_i,K)=1$.
Let $G=\pi_1(S^3\but K)$, let $T_0$ be the torsion submodule of the knot module $G'/G''$, and let
$T$ be the preimage of $T_0$ in $G'$.
If $Q_1$ and $Q_2$ represent the same conjugacy class of $G/[G',T]$, then $C_{(Q_1,\K)}-C_{(Q_2,\K)}$ 
is a rational power series.
\end{theorem}

Theorem \ref{main4} follows from Theorem \ref{torsion-jump2} and the proof of Theorem \ref{main1'}.

\begin{proof}
We may assume that $Q_1$ and $Q_2$ pass through the basepoint of $S^3\but\K$ and hence determine some 
classes $a,b\in G$.
Then by the hypothesis there exists a $g\in G$ such that $a=b^g$ modulo $[G',T]$.
Hence $a^{-1}b^g=[c_1,t_1]\cdots[c_n,t_n]$ for some $c_i\in G'$ and some $t_i\in T$.
Let $O$ be a small $3$-ball in $S^3\but\K$ containing the basepoint of $S^3\but\K$.
We may represent $c_1,\dots,c_n$ and $t_1,\dots,t_n$ by PL knots $A_1,\dots,A_n$ and $B_1,\dots,B_n$ 
in $S^3\but(\K\cup Q_1)$ which are all pairwise disjoint except that each $A_i$ meets each $B_i$
in a single point $p_i\in O$. 
We may assume that $\lk(A_i,Q_1)=\lk(B_i,Q_1)=0$ for each $i$.
Each of the wedges $A_i\vee B_i$ is a deformation retract of a punctured torus 
$\tau_i:=\alpha_i\cup\beta_i$, where $\alpha_i$ and $\beta_i$ are thin ribbons which deformation 
retract onto $A_i$ and $B_i$ and intersect in a $2$-disk (a ``square'') which deformation retracts 
onto $p_i$.
Thus each $[\partial\tau_i]=[c_i,t_i]\in G$.
We may assume that the $\tau_i$ are pairwise disjoint and also are disjoint from $\K$ and from $Q_1$.
For $i=1,\dots,n-1$ let $b_i$ be a band in $O$ connecting $\partial\tau_i$ to $\partial\tau_{i+1}$ and 
let $b_0$ be a band in $O$ connecting $Q_1$ to $\partial\tau_1$ so that the $n$ bands are pairwise 
disjoint from each other and also from the $\tau_i$ and from $Q_1$ apart at their attaching arcs.
These attaching arcs can be chosen so that the class of 
$Q_1':=Q_1\#_{b_0}\partial\tau_1\#_{b_1}\dots\#_{b_{n-1}}\partial\tau_n$ in $G$ equals 
$a[c_1,t_1]\cdots[c_n,t_n]$.
Since $\lk(A_i,Q_1)=\lk(B_i,Q_1)=0$ for each $i$, it is easy to see that $Q_1$ bounds a Seifert
surface $S$ in $S^3$ disjoint from each $\tau_i$. 
Also we can make $S$ disjoint from the bands $b_i$ by attaching to $S$ punctured parallel copies of 
the surfaces $\partial N_i$, where $N_n$ is a small regular neighborhood of $\tau_n$ and
each $N_i$ is a small regular neighborhood of $\tau_i\cup b_i\cup N_{i+1}$.
In particular, $S$ is disjoint from the surface $F:=\tau_1\cup b_1\cup\tau_2\cup\dots\cup b_{n-1}\cup\tau_n$
and we have $Q_1'=Q_1\#_{b_0}\partial F$.
Then by Proposition \ref{torsion-jump} $C_{(Q_1',\K)}-C_{(Q_1,\K)}$ is a rational power series.
On the other hand, since $b^g=a[c_1,t_1]\cdots[c_n,t_n]$, there is a free homotopy between $Q_1'$ and $Q_2$
in $S^3\but\K$. 
Then by Lemma \ref{link homotopy} $C_{(Q_2,\K)}-C_{(Q_1',\K)}$ is a polynomial.
\end{proof}

\begin{corollary} \label{main6'}
The Bing sling is not isotopic to any PL knot though knots whose knot module is a torsion module.
\end{corollary}

\begin{proof} Suppose that such an isotopy $X_t$ exists.
Let $0=t_0<\dots<t_n=1$ and $K_0,\dots,K_{n-1}$ be given by Lemma \ref{dense}.
Let $G_i=\pi_1(S^3\but X_{t_i})$.
By Theorem \ref{second commutator} $K_{i-1}$ and $K_i$ represent the same conjugacy class
in $G_i/[G_i',G_i']$.
Let $T_i\subset G_i'$ be the preimage of the torsion submodule of the knot module $G_i'/G_i''$.
By the hypothesis $T_i=G_i'$.
Hence $K_{i-1}$ and $K_i$ represent the same conjugacy class in $G_i/[G_i',T_i]$.
Then by Theorem \ref{torsion-jump2} $C_{(K_{i-1},X_{t_i})}-C_{(K_i,X_{t_i})}$ is a rational power series.
By the proof of Theorem \ref{main1} each $C_{(K_i,X_{t_i})}=C_{(K_i,X_{t_{i+1}})}$, and
$C_{(K_0,X_0)}$ is not rational.
Then $C_{(K_{n-1},X_1)}$ is not rational, contradicting our hypothesis that $X_1$ is a PL knot.
\end{proof}

\section{Boundary-link-like handlebodies}

\begin{lemma} \label{scc} Let $\Sigma$ be a closed orientable $2$-manifold. 

(a) \cite{Mey}, \cite{Sch}, \cite{Ben} {\rm (see also \cite{Ed}*{\S3}, \cite{Pu}*{\S2})} 
A nonzero class $\alpha\in H_1(\Sigma)$ is realizable by an embedded circle if and only if $\alpha$ is primitive.%
\footnote{An element of an abelian group is called {\it primitive} if it is not divisible by an integer greater than $1$.}

(b) \cite{Ed} {\rm (see also \cite{MeP}*{Corollary in \S2})}
A linearly independent collection of classes $\alpha_1,\dots,\alpha_k\in H_1(\Sigma)$ is realizable by
$k$ pairwise disjoint embedded circles if and only if the subgroup $\left<\alpha_1,\dots,\alpha_k\right>$
of $H_1(\Sigma)$ is its direct summand and $\alpha_i\cdot\alpha_j=0$ for $i\ne j$.

(c) \cite{Pu}*{Lemma 6.2} A collection of classes $\alpha_1,\dots,\alpha_k\in H_1(\Sigma)$ is realizable by
$k$ pairwise disjoint embedded circles whose union does not separate $\Sigma$ if and only if they form a basis
of a direct summand of $H_1(\Sigma)$ and $\alpha_i\cdot\alpha_j=0$ for $i\ne j$.
\end{lemma}

Let us note that (c) follows easily from (b), by applying the following lemma.

\begin{lemma} \label{separation}
Let $\Sigma$ be a closed $2$-manifold and $M$ its closed $1$-submanifold.
Then $M$ does not separate $\Sigma$ if and only if the classes of 
the connected components of $M$ are linearly independent in $H_1(\Sigma)$.
\end{lemma}

\begin{proof}
If $\Sigma\but M$ is disconnected, then the closure of any of its components, regarded as a $2$-chain in $M$, 
provides a nontrivial linear dependence between the homology classes of (some of) the components of $M$.
Conversely, if $\Sigma\but M$ is connected, then the restriction map $H^0(\Sigma)\to H^0(\Sigma\but M)$ is onto.
Then by the Poincar\'e duality so is the homomorphism $H_2(\Sigma)\to H_2^\infty(\Sigma\but M)\simeq H_2(\Sigma,M)$.
Then the boundary map $H_2(\Sigma,M)\to H_1(M)$ is zero, and hence the homomorphism $H_1(M)\to H_1(\Sigma)$ 
is injective.
Therefore the homology classes of the components of $M$ are linearly independent in $H_1(\Sigma)$. 
\end{proof}

By a {\it handlebody} in $S^3$ we mean a regular neighborhood of a PL embedded connected graph, 
or in other words a connected PL submanifold of $S^3$ which admits a handle decomposition with handles of
indices $0$ and $1$ only.

Let $V\subset S^3$ be a handlebody and let $W=\overline{S^3\but V}$.
Let $M_V$ and $L_V$ be the kernels of the inclusion induced homomorphisms 
$\phi\:H_1(\partial V)\to H_1(V)$ and $\psi\:H_1(\partial V)\to H_1(W)$.

\begin{lemma} \label{handlebody} (a) $H_i(W)\simeq H_i(V)$ (thus $W$ is a ``homology handlebody'');

(b) $H_1(\partial V)=M_V\oplus L$ and $\phi|_{L_V}\:L_V\to H_1(V)$ and $\psi|_{M_V}\:M_V\to H_1(W)$ are isomorphisms;

(c) The intersection form $H_1(\partial V)\x H_1(\partial V)\to\Z$ vanishes on $L_V\x L_V$ and on $M_V\x M_V$. 
\end{lemma}

\begin{proof}[Proof. (a)]
By the Alexander duality $\tilde H_i(W)\simeq\tilde H^{3-i-1}(V)$.
Since $V$ is a regular neighborhood of a connected graph, $\tilde H^{3-i-1}(V)=0=\tilde H_i(V)$ for $i\ne 1$ 
and $H^1(V)\simeq\Hom\big(H_1(V),\Z)\simeq H_1(V)$.
\end{proof}

\begin{proof}[(b)] Since $H_1(S^3)=0=H_2(S^3)$, by the Mayer--Vietoris exact sequence
the homomorphism $\phi+\psi\:H_1(\partial V)\to H_1(V)\oplus H_1(W)$ is an isomorphism.
Hence the restriction $(\phi+\psi)|_{L_V}$ is an isomorphism onto its image, but clearly
it coincides with $\phi_{L_V}$ and so its image is $H_1(V)$.
Thus $\phi|_{L_V}\:L_V\to H_1(V)$ is an isomorphism, and similarly so is $\psi|_{M_V}\:M_V\to H_1(W)$.
Also we get that the preimage of $H_1(V)$ under $\phi+\psi$ is $L_V$, and similarly
the preimage of $H_1(W)$ under $\phi+\psi$ is $M_V$.
Hence $H_1(\partial V)=M_V\oplus L_V$.
\end{proof}

\begin{proof}[(c)] Let $\alpha,\beta\in L_V$.
These can be represented by elements of $\pi_1(\partial W)$, which in turn can be represented by loops $f,g\:S^1\to\partial W$ 
which are transverse to each other.
Since $f$ and $g$ are null-homologous in $W$, they bound maps $F\:\Sigma\to W$ and $G\:\Sigma'\to W$ of some oriented surfaces,
which may also be assumed to be transverse to each other.
Then the oriented 0-manifold $f\cap g:=\{(x,y)\in S^1\x S^1\mid f(x)=g(y)\}$ is the oriented boundary of
the oriented 1-manifold $F\cap G:=\{(x,y)\in \Sigma\x\Sigma'\mid F(x)=G(y)\}$.
Hence $\alpha\cdot\beta=0$.
Similarly $\gamma\cdot\delta=0$ for any $\gamma,\delta\in M_V$.
\end{proof}

An oriented PL embedded circle in $\partial V$ representing a nonzero element of $L_V$ 
will be called a {\it longitude} of $V$.
A collection of $\rk L_V$ pairwise disjoint longitudes of $V$ whose union is non-separating in $\partial V$
will be called a {\it system of longitudes} for $V$.

\begin{proposition} Every basis of $L_V$ can be represented by a system of longitudes of $V$,
and every system of longitudes of $V$ represents a basis of $L_V$.
\end{proposition}

Of course, here one can speak of ``bases of $H_1(V)$'' instead of ``bases of $L_V$''.

\begin{proof} This follows immediately from Lemmas \ref{handlebody} and \ref{scc}.
\end{proof}

\begin{proposition} Suppose that a basis of $L_V$ is represented by a collection of $g:=\rk L_V$ 
pairwise disjoint embedded closed $1$-manifolds $M_i\subset\partial V$ whose
connected components are non-separating in $\partial V$.
Then the following are equivalent:
\begin{enumerate}
\item each $M_i$ is connected;
\item $M_1\cup\dots\cup M_g$ is non-separating in $\partial V$;
\item $(M_1,\dots,M_g)$ is a system of longitudes of $V$.
\end{enumerate}
\end{proposition}

\begin{proof} Since (3) is the conjunction of (1) and (2), it suffices to show that (1)$\Leftrightarrow$(2).

{\it (1)$\Rightarrow$(2).} This follows from Lemma \ref{separation}.

{\it (2)$\Rightarrow$(1).} By Lemma \ref{separation} the classes of the components $M_{ij}$ of the $M_i$ 
are linearly independent in $H_1(\Sigma)$.
On the other hand, since the $M_{ij}$ are pairwise disjoint, they have zero pairwise intersection numbers.
Hence the $M_{ij}$ form a basis of a subspace $\Lambda$ of $H_1(\partial V)$ such that
the intersection form vanishes on $\Lambda\x\Lambda$.
Then $\rk\Lambda\le\frac12\rk H_1(\partial V)=g$, and therefore there are at most $g$ of the $M_{ij}$.
Hence each $M_i$ is in fact connected.
\end{proof}

\begin{lemma} \label{longitude extension} Every longitude of $V$ extends to a system of longitudes of $V$.
\end{lemma}

\begin{proof} By Lemma \ref{scc}(a) the given longitude $\lambda$ represents 
a primitive homology class.
Hence $L_V/\left<\lambda\right>$ is torsion free, and therefore (being finitely generated) is free abelian.
Then the short exact sequence $0\to\left<\lambda\right>\to L_V\to L_V/\left<\lambda\right>\to 0$ splits,
and so $\left<\lambda\right>$ is a direct summand of $L_V$.
Let $\lambda_1,\dots,\lambda_{g-1}$ be a basis of a complementary summand.
Then $\lambda_1,\dots,\lambda_{g-1},\lambda$ is a basis of $L_V$.
By Lemma \ref{handlebody} $L_V$ is a direct summand of $H_1(\partial V)$ and
the intersection form vanishes on $L_V\x L_V$.
Then by \cite{Ed}*{Proof of Proposition 4.2} $\lambda_1,\dots,\lambda_{g-1}$ can be represented by 
PL embedded circles in $\partial V$, pairwise disjoint from each other and disjoint from $\lambda$.
By Lemma \ref{separation} these form a system of longitudes of $V$.
\end{proof}

\begin{lemma} \label{meridian} Let $m\in M_V$ be a nonzero primitive homology class.
Then $m$ can be represented by the boundary of a properly embedded disk in $V$.
\end{lemma}

\begin{proof} There is a collection of pairwise disjoint properly embedded disks 
$D_1,\dots,D_g$ in $V$ such that $V$ cut open along these disks is the $3$-ball.
On the other hand, by Lemma \ref{scc}(a) $m$ can be realized by an embedded circle $\mu_1$
in $\partial V$, and by the proof of Lemma \ref{longitude extension} this extends to
a collection of embedded circles $\mu_1,\dots,\mu_g$ whose union is non-separating in $\partial V$
and whose homology classes form a basis of $M_V$.
Then it is easy to find embedded circles $\lambda_1,\dots,\lambda_g$ such that
each $\lambda_i$ meets $\mu_i$ transversely in a single point and is disjoint from
$\mu_j$ for $j\ne i$.
Arguing by induction, it is easy to see that the $\lambda_i$ can be chosen pairwise disjoint.
A regular neighborhood $T_i$ of $\lambda_i\cup\mu_i$ is a punctured torus, and we may assume 
that the $T_i$ are pairwise disjoint.
By counting the Euler characteristic $\partial V\but T$ is a sphere with $g$ punctures.
If we repeat the same construction starting from the $\partial D_i$ in place of the $\mu_i$,
we will again end up with punctured tori and a sphere with $g$ punctures.
Using this, it is easy to construct a homeomorphism $\phi$ of $\partial V$ onto itself sending 
each $\partial D_i$ onto $\mu_i$. 

In particular, we get a symplectic automorphism $\phi_*$ of $H_1(\partial V)$ sending each 
$[\partial D_i]$ onto $[\mu_i]$.
Since $\phi_*(M)=M$, by \cite{Bir}*{Lemma 2.2} or \cite{Hir}*{Theorem 2.1} $\phi_*$ is 
realized by a homeomorphism $\psi$ of $V$ onto itself.
The latter sends $\partial D_1$ onto an embedded circle which represents the homology class $m$
and bounds the disk $\psi(D_1)$.
\end{proof}

\begin{lemma} \label{steenrod-realization}
If $M$ is a $3$-manifold (possibly noncompact) and $Q$ is a smooth knot in $\partial M$ which is 
null-homologous in $M$, then $Q$ bounds a smooth embedded surface in $M$.
\end{lemma}

Presumably this is well-known, but we include a proof for completeness.

\begin{proof}
Since $[Q]\in H_1(\partial M)$ maps to $0\in H_1(M)$, it is the image of some $\alpha\in H_2(M,\partial M)$.
Then the Poincar\'e dual%
\footnote{See for instance \cite{Hat}*{Theorem 3.35} for the required version of the Poincar\'e duality.}
$D[Q]\in H^1_c(\partial M)$ is the image of $D(\alpha)\in H^1_c(M)$.
The Pontryagin construction (see for instance \cite{Mi3}) 
yields a smooth map $f\:\partial M\to S^1$ such that $f(U)=1$ 
for some neighborhood of infinity $U$ (=complement of a compact set) and $Q=f^{-1}(-1)$.
Its homotopy class corresponds to $D[Q]$.
Since $D[Q]$ is the image of $\alpha$, $f$ extends to a smooth map $\bar f\:M\to S^1$ such that 
$\bar f(\bar U)=1$ for some neighborhood of infinity $\bar U$.
Then $\bar f^{-1}(-1)$ is the desired embedded surface.
\end{proof}

\begin{lemma}\label{Mayer-Vietoris} 
Let $K=\bigvee_{i=1}^n K_i$ be a wedge of circles embedded in $S^3$ and $Q$ be a knot in $S^3\but K$.
If $Q$ is null-homologous in each $S^3\but K_i$, then $Q$ is null-homologous in $S^3\but K$.
\end{lemma}

\begin{proof}[Algebraic proof] 
By an inductive application of the Mayer--Vietoris exact sequence, using that $H_2(S^3\but pt)=0$.
\end{proof}

\begin{proof}[Geometric proof] $Q$ bounds a Seifert surface $\Sigma$ which is disjoint from the wedge point 
and transverse to each $K_i$.
Since $\lk(Q,K_i)=0$, we have $\Sigma\cdot K_i=0$. 
Then $K_i$ contains an arc $J$ such that $\partial J$ consists of two intersection points of $\Sigma$ and $K_i$ 
of opposite signs and $J$ itself is disjoint from the other intersection points and from the wedge point.
By attaching a tube to $\Sigma$ going along $J$ we remove the two intersection points.
Proceeding in this fashion we obtain a new Seifert surface $\Sigma'$ for $Q$ which is disjoint from $K$.
\end{proof}

\begin{theorem} \label{disjoint seifert}
Let $g=\rk H_1(V)$. If $g=2$, then the following are equivalent:
\begin{enumerate}
\item some system of longitudes of $V$ bounds disjoint Seifert surfaces in $W$;
\item every nonzero primitive class of $L_V$ is represented by a longitude of $V$ which bounds 
a Seifert surface $\Sigma$ in $W$ such that every knot in $\Sigma$ is null-homologous in $W$;
\item some longitude of $V$ bounds a Seifert surface $\Sigma$ in $W$ 
such that every knot in $\Sigma$ is null-homologous in $W$.
\end{enumerate}
For arbitrary $g$ we still have (1)$\Rightarrow$(2)$\Rightarrow$(3).
\end{theorem}

The proof of (3)$\Rightarrow$(1) for $g=2$ is found in \cite{Wr}*{Theorem 3.3}, of which the author was unaware when writing it up.
The idea of proof of (1)$\Rightarrow$(2) comes from \cite{JM}*{proof of Theorem 5}.

\begin{proof} {\it (1)$\Rightarrow$(2).}
Let $l\in L_V$ be the given primitive class.
By Theorem \ref{boundary-retractable} there exists a map $f$ from $W$ to a handlebody $U$ which restricts 
to a homeomorphism $\partial W\to\partial U$.
It is clear that $L_U=M_V$ and $M_U=L_V$ (see Lemma \ref{handlebody}).
Then by Lemma \ref{meridian} $l\in L_V=M_U$ is realized by an embedded circle $\lambda$ in 
$\partial V=\partial W=\partial U$ which bounds a disk $D$ in $U$.
We may assume that $f$ is transverse to $D$.
Then $\Sigma:=f^{-1}(D)$ is a Seifert surface for $\lambda$ in $W$.
Moreover, every knot $Q\subset\Sigma$ is sent by $f$ into $D$, and hence represents a conjugacy class
in the kernel of $f_*\:\pi_1(W)\to\pi_1(U)$.
Then by \cite{JM}*{Corollary to Theorem 1} $Q$ represents a conjugacy class in $\gamma_\omega\pi_1(W)$,
and in particular in $\gamma_2\pi_1(W)$.
Hence it is null-homologous in $W$.

{\it (2)$\Rightarrow$(3).} This is trivial.

{\it (3)$\Rightarrow$(1) for $g=2$.}
The hypothesis implies that each knot $Q$ in the interior of $\Sigma$ is null-homologous in 
$S^3\but V$.
On the other hand, by Lemma \ref{longitude extension} $\lambda$ extends to 
a system of longitudes $\lambda,\lambda'$.
Since $\lambda'\subset V$, we have $\lk(Q,\lambda')=0$.
Hence $\lambda'$ is null-homologous in the complement to every knot in $\Sigma$.
Now $\Sigma$ deformation retracts onto a wedge of circles $X=\bigvee_{k=1}^{2g_1} Q_j$.
By Lemma \ref{Mayer-Vietoris} $\lambda'$ is null-homologous in $S^3\but X$.
By Lemma \ref{steenrod-realization} $\lambda'$ bounds a Seifert surface $F$ in $S^3\but X$.
We may assume that $F$ intersects $\Sigma$ transversely.
It is easy to see that $\Sigma\but X$ is homeomorphic to $S^1\x [0,1)$.
Hence every component $C$ of $F\cap\Sigma$ bounds a surface $F_C\subset\Sigma$.
If $C$ is innermost on $\Sigma$ among all such components, then by surgering $F$ along $F_C$ we 
can get rig of the component $C$.
In the end we obtain a Seifert surface $\Sigma'$ for $\lambda'$ in $S^3\but\Sigma$.
\end{proof}

\begin{remark} If $V$ is of boundary type, it is not true that {\it every} longitude of $V$ bounds 
a Seifert surface $\Sigma$ in $W$ such that every knot in $\Sigma$ is null-homologous in $W$.
For instance, let $V$ be the unknotted handlebody of genus $g\ge 2$.
Thus $\pi_1(\partial V)\simeq\left<m_1,\dots,m_g,l_1,\dots,l_g\mid [m_1,l_1]\cdots[l_g,m_g]=1\right>$
and the map $\pi_1(\partial V)\to\pi_1(V)$ sends each $m_i$ to $1$ and each $l_1$ to the $i$th free generator;
and the map $\pi_1(\partial V)\to\pi_1(W)$ sends each $l_i$ to $1$ and each $m_1$ to the $i$th free generator.
Then any longitude of $V$ representing $[m_i,m_j]l_k$, where $i\ne j$, is not going to be represented by an element of 
$\pi_1(W)''$.
Hence any such longitude does not bound a Seifert surface $\Sigma$ in $W$ such that every knot in $\Sigma$ 
is null-homologous in $W$.
\end{remark}

\begin{theorem} \label{boundary-retractable2}
Let $V$ be a genus $g$ PL handlebody in $S^3$ and let $W=\overline{S^3\but V}$. 
Then the following conditions are equivalent to those of Theorem \ref{boundary-retractable}:
\begin{enumerate}
\item[(1$'$)] $W$ admits a retraction onto a PL embedded wedge $Q\subset\partial W$ of 
$g$ whiskered circles,%
\footnote{A {\it whiskered circle} is the graph ``letter {\sf P}'', endowed with a basepoint, 
which is its bottom point.} 
whose circles bound pairwise disjoint PL embedded disks in $V$;
\item[(2$'$)] $\pi_1(W)$ admits a homomorphism to the free group $F_g$ such that for some 
$Q\subset\partial W$ as in (1$'$) the composition $\pi_1(Q)\to\pi_1(W)\to F_g$ is surjective;
\item[(3$'$)] $\pi_1(W)/\gamma_\omega\pi_1(W)$ is a free group and it is the image of $\pi_1(Q)$
for some $Q\subset\partial W$ as in (1$'$);
\item[(4$'$)] $W$ contains pairwise disjoint properly
PL embedded surfaces $\Sigma_1,\dots,\Sigma_g$ and $V$ contains pairwise disjoint properly 
PL embedded disks $D_1,\dots,D_g$ such that each $\partial\Sigma_i$ is connected and meets 
$\partial D_i$ transversely in a single point, and $\partial\Sigma_i\cap\partial D_j=\emptyset$ 
for $i\ne j$.
\end{enumerate}
\end{theorem}

\begin{proof}
{\it (5)$\Rightarrow$(4$'$).} This is similar to \cite{BeF}*{proof of Theorem 7.3}.
The given map $f\:W\to U$ extends to $f\cup\id_V\:S^3\to U\cup_h V$, where $h$
is the homeomorphism $f|_{\partial W}\:\partial V\to\partial U$.
Clearly $M:=U\cup_h V$ is a homology sphere (see Lemma \ref{handlebody}).
Since $S^3$ is simply-connected, $f$ factors through the universal covering $p\:\tilde M\to M$.
Since $M$ is orientable and $f$ has degree $1$, it induces an isomorphism on integral $3$-homology,
and hence so does $p$.
It follows that $M$ is simply-connected.
Thus by Perelman's theorem $M\cong S^3$.
Then $U\cup V$ is a Heegaard splitting of $S^3$, and hence by Waldhausen's theorem (see \cite{Schu}) 
it must be the standard one of genus $g$.
Then $V$ contains pairwise disjoint properly PL embedded $2$-disks $D_1,\dots,D_g$ 
and $U$ contains pairwise disjoint properly PL embedded $2$-disks $\Delta_1,\dots,\Delta_g$ 
such that each $\partial D_i$ meets $\partial\Delta_i$ transversely 
in a single point and $\partial D_i\cap\partial\Delta_j=\emptyset$ for $i\ne j$.
We may assume that the map $f\:W\to U$ is transverse to each $\Delta_i$, and then it remains to set
$\Sigma_i=f^{-1}(\Delta_i)$.

{\it (4$'$)$\Rightarrow$(4).} A regular neighborhood $T_i$ of $\partial\Sigma_i\cup\partial D_i$ 
is a punctured torus, and we may assume that the $T_i$ are pairwise disjoint.
Then $T:=T_1\cup\dots\cup T_g$ is non-separating in $\partial V$ (in fact, by counting the 
Euler characteristic $\partial V\but T$ is a sphere with $g$ punctures).
In particular, $\Sigma_1\cup\dots\cup\Sigma_g$ is non-separating in $\partial V$.

{\it (1$'$)$\Leftrightarrow$(2$'$)$\Leftrightarrow$(3$'$)$\Leftrightarrow$(4$'$).}
Similarly to \cite{Sm} (see also \cite{JM}, \cite{J1}).
\end{proof}

\section{Proof of Theorem \ref{main6}}

\begin{theorem} \label{alexander-module}
Let $\K$ be a knot in $S^3$ which is the intersection of a sequence of handlebodies $V_0\supset V_1\supset\dots$ 
and let $K_i\subset\partial V_i$ be knots such that each $K_i$ is homologous to $\K$ in $V_i$.
Suppose that each $K_i$ bounds a Seifert surface $\Sigma_i$ in $W_i:=\overline{S^3\but V_i}$
such that each knot in $\Sigma_i$ is null-homologous in $W_i$.
Then the knot module of $\K$ is a torsion module.
\end{theorem}

\begin{proof} The knot module of $\K$ is the $\Z[t^{\pm1}]$-module $H_1(X)$, 
where $X$ is the infinite cyclic cover of $S^3\but K$.
Every element of $H_1(X)$ can be represented by a knot $\tilde Q_0\subset X$.
We may assume that $\tilde Q_0$ projects to a knot $Q_0$ (i.e.\ $Q_0$ has no self-intersections) in $S^3\but K$.
Since $Q_0$ is null-homologous in $S^3\but K$, it bounds a Seifert surface $F_0$ in $S^3\but K$.
Then $F_0$ is disjoint from some $V_n$.
We may assume that $F_0$ intersects $\Sigma_n$ transversely.
Their intersection $Q_1$ is null-homologous in $W_n$, and hence bounds a Seifert surface $F_1$
in $W_n$.
We may assume that $F_1$ intersects $\Sigma_n$ transversely.
Their intersection $Q_2$ is null-homologous in $W_n$
Continuing in this fashion, we obtain links $Q_1,Q_2,\dots\subset\Sigma_n$ which are null-homologous in $W_n$
and in particular in $S^3\but K$.
Let $\tilde Q_1,\tilde Q_2,\dots$ be their lifts to $X$.
They lie in the preimage $\tilde W_n$ of $W_n$ in $X$.
This $\tilde W_n$ can also be described as the cover of $W_n$ corresponding to (the kernel of) the homomorphism
$\pi_1(W_n)\to\Z$ given by $[Z]\mapsto\lk(Z,\K)$.
(In other words, this homomorphism is determined by the class in $H^1(W_n)$ which is Alexander dual 
to $[\K]\in H_1(V_n)$.)
Hence a fundamental domain of the action of $\Z$ on $\tilde W_n$ is a $3$-manifold with boundary
whose interior as a subset of $\tilde W_n$ is homeomorphic to $W_n\but\Sigma_n$; two consecutive copies of 
the fundamental domain meet along a copy of $\Sigma_n$.
Then each $\tilde Q_i$ (starting from $i=0$) is homologous in $\tilde W_n$ to $(1-t)\tilde Q_{i+1}$.
Namely, $\tilde Q_i-(1-t)\tilde Q_{i+1}=\partial\tilde F_i$, where $\tilde F_i$ is a surface with
boundary whose interior as a subset of $\tilde W_n$ is homeomorphic to $F_i\but\Sigma_n$.
Thus each $[\tilde Q_i]=(1-t)[\tilde Q_{i+1}]\in H_1(\tilde W_n)$.
Since $W_n$ is a compact $3$-manifold, $H_1(\tilde W_n)$ is finitely generated over $\Z[t^{\pm1}]$.
Since $\Z[t^{\pm1}]$ is Noetherian, so is $H_1(\tilde W_n)$.
Hence its submodules $\left<[Q_0]\right>\subset\left<[Q_1]\right>\subset\dots$ stabilize at some point.
Thus there exist a $k$ and a Laurent polynomial $p$ such that $[Q_k]=p[Q_{k-1}]$. 
Since $[Q_{k-1}]=(1-t)[Q_k]$, we get $[Q_k]=p(1-t)[Q_k]$, or $\big(p(t-1)-1\big)[Q_k]=0$.
Hence also $\big(p(t-1)-1\big)[Q_0]=\big(p(t-1)-1\big)(1-t)^k[Q_k]=0$.
\end{proof}

\begin{corollary} \label{main6''}
If a knot $\K\subset S^3$ is the intersection of a nested sequence of boundary-link-like handlebodies, 
then the knot module of $\K$ is a torsion module.
\end{corollary}

Theorem \ref{main6} follows from Corollaries \ref{main6'} and \ref{main6''}.

\begin{proof} Let $V_1\supset V_2\supset\dots$ be the given handlebodies.
Since $H_2(V_i)=0$, we have $\lim^1 H_2(V_i)=0$ and consequently by Milnor's short exact sequence
$\lim H_1(V_i)\simeq H_1(\K)$.
Moreover, this isomorphism sends the thread consisting of the classes $[\K]\in H_1(V_i)$
to the fundamental class in $H_1(\K)$, which is nonzero.
Hence $[\K]\in H_1(V_i)$ is nonzero for all sufficiently large $i$.
By omitting a finite number of initial handlebodies we may assume that all of these classes are nonzero.
Moreover, all of them are primitive, since so is the fundamental class of $H_1(\K)$.

Let $W_i=\overline{S^3\but V_i}$.
Since $L_V=\ker\big(\psi\:H_1(\partial V_i)\to H_1(W_i)\big)$ is mapped isomorphically onto $H_1(V)$ 
by the restriction of $\phi\:H_1(\partial V_i)\to H_1(V_i)$ (see Lemma \ref{handlebody}(b)),
each $[\K]\in V_i$ is the image of some nonzero primitive class $l_i\in L_{V_i}$.
By Lemma \ref{disjoint seifert} this class $l_i$ is represented by a longitude $K_i$ of $V$ which bounds 
a Seifert surface $\Sigma$ in $W$ such that every knot in $\Sigma$ is null-homologous in $W$.
Then by Theorem \ref{alexander-module} the knot module of $\K$ is a torsion module.
\end{proof}

\appendix

\section{Every knot is $I$-equivalent to the unknot}\label{app}

This Appendix contains a proof of the following result, due largely to C. Giffen.

\begin{theorem} \label{i-eq} Every topological knot is I-equivalent to the unknot.
\end{theorem}

Theorem \ref{i-eq} follows from Proposition \ref{insertion} and Theorem \ref{giffen}.

An alternative proof of Theorem \ref{i-eq} can be found in Ancel's paper \cite{An0}.

In fact, the reason why I wrote down a proof of Theorem \ref{i-eq} was that I had previously stated this result as known,
with attribution to Giffen \cite{M21}. 
In August 2021 I was contacted by M. Freedman who mentioned among other things that he and R. Ancel are wondering about 
the details of the proof.
In a few days I've sent him (Freedman) an email with a writeup of the proof of Theorem \ref{giffen} (essentially the same as below),
suggesting that he forward it to Ancel.
I've also sent the same proof to Ancel himself in response to his email in early September, but for some reason he did not 
get my reply.
Eventually he got my email in December 2021, but by that time he already wrote down his own account \cite{An0}.

\subsection{F-isotopy implies I-equivalence}

\begin{theorem}[Giffen \cite{Gi1}, \cite{Gi2}] \label{giffen} Every topological knot in $S^1\x D^2$ that is homotopic to the core circle
is $I$-equivalent to the core circle.
\end{theorem}

Ancel's alternative exposition of what appears to be more or less the same construction \cite{An0} is more detailed and more thoughtful, 
but also much longer.
 
\begin{proof} Let $k\:S^1\to S^1\x D^2$ be the given knot.
The identity $\id\:S^1\x D^2\to S^1\x D^2\x\{0\}$ and the core circle embedding $S^1\x\{\infty\}\to S^1\x D^2\x\{1\}$ 
cobound an embedding $E\:S^1\x CD^2\to S^1\x D^2\x [0,1]$, where $CD^2$ is the cone over $D^2$ with apex $\infty$.
We may identify $CD^2$ with the one-point compactification of $D^2\x[0,\infty)$, with $\infty$ being the added point.
Let $\tilde w$ denote the composition $S^1\x\R\xr{k\x\id_{\R}} S^1\x D^2\x\R\xr{\tilde\phi} S^1\x D^2\x\R$, 
where $\tilde\phi$ is the level-preserving homeomorphism given by $\tilde\phi(\alpha,x,t)=\big(\alpha-[t],x,t\big)$;
hereafter we identify $S^1$ with $\R/\Z$.
The closure $\bar A$ of $A:=\tilde w\big(S^1\x[0,\infty)\big)$ in $S^1\x CD^2$ is the union of $A$ and
$S^1\x\{\infty\}$.
We will show that $\bar A$ is homeomorphic to $S^1\x I$; this easily implies that $E(\bar A)$ is the image of 
an $I$-equivalence between $k$ and the core circle.

Let $K=k(S^1)$ and let $\tilde W=\tilde w(S^1\x\R)$.
Let $\phi$ be the self-homeomorphism of $S^1\x D^2\x S^1$ given by $\phi(\alpha,x,\beta)=(\alpha-\beta,x,\beta)$,
and let $W=\phi(K\x S^1)$.
Then $\tilde W$ is the preimage of $W$ under the covering $\id_{S^1\x D^2}\x q\:S^1\x D^2\x\R\to S^1\x D^2\x S^1$,
where $q$ is the quotient map $\R\to S^1$.
On the other hand, let $\psi$ be the self-homeomorphism of $S^1\x D^2\x S^1$ given by
$\psi(\alpha,x,\beta)=(\alpha,x,\alpha-\beta)$.
Then $\psi$ takes $K\x S^1$ onto itself.
Thus $W=\chi(K\x S^1)$, where $\chi=\phi\circ\psi$. 
Let us note that $\chi(\alpha,x,\beta)=(\beta,x,\alpha-\beta)$.

Let $\tilde K$ be the preimage of $K$ under $q\x\id_{D^2}\:\R\x D^2\to S^1\x D^2$.
The covering translation $\sigma\:\R\x D^2\to\R\x D^2$, given by $\sigma(t,x)=(t+1,x)$, restricts to
a homeomorphism $\sigma_K\:\tilde K\to\tilde K$.
Since $\sigma$ is isotopic to $\id_{\R\x D^2}$, its mapping torus $MT(\sigma)$ is homeomorphic to
$\R\x D^2\x S^1$.
Writing $MT(\sigma)$ as $\R\x D^2\x[0,1]/_{(t,\,x,\,0)\sim(t+1,\,x,\,1)}$, a specific homeomorphism 
$\tilde\chi\:MT(\sigma)\to S^1\x D^2\x\R$ is given by $\tilde\chi\big([(t,x,s)]\big)=([s],x,t-s)$.
Since $\sigma$ is a lift of $\id_{S^1\x D^2}$, there is a covering
$p\:MT(\sigma)\to MT(\id_{S^1\x D^2})$ and a commutative diagram
\[\begin{CD}
MT(\sigma)@>\tilde\chi>>S^1\x D^2\x\R\\
@VpVV@V\id_{S^1\x D^2}\x qVV\\
S^1\x D^2\x S^1@>\chi>>S^1\x D^2\x S^1.
\end{CD}\]
It is easy to see that $MT(\sigma_K)=p^{-1}(K\x S^1)$.
Since $K\x S^1=\chi^{-1}(W)$, we obtain that $MT(\sigma_K)=\tilde\chi^{-1}(\tilde W)$.

We may regard the suspension $\Sigma D^2$ as the compactification of $\R\x D^2$ obtained by
adding the two suspension points $\pm\infty$.
Then $\sigma$ uniquely extends to a homeomorphism $\bar\sigma\:\Sigma D^2\to\Sigma D^2$.
In turn, $\bar\sigma$ restricts to a homeomorphism $\bar\sigma_K\:\bar K\to\bar K$, where 
$\bar K=\tilde K\cup\{\pm\infty\}$.
Since $\tilde\chi$ commutes with the ``projection'' $MT(\sigma)\to S^1$, $[(t,x,s)]\mapsto [s]$, and
the projection $S^1\x D^2\x\R\to S^1$, it extends to a homeomorphism
$\bar\chi\:MT(\bar\sigma)\to S^1\x\Sigma D^2$.
(To be precise, $\Sigma D^2$ is regarded here as a compactification of $D^2\x\R$ rather than $\R\x D^2$.)
Clearly, $MT(\bar\sigma_K)=\bar\chi^{-1}(\bar W)$, where $\bar W=\tilde W\cup S^1\x\{\pm\infty\}$.
On the other hand, since $\bar\sigma_K$ is isotopic to $\id_{\bar K}$, its mapping torus
is homeomorphic to $\bar K\x S^1$.
Thus $\bar W$ is homeomorphic to $S^1\x [-1,1]$.
Let us note that this homeomorphism sends $\tilde W$ onto $S^1\x (-1,1)$.
Now $\bar A=\bar W\cap S^1\x CD^2$, where $\bar W\cap S^1\x D^2$ is a copy of $S^1$ contained in $\tilde W$,
so it is clear that $\bar A$ is homeomorphic to $S^1\x [0,1]$.
\end{proof}

\subsection{Every knot is F-isotopic to the unknot}

\begin{proposition} \label{insertion} 
Every oriented topological knot $K$ in $S^3$ lies in the interior of an unknotted solid torus $T$ 
such that the inclusion $K\to T$ is a homotopy equivalence. 
\end{proposition}

Ancel's proof of Proposition \ref{insertion} in \cite{An0} is decidedly simpler than what follows.
But perhaps the approach below might still be of some interest, as it answers some questions that naturally arise in connection with
this statement.

\begin{lemma} \label{conway'} If $K$ is a topological knot in $S^3$, then every PL knot $Q$ in $S^3\but K$ is
homotopic in $S^3\but K$ to an unknot.
\end{lemma}

It turns out that this lemma is proved in the literature: McMillan proved that if $X$ is a compact 
$1$-dimensional ANR in $S^3$, then $X$ is the intersection of a nested sequence of handlebodies \cite{McM}*{Lemmas 2 and 3}, 
and Bothe noted that this property implies that every knot $Q$ in $S^3\but X$ is homotopic in $S^3\but X$ to a knot which 
is unknotted as a knot in $S^3$ \cite{Bo2}*{assertion ``(A) implies (C)'' on p.\ 176}.

\begin{proof}
Let us pick some point in $S^3\but(K\cup Q)$ and identify its complement in $S^3$ with $\R^3$ by a 
PL homeomorphism.
By general position there exists a linear surjection $\pi\:\R^3\to\R^2$ whose restriction to $Q$
has only finitely many double points and no triple points.
If $J\subset Q$ is a small arc whose image does not contain any double points, and $f\:Q\to\R$ 
is a PL map which is monotone on $Q\but J$, then $\pi\x f\:Q\to\R^2\x\R$ is easily seen to be
a PL unknot.
Hence we can unknot $Q$ (in $\R^3$) by changing some of the crossings in its plane diagram $\pi(Q)$. 
Unfortunately, we may not assume that these crossings avoid $\pi(K)$.%
\footnote{K. Borsuk, answering a question of R. H. Fox, constructed a knot in $\R^m$ for each $m\ge 3$ 
whose projection to every hyperplane has a nonempty interior \cite{Bor}.}
However, by a well-known theorem asserting that $K$ is tame in the sense of Shtan'ko (or has embedding
dimension $1$ in another terminology) \cite{McM}*{Lemma 3}, \cite{Bo}*{Satz 4} (see also \cite{Sh2}*{Theorem 4}), there exists 
an ambient isotopy $h_t\:\R^3\to\R^3$, $h_0=\id$, keeping $Q$ fixed and such that for every double point 
$\pi(p)=\pi(q)$, 
$p,q\in Q$, the straight line segment $[p,q]$ between $p$ and $q$ is disjoint from $h_1(K)$.%
\footnote{The construction of $h_t$ is quite easy. 
Let $J_{pp'}$ (resp.\ $J_{qq'}$) be a small arc of $Q$ with $p$ (resp.\ $q$) as one of its endpoints.
The join $J_{pp'}*J_{qq'}$ is a ball, and its boundary sphere $\Sigma$ is not separated by 
$\Sigma\cap K$ since $\tilde H_0(\Sigma\but K)\simeq H^1(\Sigma\cap K)=0$ due to $K\not\subset\Sigma$.
Hence $p$ and $q$ are joined by a PL arc $J_{pq}$ in $\Sigma\but K$.
We may assume that $J_{pq}$ meets $J_{pp'}$ only in $p$ and $J_{qq'}$ only in $q$.
Then there is a PL ambient isotopy of $\Sigma$ keeping $J_{pp'}\cup J_{qq'}$ fixed from $\id_\Sigma$
to a homeomorphism sending $J_{pq}$ onto $[p,q]$.
It extends to a PL ambient isotopy $h_t^{pq}$ of $\R^3$ with support in a small neighborhood of 
$J_{pp'}*J_{qq'}$ and keeping $Q$ fixed.
The composition $h_t^{p_1q_1}h_t^{p_2q_2}\dots$ of all such isotopies is the desired isotopy $h_t$.}
Therefore the desired crossing changes can be realized by a PL homotopy $K_t$ within $\R^3\but h_1(K)$.
The resulting PL knot $K_1$ is unknotted in $\R^3$ and homotopic to $Q$ in $\R^3\but h_1(K)$.
Hence $h_1^{-1}(K_1)$, which is still an unknot, is homotopic to $h_1^{-1}(Q)=Q$ in $\R^3\but K$.
\end{proof}

Given disjoint oriented topological knots $K$ and $Q$ in $S^3$, let $\lk(K,Q)$ denote
the integer $[Q]\in H_1(S^3\but K)\overset{A}\simeq H^1(K)\overset{P}\simeq H_0(K)=\Z$, 
where $A$ is the Alexander duality (involving \v Cech cohomology and assuming a fixed orientation 
of $S^3$) and $P$ is the Poincare duality $c\mapsto [K]\smallfrown c$.

\begin{lemma} \label{symmetry} $\lk(K,Q)=\lk(Q,K)$.
\end{lemma}

\begin{proof}
It is easy and well-known that $\lk(K,Q)=\lk(Q,K)$ for PL links $(K,Q)$.
Also it is easy to show that any two sufficiently close PL approximations of a given topological link
$(K,Q)$ are link homotopic (that is, homotopic by a homotopy that keeps the two components disjoint), 
and therefore have the same linking number.
This yields an extension $\lk_{PL}$ of $\lk$ from PL to topological links, which is symmetric.
It remains to show that $\lk=\lk_{PL}$.

Let $T$ be some triangulation of $S^3$ and let $N_i$ be the union of all simplexes of the $i$th 
derived subdivision $T^{(i)}$ that meet $K$.
Then each $N_i$ is a closed neighborhood of $K$ which is a $3$-manifold with boundary,
$N_1\supset N_2\supset\dots$ and $\bigcap_i N_i=K$.
The Alexander duality isomorphism $A\:H_1(S^3\but K)\xr{\simeq} H^1(K)$ is the direct limit of the
Alexander duality isomorphisms $A_i\:H_1(S^3\but N_i)\xr{\simeq} H^1(N_i)$.
There exists an $r$ such that $Q\cap N_r=\emptyset$.
Let $i\:K\to N_r$ and $i'\:Q\to S^3\but N_r$ be the inclusion maps.
As long as we fix homeomorphisms of $K$ and $Q$ with $S^1$, it makes sense to say that $i$ and $i'$
are homotopic to PL embeddings $e\:K\to N_r$ and $e'\:Q\to S^3\but N_r$.

Now $\lk(K,Q)$ is the image of $i'_*[Q]\in H_1(S^3\but N_r)$ in $H_1(S^3\but K)\overset{A}
\simeq H^1(K)\overset{P}\simeq H_0(K)=\Z$.
This is the same as the image of $A_r(i'_*[Q])\in H^1(N_r)$ in $H^1(K)\overset{P}\simeq H_0(K)=\Z$.
Thus $\lk(K,Q)=i^*\big(A_r(i'_*[Q])\big)\CAP [K]\in H_0(K)=\Z$.
Writing $K'=e(K)$ and $Q'=e'(Q)$, we similarly have 
$\lk(K',Q')=e^*\big(A_r(e'_*[Q])\big)\CAP [K]\in H_0(K)=\Z$.
But $i^*=e^*$ and $i'_*=e'_*$.
Thus $\lk(K,Q)=\lk(K',Q')$, and it follows that $\lk(K,Q)=\lk_{PL}(K,Q)$.
\end{proof}

\begin{remark}
A direct algebraic (but rather long) proof of Lemma \ref{symmetry} can be found in \cite{M00}.
It shows that each of $\lk(K,Q)$ and $\lk(Q,K)$ is equal to $\lk_S(K,Q):=\big<S[Q],\,i_*[K]\big>$, 
where $i\:K\to S^3\but Q$ is the inclusion and $S\:H_1(Q)\to H^1(S^3\but Q)$ is the Steenrod duality.
In more detail, the equality $\lk(K,Q)=\lk_S(K,Q)$ is an easy consequence of the Sitnikov duality (for arbitrary
subsets of $S^3$, not necessarily open or closed).
To prove that $\lk_S(K,Q)=\lk(Q,K)$, one defines an evaluation pairing 
$H^k(X)\otimes H_k(X)\to\Z$ between \v Cech cohomology and Steenrod homology of the compactum $X$, 
and shows that when $X\subset S^n$, it is determined by the usual evaluation pairing
$H^k(S^n\but X)\otimes H_k(S^n\but X)\to\Z$.
\end{remark}

\begin{lemma} \label{insertion1}
If $K$ is an oriented topological knot in $S^3$, then there exists an oriented PL knot $Q$ 
in $S^3\but K$ such that $\lk(Q,K)=1$ and $Q$ is unknotted (in $S^3$).
\end{lemma}

Proposition \ref{insertion} is an easy consequence of Lemma \ref{insertion1}.

\begin{proof} The positive generator of 
$H_1(S^3\but K)\overset{A}\simeq H^1(K)\overset{P}\simeq H_0(K)=\Z$ can certainly be represented
by a PL knot $Q'$ in $S^3\but K$.
By Lemma \ref{conway'} $Q'$ is homotopic to a PL unknot $Q$ in $S^3\but K$.
Hence by Lemma \ref{symmetry} $\lk(Q,K)=\lk(K,Q)=\lk(K,Q')=1$.
\end{proof}

\section*{Acknowledgements}

I would like to thank P. M. Akhmetiev, F. D. Ancel, M. H. Freedman, L. Funar, M. Il'insky and A. Zastrow 
for useful discussions and thoughtful remarks. 
Also I'm grateful to R. D. Edwards, M. Farber, J. Keesling and S. P. Novikov for stimulating my interest 
in isotopy of wild knots (around 2005).



\end{document}